%% file: main.tex
\documentclass[11pt]{article}

\usepackage{amsmath,amssymb,amsthm,mathtools}
\usepackage{bm,booktabs,multirow,enumitem,graphicx,placeins}
\usepackage[authoryear,round]{natbib}
\usepackage[colorlinks,citecolor=blue,urlcolor=blue,linkcolor=blue]{hyperref}

\ifdefined\usingimsart
  \startlocaldefs
\else
  \usepackage[margin=2.7cm]{geometry}
  \newenvironment{frontmatter}{}{}
  \newcommand{\kwdsep}{}
  \newenvironment{keyword}[1][]{\par\noindent\gdef\kwdsep{}%
    \if\relax\detokenize{#1}\relax\textbf{Keywords:}\else\textbf{MSC2020 subject classifications:}\fi\ \kern0pt}{\unskip.\par}
  \newcommand{\kwd}[2][]{\unskip\if\relax\detokenize{#1}\relax\kwdsep\else#1\fi#2\gdef\kwdsep{, }}
  \newcommand{\runtitle}[1]{}\newcommand{\runauthor}[1]{}
\fi

\numberwithin{equation}{section}
\theoremstyle{plain}
\newtheorem{theorem}{Theorem}[section]
\newtheorem{proposition}[theorem]{Proposition}
\newtheorem{lemma}[theorem]{Lemma}
\newtheorem{corollary}[theorem]{Corollary}
\theoremstyle{remark}
\newtheorem{remark}[theorem]{Remark}
\theoremstyle{definition}
\newtheorem{definition}[theorem]{Definition}
\newtheorem{assumption}{Assumption}

\newcommand{\R}{\mathbb{R}}
\newcommand{\Rb}{\overline{\mathbb{R}}}
\newcommand{\bu}{\mathbf{u}}
\newcommand{\bv}{\mathbf{v}}
\newcommand{\bx}{\mathbf{x}}
\newcommand{\by}{\mathbf{y}}
\newcommand{\G}{\mathbb{G}}
\newcommand{\PP}{\mathbb{P}}
\newcommand{\EE}{\mathbb{E}}
\newcommand{\Ln}{\widehat L_n}
\newcommand{\ph}{\widehat p_n}
\newcommand{\that}{\widehat t_n}
\newcommand{\nG}[1]{\lVert #1\rVert_G}
\newcommand{\nI}[1]{\lVert #1\rVert_\infty}
\newcommand{\ip}[2]{\langle #1,#2\rangle_G}
\newcommand{\1}{\mathbb{I}}
\newcommand{\Dd}{D'_\ast}
\newcommand{\dto}{\rightsquigarrow}
\DeclareMathOperator*{\argmin}{arg\,min}

\ifdefined\usingimsart
  \endlocaldefs
\fi

\begin{document}

\ifdefined\usingimsart
\begin{frontmatter}
\title{Testing independence in contamination models}
\runtitle{Testing independence in contamination models}
\begin{aug}
\author[A]{\fnms{Olivier}~\snm{Lopez}}
\and
\author[B]{\fnms{Xavier}~\snm{Milhaud}\ead[label=e1]{xavier.milhaud@univ-lyon1.fr}}
\and
\author[C]{\fnms{Pierre}~\snm{Vandekerkhove}\ead[label=e2]{pierre.vandekerkhove@univ-eiffel.fr}}
\address[A]{CREST, ENSAE, Institut Polytechnique de Paris, 91120 Palaiseau, France}
\address[B]{Univ Lyon, UCBL, ISFA LSAF EA2429, F-69007 Lyon, France\printead[presep={\\ }]{e1}}
\address[C]{Universit\'e Gustave Eiffel, LAMA (UMR 8050), 77420 Champs-sur-Marne, France\printead[presep={\\ }]{e2}}
\runauthor{O. Lopez, X. Milhaud and P. Vandekerkhove}
\end{aug}
\else
\title{Testing independence in contamination models}
\author{Olivier Lopez$^{(1)}$ \and Xavier Milhaud$^{(2)}$ \and Pierre Vandekerkhove$^{(3)}$}
\date{\small $^{(1)}$CREST, ENSAE, Institut Polytechnique de Paris, 91120 Palaiseau, France\\
$^{(2)}$Univ Lyon, UCBL, ISFA LSAF EA2429, F-69007 Lyon, France\\
$^{(3)}$Universit\'e Gustave Eiffel, LAMA (UMR 8050), 77420 Champs-sur-Marne, France\\[1ex]
\today}
\maketitle
\fi

\begin{abstract}
We consider a bivariate contamination model $L=(1-p)F_0+pF$ in which the
distribution $F_0$ of a gold standard is known while the mixing weight $p$ and
the distribution $F$ of the contaminant component are unknown, and we test
whether the two coordinates of the unknown component are independent, i.e.\
$F=F_1\otimes F_2$. Classical tests of independence applied to the observed
sample are meaningless in this setting, since a mixture of product
distributions is not a product distribution. Following the Inversion--Best
Matching (IBM) principle, we invert the contamination model for every
candidate weight, which yields two parametric families of functions: the
inverted cdfs of the unknown component and the products of their marginals.
The weight is estimated by the value at which these two families are closest
in a Cram\'er--von Mises sense, and the empirical minimal distance scaled by
$n$ (the sample size) is the test statistic. We show that the discrepancy
between the two families has an exact quadratic structure in $1/p$. The
contrast is then a quartic polynomial in $1/p$ with explicit empirical
coefficients, so the weight estimator is obtained exactly, by comparing a
finite number of candidates: the roots of a cubic equation and the endpoints
of the parameter set. This structure also yields a sharp identifiability
condition under the null hypothesis, and a transparent asymptotic theory:
under the null, the weight estimator is $\sqrt n$-consistent and
asymptotically normal and the test statistic converges to a double
Wiener--It\^o integral with respect to a Brownian bridge, while it diverges
under any fixed alternative. The null distribution is calibrated by a
multiplier bootstrap whose validity is established. A simulation study
illustrates the finite-sample behaviour of the procedure and the failure of
naive independence tests. Two applications to gene expression data, in which
the known source is the exact law of pairs of randomised rank statistics for
null genes, show that the effects of differentially expressed genes are
dependent across replicates and between location and dispersion, while an
application to spam filtering, with a known source given by an archive of
legitimate emails, illustrates an unknown component with independent
coordinates.
\end{abstract}



\ifdefined\usingimsart
\end{frontmatter}
\fi

\section{Introduction}\label{sec:intro}

Contamination (or admixture) models describe a population in which a known
``background'' source is mixed, in unknown proportion, with an unknown source
of interest. In its simplest form, one observes an i.i.d.\ sample from a
cumulative distribution function (cdf)
\begin{equation}\label{eq:model}
  L(\bu) \;=\; (1-p^\ast)\,F_0(\bu) + p^\ast F(\bu),
  \qquad \bu=(u_1,u_2)^\top\in\R^2,
\end{equation}
where the cdf $F_0$ is known while the weight $p^\ast\in(0,1)$ and the cdf $F$
are unknown. Such models arise in multiple testing, where $F_0$ is the null
distribution of $p$-values or test statistics \citep{Efron2001,PatraSen2016},
in astronomy and genetics, and in actuarial science, where a portfolio is
partly driven by a reference mortality table and partly by an unknown excess
component \citep{MPSV2024,MPSV2024b}. Estimation of $(p^\ast,F)$ in the
univariate case has been studied under various shape constraints
\citep{BDV2006,BMV2006,PatraSen2016}; testing problems involving the unknown
component were addressed in \citet{PommeretVandekerkhove2019},
\citet{MPSV2022} and \citet{MPSV2024}, the latter introducing the
Inversion--Best Matching (IBM) methodology on which the present work builds.

In the bivariate setting~\eqref{eq:model}, a natural and practically important
question is whether the two coordinates of the \emph{unobserved} component are
independent, that is, whether
\begin{equation}\label{eq:hyp}
  H_0:\ F = F_1\otimes F_2
  \qquad\text{against}\qquad
  H_1:\ F \neq F_1\otimes F_2,
\end{equation}
where $F_1,F_2$ denote the marginal cdfs of $F$ and
$(F_1\otimes F_2)(\bu)=F_1(u_1)F_2(u_2)$. In the actuarial example, $H_0$ means
that, once the known reference mortality has been filtered out, the residual
lifetimes of two related individuals do not depend on each other; in multiple
testing, it means that two statistics computed on the non-null units are
independent.

Classical tests of independence, based on Hoeffding's or the
Blum--Kiefer--Rosenblatt (BKR) process $\Ln-\widehat L_{n,1}\otimes \widehat
L_{n,2}$ \citep{Hoeffding1948,BKR1961,Deheuvels1981}, or on the empirical
copula \citep{GenestRemillard2004,KojadinovicHolmes2009}, test independence of
the \emph{observed} vector and are therefore inadequate here: even when both
$F_0$ and $F$ are product cdfs, $L$ is a mixture of two product cdfs and is in
general not a product. Such tests reject $H_0$ with probability tending to one
under the null hypothesis~\eqref{eq:hyp}; see Section~\ref{sec:simu}.

The IBM principle proceeds in two steps. In the \emph{inversion} step, the
contamination model is inverted for every candidate weight $p$. This produces two
parametric families indexed by $p$: the family $\mathcal F_1$ of candidate
cdfs $F(\cdot\,;p)$ for the unknown component, and the family $\mathcal F_2$
of the products $F_1(\cdot\,;p)\otimes F_2(\cdot\,;p)$ of their marginals.
Under $H_0$ the two families meet at the true weight. In the
\emph{best-matching} step, we look for the weight at which the two families
are closest in an integrated squared (Cram\'er--von Mises) sense. The
minimiser estimates the weight, and the minimal distance is the test
statistic. Our main contributions are the following.
\begin{enumerate}[label=(\roman*),leftmargin=*]
\item We show (Lemma~\ref{lem:algebra}) that the discrepancy between the
  two inverted families is an exact quadratic polynomial in $t=1/p$, whose coefficients are explicit
  functionals of the dependence of $F_0$, of the dependence of $F$ and of the
  marginal gap $(F_1-F_{0,1})\otimes(F_2-F_{0,2})$. This yields identification conditions for $p^\ast$ under $H_0$
  and for the separation of the families under $H_1$
  (Proposition~\ref{prop:ident}). It also makes the estimator of the weight
  exactly computable, without numerical optimisation. The empirical contrast is a
  quartic polynomial in $t$ with explicit coefficients, so its minimiser is either
  an endpoint of the parameter set or one of the at most three real roots of a
  cubic equation, and these roots are given by Cardano's formula.
\item We establish almost sure rates of convergence under both hypotheses
  (Theorem~\ref{thm:consistency}) and a joint functional central limit theorem
  for the estimated weight and the estimated discrepancy around their
  pseudo-true values, valid under both hypotheses (Theorem~\ref{thm:unified}).
  The test statistic is an exact quadratic functional of this pair
  (Corollary~\ref{cor:repr}), which yields the joint asymptotic distribution of the
  weight estimator and of the test statistic under $H_0$
  (Corollary~\ref{thm:null}), and, under a fixed alternative, central limit
  theorems for the weight estimator and for the contrast around their pseudo-true
  values (Corollary~\ref{thm:alt}). These yield consistency of the test and an
  asymptotic approximation of its power.
\item We propose a multiplier bootstrap calibration and prove its validity
  (Theorem~\ref{thm:bootstrap}); the procedure is tuning-free apart from the
  choice of the integrating measure and of a range for the weight.
\end{enumerate}

The paper is organised as follows. Section~\ref{sec:model} introduces the inverted
parametric families and the IBM contrast, and studies identifiability. Section~\ref{sec:asympt} contains the
estimator, the test and the asymptotic theory. Section~\ref{sec:simu} reports
a simulation study, Section~\ref{sec:real} applications to gene
expression data and to spam filtering, and Section~\ref{sec:discussion} discusses extensions. All
proofs are deferred to the Appendix.

\paragraph{Notation.} Vectors are written in bold. $\Rb=[-\infty,+\infty]$ and
$\bu\le\bv$ is understood componentwise. For a bounded function $w$ on
$\Rb^2$ we write $w^{(1)}(x)=w(x,+\infty)$ and $w^{(2)}(y)=w(+\infty,y)$ for its
``marginals'', and $(f\otimes g)(\bu)=f(u_1)g(u_2)$ for functions $f,g$ on
$\Rb$. $\ell^\infty(\Rb^2)$ is the space of bounded functions with the sup-norm
$\nI{\cdot}$. $G$ is a probability measure on $\R^2$ (the integrating
measure) and $\ip{f}{g}=\int fg\,dG$, $\nG{f}=\ip ff^{1/2}\le\nI f$. Weak
convergence in the sense of Hoffmann-J{\o}rgensen is denoted by $\dto$
\citep{vdVW1996}. Two parametrisations of the weight are used: $p$ and its
inverse $t=1/p$. Dots denote derivatives with respect to $p$ (e.g.\ $\dot D$,
$\ddot d$), and primes denote derivatives with respect to $t$ (e.g.\ $D'$,
$d''$). They are linked by the chain rule $\partial_tf(1/t)=-p^2\dot f(p)$;
see~\eqref{eq:chain}.

\section{Inverted parametric families and identifiability}\label{sec:model}

\subsection{Inversion and best matching}

Let $F_{0,1},F_{0,2}$ and $L_1,L_2$ be the marginal cdfs of $F_0$ and $L$. The
marginals of~\eqref{eq:model} satisfy
$L_j=(1-p^\ast)F_{0,j}+p^\ast F_j$, $j=1,2$. Since $F_0$ is known, model
\eqref{eq:model} can be inverted for every candidate weight $p$. This
\emph{inversion step} produces two parametric families of functions on $\R^2$,
indexed by the weight. The first one is the family of candidate cdfs for the
unknown component,
\begin{equation}\label{eq:family1}
  \mathcal F_1=\Bigl\{F(\cdot\,;p)=\frac{L(\cdot)-(1-p)F_0(\cdot)}{p},\ p\in\Theta\Bigr\},
\end{equation}
and the second one is the family of products of their marginals,
\begin{equation}\label{eq:family2}
  \mathcal F_2=\Bigl\{F_{1\otimes2}(\cdot\,;p)=F_1(\cdot\,;p)\otimes F_2(\cdot\,;p),\
  F_j(\cdot\,;p)=\frac{L_j(\cdot)-(1-p)F_{0,j}(\cdot)}{p},\ j=1,2,\ p\in\Theta\Bigr\},
\end{equation}
where the weight is searched over a compact set
$\Theta=[\theta_l,\theta_u]\subset(0,1]$. Note that $F_j(\cdot\,;p)$ is
exactly the $j$th marginal of $F(\cdot\,;p)$, and that
$F(\cdot\,;p^\ast)=F$ and $F_j(\cdot\,;p^\ast)=F_j$. Hence $H_0$ holds if and
only if the two families intersect at $p=p^\ast$. The \emph{best-matching
step} measures how close the two families can be brought together through
the discrepancy process and the IBM contrast
\begin{equation}\label{eq:D}
  D(\bu;L,p)=F(\bu;p)-F_{1\otimes2}(\bu;p),\qquad
  d(p)=\int_{\R^2}D^2(\bu;L,p)\,dG(\bu)=\nG{D(\cdot\,;L,p)}^2 .
\end{equation}
Under $H_0$, $d(p^\ast)=0$. In the proofs it is convenient to reparametrise by
$t=1/p\in T=[t_l,t_u]$, with $t_l=1/\theta_u\ge1$ and $t_u=1/\theta_l<\infty$;
we write $t^\ast=1/p^\ast$.

\subsection{Explicit form of the inverted families}

Inserting $L=(1-p^\ast)F_0+p^\ast F$ and
$L_j=(1-p^\ast)F_{0,j}+p^\ast F_j$ into~\eqref{eq:family1}--\eqref{eq:family2}
shows that each inverted function is an explicit combination of the true
source and of the known one:
\begin{equation}\label{eq:inverted-explicit}
  F(\bu;p)=\frac{(p-p^\ast)F_0(\bu)+p^\ast F(\bu)}{p},\qquad
  F_j(u_j;p)=\frac{(p-p^\ast)F_{0,j}(u_j)+p^\ast F_j(u_j)}{p},\quad j=1,2.
\end{equation}
Define the known function $K_0$, the dependence function $K$ of the unknown
component and the marginal-gap function $\Delta$ by
\begin{equation}\label{eq:KDelta}
  K_0=F_0-F_{0,1}\otimes F_{0,2},\qquad
  K=F-F_1\otimes F_2,\qquad
  \Delta=(F_1-F_{0,1})\otimes(F_2-F_{0,2}),
\end{equation}
and the $L$-measurable functions
\begin{equation}\label{eq:ABC}
  A=L-F_0,\quad A_j=L_j-F_{0,j},\quad
  B^\ast=A-A_1\otimes F_{0,2}-F_{0,1}\otimes A_2,\quad
  C^\ast=A_1\otimes A_2 .
\end{equation}
The functions $B^\ast$ and $C^\ast$ do not depend on the candidate weight $p$.
They are determined by the true model, through $L=(1-p^\ast)F_0+p^\ast F$,
and the star stresses this dependence on $p^\ast$: by
Lemma~\ref{lem:algebra}(iii) below, $B^\ast=p^\ast(K+\Delta-K_0)$ and
$C^\ast=p^{\ast2}\Delta$. Their empirical counterparts $B_n$ and $C_n$, defined
in Section~\ref{sec:asympt}, are obtained by replacing $L$ with $\Ln$.

\begin{lemma}[Structure of the discrepancy]\label{lem:algebra}
For every $p\in(0,1]$, with $t=1/p$ and $s=p^\ast/p=p^\ast t$, the following hold.
\begin{enumerate}[label=(\roman*),leftmargin=*]
\item \emph{Expansion on the sources.}
\begin{equation}\label{eq:expansion}
\begin{aligned}
  p^2D(\cdot\,;L,p)={}&p\,p^\ast F-p^{\ast2}F_1\otimes F_2
  +(p-p^\ast)\bigl[p\,F_0-p^\ast\bigl(F_{0,1}\otimes F_2+F_1\otimes F_{0,2}\bigr)\bigr]\\
  &-(p-p^\ast)^2F_{0,1}\otimes F_{0,2}.
\end{aligned}
\end{equation}
\item \emph{Dependence form.}
\begin{equation}\label{eq:structure}
  D(\cdot\,;L,p)=(1-s)\,K_0+s(1-s)\,\Delta+s\,K .
\end{equation}
In particular, under $H_0$,
\begin{equation}\label{eq:structureH0}
  D(\cdot\,;L,p)=(1-s)\bigl(K_0+s\Delta\bigr)
  =(t-t^\ast)\,\Dd-p^{\ast2}(t-t^\ast)^2\Delta,
  \qquad \Dd:=-p^\ast\,(K_0+\Delta).
\end{equation}
\item \emph{Observable form.}
\begin{equation}\label{eq:poly}
  D(\cdot\,;L,p)=K_0+\frac{B^\ast}{p}-\frac{C^\ast}{p^2}=K_0+t\,B^\ast-t^2\,C^\ast ,
\end{equation}
where the coefficients, computed from $L$ and $F_0$ through~\eqref{eq:ABC},
are expressed in terms of the sources as
\begin{equation}\label{eq:BC}
  B^\ast=p^\ast\,(K+\Delta-K_0),\qquad C^\ast=p^{\ast2}\,\Delta .
\end{equation}
\end{enumerate}
\end{lemma}

Forms~(i) and~(ii) are linked by multiplying~\eqref{eq:structure} by $p^2$:
\begin{equation}\label{eq:bridge}
  p^2D(\cdot\,;L,p)=p(p-p^\ast)\,K_0+p^\ast(p-p^\ast)\,\Delta+p\,p^\ast K .
\end{equation}
Expanding $K_0$, $\Delta$ and $K$ gives~\eqref{eq:expansion}; conversely,
\eqref{eq:bridge} is obtained from~\eqref{eq:expansion} in three steps
(proof of Lemma~\ref{lem:algebra} in Appendix~\ref{app:proofs}). Note that $K_0$
does not appear as a block in~\eqref{eq:expansion}: its part
$-p(p-p^\ast)F_{0,1}\otimes F_{0,2}$ combines with the part
$+p^\ast(p-p^\ast)F_{0,1}\otimes F_{0,2}$ coming from $\Delta$ into
$-(p-p^\ast)^2F_{0,1}\otimes F_{0,2}$.

The three forms serve different purposes. The expansion~\eqref{eq:expansion}
writes the discrepancy between the two families directly in terms of the
sources and leads to the identification conditions below. The dependence
form~\eqref{eq:structure} explains what the contrast measures: the
discrepancy at $p$ combines the dependence of the known source, the marginal gap between the two sources and the dependence
of the unknown component. The observable form~\eqref{eq:poly} involves only
$L$ and the known $F_0$ and is the basis of the estimator of
Section~\ref{sec:asympt}: $D$ is a quadratic polynomial in $1/p$. Finally,
$\Dd$ is the derivative of $t\mapsto D(\cdot\,;L,1/t)$ at $t^\ast$, as
follows by differentiating~\eqref{eq:structureH0} in $t$; its relation to the
derivative in $p$ is given in Lemma~\ref{lem:contrast}(ii).

\subsection{Identifiability}

Since only $L$ is observed, the pair $(p^\ast,F)$ in~\eqref{eq:model} is in
general not identifiable \citep{BDV2006,PatraSen2016}. What the independence
constraint provides is precisely a way to single out $p^\ast$ under $H_0$.
The identification conditions below are stated in terms of the following
notion.

\begin{definition}[Linear independence in $L^2(G)$]\label{def:linind}
Let $G$ be a probability measure on $\R^2$ and let $f_1,\dots,f_m$ be bounded
measurable functions on $\R^2$. The family $\{f_1,\dots,f_m\}$ is
\emph{linearly independent in $L^2(G)$} if, for every
$\lambda=(\lambda_1,\dots,\lambda_m)^\top\in\R^m$,
\[
  \sum_{k=1}^m\lambda_kf_k(\bu)=0\quad\text{for $G$-almost every }\bu\in\R^2
  \qquad\Longrightarrow\qquad\lambda_1=\dots=\lambda_m=0 .
\]
\end{definition}

Since
$\bigl\lVert\sum_k\lambda_kf_k\bigr\rVert_G^2=\lambda^\top M_G\lambda$, with
Gram matrix $M_G=\bigl(\ip{f_k}{f_l}\bigr)_{1\le k,l\le m}$, the family is
linearly independent in $L^2(G)$ if and only if $M_G$ is non-singular. Equivalently,
its smallest eigenvalue $\mu_G=\lambda_{\min}(M_G)$ is positive, in which case
\begin{equation}\label{eq:gram}
  \Bigl\lVert\sum_{k=1}^m\lambda_kf_k\Bigr\rVert_G\ \ge\ \sqrt{\mu_G}\,|\lambda|\qquad\text{for all }\lambda\in\R^m .
\end{equation}
The notion depends on $G$ only through the sets it charges. Linear
independence in $L^2(G)$ implies ordinary linear independence of the
$f_k$ as functions on $\R^2$, and the two notions coincide when the $f_k$ are
continuous and $G$ has full support. For the empirical choice $G=\Ln$ used in
Section~\ref{sec:simu}, the definition reduces to the $n\times m$ matrix
$\bigl(f_k(\bm X_i)\bigr)_{i,k}$ having rank $m$.

Expansion~\eqref{eq:expansion} suggests the following identification
conditions under $H_0$ and $H_1$, respectively.

\medskip
\noindent\textbf{(I)} The family
$\{F_1\otimes F_2,\ F_0,\ F_{0,1}\otimes F_2+F_1\otimes F_{0,2},\ F_{0,1}\otimes F_{0,2}\}$
is linearly independent in $L^2(G)$.

\smallskip
\noindent\textbf{(II)} The family
$\{F,\ F_1\otimes F_2,\ F_0,\ F_{0,1}\otimes F_2+F_1\otimes F_{0,2},\ F_{0,1}\otimes F_{0,2}\}$
is linearly independent in $L^2(G)$.
\medskip

Condition~(I) excludes a known source of product form. The
asymptotic theory only requires the following weaker condition. Let
$S=\{p^\ast/p:\ p\in\Theta\}=[p^\ast t_l,\,p^\ast t_u]$, which contains $1$
when $p^\ast\in\Theta$.

\begin{assumption}\label{A:theta}
$p^\ast$ belongs to the interior of $\Theta=[\theta_l,\theta_u]\subset(0,1]$.
\end{assumption}

\begin{assumption}\label{A:ident}
$c_G:=\inf_{s\in S}\nG{K_0+s\Delta}>0$.
\end{assumption}

Since $s\mapsto\nG{K_0+s\Delta}^2$ is a quadratic polynomial, Assumption
\ref{A:ident} fails only if $K_0=-s\Delta$ $G$-a.e.\ for some $s\in S$. It
holds under~(I) (Proposition~\ref{prop:ident}), and more generally whenever
$K_0$ and $\Delta$ are linearly independent in $L^2(G)$. Unlike~(I), it
allows $F_0=F_{0,1}\otimes F_{0,2}$: then $K_0=0$ and \ref{A:ident} reduces to
$\nG{\Delta}>0$, i.e.\ both marginals of the unknown component differ from
those of the known one on a set charged by $G$.

\begin{proposition}\label{prop:ident}
\begin{enumerate}[label=(\roman*),leftmargin=*]
\item Condition~(I) implies Assumption~\ref{A:ident} for every compact $\Theta\subset(0,1]$.
\item Under $H_0$ and~\ref{A:ident}, the mapping $p\mapsto d(p)$ is a contrast
function on $\Theta$: for every $p\in\Theta$,
\begin{equation}\label{eq:lowerbound}
  \nG{D(\cdot\,;L,p)}\;\ge\;c_G\,p^\ast\,|t-t^\ast|,\qquad t=1/p,
\end{equation}
so that $d(p)=0$ if and only if $p=p^\ast$, and $d(p)\ge c_G^2p^{\ast2}(1/p-1/p^\ast)^2$.
\item Under $H_1$, let $d_{\min}=\min_{p\in\Theta}d(p)$. Then $d_{\min}>0$ if
and only if $F(\cdot\,;p)\neq F_{1\otimes2}(\cdot\,;p)$ on a set of positive
$G$-measure for every $p\in\Theta$, i.e.\ if and only if the families
$\mathcal F_1$ and $\mathcal F_2$ do not intersect. This holds under
condition~(II), and more generally whenever
$K\notin\operatorname{span}\{K_0,\Delta\}$ in $L^2(G)$.
\end{enumerate}
\end{proposition}

We retain in particular that, if $p^c\in\argmin_{p\in\Theta}d(p)$,
\begin{equation}\label{eq:pc}
  p^c=p^\ast\ \text{under $H_0$},\qquad d(p^c)=d_{\min}>0\ \text{under $H_1$ and (II)}.
\end{equation}

\begin{remark}
Part~(iii) shows what is really tested. In terms of the observable $L$, the
null hypothesis is ``the inverted families $\mathcal F_1$ and $\mathcal F_2$
intersect over $\Theta$''. The test is consistent against every alternative
with $d_{\min}>0$. When $G$ has full support and the cdfs are continuous,
$L^2(G)$-statements are equivalent to pointwise statements. The
alternatives that are \emph{absorbed} by a wrong weight, i.e.\ with
$d_{\min}=0$, are characterised exactly in Appendix~\ref{app:example}: they
are those for which $F=\lambda H+(1-\lambda)F_0$ for a function $H$ of
product form and some $\lambda\in\{p/p^\ast:\ p\in\Theta\}$. When $\lambda<1$,
the unknown component secretly contains a fraction $1-\lambda$ of the
known source, and $L$ admits a second representation with weight
$\lambda p^\ast$ and an independent component $H$. Their dependence functions
satisfy $K=(1-\lambda)K_0+\lambda^{-1}(1-\lambda)\Delta$, a one-parameter curve
in $\operatorname{span}\{K_0,\Delta\}$, and no procedure based
on~\eqref{eq:model} alone can detect them.
\end{remark}

\subsection{Regularity of the contrast}

The next lemma collects the analytic properties of $d$ used in
Section~\ref{sec:asympt}. Dots denote derivatives with respect to $p$.

\begin{lemma}\label{lem:contrast}
Let $\Theta=[\theta_l,\theta_u]\subset(0,1]$ with $\theta_l>0$.
\begin{enumerate}[label=(\roman*),leftmargin=*]
\item Under $H_0$ or $H_1$, $p\mapsto D(\bu;L,p)$ is infinitely differentiable
on $(0,1]$, uniformly in $\bu$, and differentiating the observable
form~\eqref{eq:poly} gives
\[
  \dot D(\cdot\,;L,p)=-\frac{B^\ast}{p^2}+\frac{2C^\ast}{p^3},\qquad
  \ddot D(\cdot\,;L,p)=\frac{2B^\ast}{p^3}-\frac{6C^\ast}{p^4},
\]
and $d$ is Lipschitz on $\Theta$:
$|d(p)-d(p')|\le 50\,\theta_l^{-5}\,|p-p'|$ for $p,p'\in\Theta$.
\item Under $H_0$ and~\ref{A:theta}--\ref{A:ident}, $d(p^\ast)=\dot d(p^\ast)=0$ and
the minimum at $p^\ast$ is non-degenerate:
\begin{equation}\label{eq:ddot}
  \ddot d(p^\ast)=2\int_{\R^2}\dot D^2(\bu;L,p^\ast)\,dG(\bu)
  =\frac{2}{p^{\ast2}}\,\nG{K_0+\Delta}^2\;\ge\;\frac{2c_G^2}{p^{\ast2}}>0 .
\end{equation}
Moreover, the derivatives of $D$ at the true weight in the two
parametrisations are linked by
\begin{equation}\label{eq:Ddot-link}
  K_0+\Delta\;=\;-\frac{\Dd}{p^\ast}\;=\;p^\ast\,\dot D(\cdot\,;L,p^\ast).
\end{equation}
Without~\ref{A:ident} the minimum can be degenerate: if $K_0=-\Delta$, then $\ddot d(p^\ast)=0$.
\item Under $H_1$, let $p^c$ be an interior minimiser of $d$ on $\Theta$. Then
$\dot d(p^c)=0$ and
\begin{equation}\label{eq:ddotc}
  \ddot d(p^c)=2\nG{\dot D(\cdot\,;L,p^c)}^2+2\ip{D(\cdot\,;L,p^c)}{\ddot D(\cdot\,;L,p^c)}\ \ge\ 0 .
\end{equation}
Unlike in~(ii), $D(\cdot\,;L,p^c)\neq0$ when $d_{\min}>0$, so the second term
does not vanish and has no definite sign. (When $d_{\min}=0$, the alternative is
absorbed in the sense of Appendix~\ref{app:example}, and~(ii) applies to the
absorbing representation.) Strict positivity holds under either of the
following conditions:
\begin{enumerate}[label=(\alph*),leftmargin=*]
\item $\sqrt{d_{\min}}\;\nG{\ddot D(\cdot\,;L,p^c)}<\nG{\dot D(\cdot\,;L,p^c)}^2$,
which holds in particular for alternatives close to the null hypothesis, where
$d_{\min}$ is small;
\item $t\mapsto d(1/t)$ is not of the form $d_{\min}+\nG{C^\ast}^2(t-1/p^c)^4$ on $T$.
\end{enumerate}
Condition~(b) is also necessary: $\ddot d(p^c)=0$ if and only if
$d(1/t)=d_{\min}+\nG{C^\ast}^2(t-1/p^c)^4$, a non-generic situation.
\end{enumerate}
\end{lemma}

\section{Estimation, test statistic and asymptotic theory}\label{sec:asympt}

\subsection{The estimator and the test}

Let $\bm X_1,\dots,\bm X_n$ be i.i.d.\ with cdf $L$, and let
\[
  \Ln(\bu)=\frac1n\sum_{i=1}^n\1\{\bm X_i\le\bu\},\qquad
  \widehat L_{n,j}(u_j)=\frac1n\sum_{i=1}^n\1\{X_{i,j}\le u_j\},
\]
the latter being the marginals of $\Ln$. Replacing $L$ by $\Ln$ in
\eqref{eq:family1}--\eqref{eq:family2} gives the empirical inverted families
\begin{align*}
  \widehat{\mathcal F}_{1}&=\Bigl\{\widehat F_n(\cdot\,;p)=\frac{\Ln-(1-p)F_0}{p},\ p\in\Theta\Bigr\},\\
  \widehat{\mathcal F}_{2}&=\Bigl\{\widehat F_{n,1}(\cdot\,;p)\otimes\widehat F_{n,2}(\cdot\,;p),\
  \widehat F_{n,j}(\cdot\,;p)=\frac{\widehat L_{n,j}-(1-p)F_{0,j}}{p},\ j=1,2,\ p\in\Theta\Bigr\},
\end{align*}
and the empirical discrepancy
$D(\cdot\,;\Ln,p)=\widehat F_n(\cdot\,;p)-\widehat F_{n,1}(\cdot\,;p)\otimes\widehat F_{n,2}(\cdot\,;p)$.
The IBM estimator of the weight is the value of $p$ that brings the two
empirical families closest together, and the test statistic is the
corresponding minimal distance:
\begin{equation}\label{eq:estimator}
  \ph\in\argmin_{p\in\Theta}d_n(p),\qquad
  d_n(p)=\int_{\R^2}D^2(\bu;\Ln,p)\,dG(\bu),\qquad
  \mathsf T_n=n\,d_n(\ph).
\end{equation}

\paragraph{Computation.} By Lemma~\ref{lem:algebra}(iii) applied to $\Ln$,
the empirical discrepancy is a quadratic polynomial in $t=1/p$,
\begin{equation}\label{eq:Dn}
\begin{gathered}
  D_n(\cdot\,;t):=D(\cdot\,;\Ln,1/t)=K_0+t\,B_n-t^2C_n,\\
  B_n=A_n-A_{n,1}\otimes F_{0,2}-F_{0,1}\otimes A_{n,2},\qquad C_n=A_{n,1}\otimes A_{n,2},
\end{gathered}
\end{equation}
with $A_n=\Ln-F_0$ and $A_{n,j}=\widehat L_{n,j}-F_{0,j}$. Consequently
\begin{equation}\label{eq:dn}
  \delta_n(t):=d_n(1/t)=\nG{D_n(\cdot\,;t)}^2=\sum_{k=0}^4 a_{n,k}\,t^k
\end{equation}
is a quartic polynomial. Its coefficients are identified by expanding the
squared norm by bilinearity and grouping the nine inner products according to
the power of $t$:
\[
\begin{aligned}
  \delta_n(t)&=\bigl\langle K_0+tB_n-t^2C_n,\;K_0+tB_n-t^2C_n\bigr\rangle_G\\
  &=\underbrace{\nG{K_0}^2}_{a_{n,0}}
   +t\,\underbrace{2\ip{K_0}{B_n}}_{a_{n,1}}
   +t^2\,\underbrace{\bigl(\nG{B_n}^2-2\ip{K_0}{C_n}\bigr)}_{a_{n,2}}
   +t^3\,\underbrace{\bigl(-2\ip{B_n}{C_n}\bigr)}_{a_{n,3}}
   +t^4\,\underbrace{\nG{C_n}^2}_{a_{n,4}} .
\end{aligned}
\]
Equivalently $a_{n,k}=\delta_n^{(k)}(0)/k!$. These five numbers are computed once.
For the empirical choice $G=\Ln$, each inner product is a sample average, e.g.\
$\ip{B_n}{C_n}=n^{-1}\sum_{m=1}^nB_n(\bm X_m)C_n(\bm X_m)$. Writing $\that=1/\ph$, the
minimisation in~\eqref{eq:estimator} is therefore exact: $\that$ is either an
endpoint of $T$ or one of the at most three real roots in $T$ of the cubic
$\delta_n'$. These roots are given by Cardano's formula, or by any polynomial
root solver, and $\that$ is the candidate with the smallest value of $\delta_n$.
The estimator therefore has no single closed-form expression, but it is
obtained exactly, without iterative optimisation. To avoid any ambiguity, $d$ and $d_n$ are always evaluated at a
weight $p$, while their expressions in $t=1/p$ are denoted
\[
  \delta_n(t)=d_n(1/t),\qquad \delta(t)=d(1/t),\qquad t\in T,
\]
so that, for instance, $\delta_n(\that)=d_n(\ph)$ and $\argmin_T\delta=\{1/p:\ p\in\argmin_\Theta d\}$. The test rejects $H_0$ for large values
of $\mathsf T_n$.

\begin{remark}[Which candidate is the global minimiser]\label{rem:roots}
The cubic $\delta_n'$ may have up to three real roots, but
identifiability (Assumption~\ref{A:ident}) only guarantees that $d$ has a
single \emph{zero} on $\Theta$, not a single critical point. Under $H_0$,
\eqref{eq:structureH0} gives $d\propto(1-s)^2\nG{K_0+s\Delta}^2$ with
$s=p^\ast/p$. Besides the double zero at $s=1$, this quartic typically has a
local maximum and a second local minimum. In the designs of
Section~\ref{sec:simu}, the second local minimum lies at $s\le0$, i.e.\
outside $T$, and the local maximum lies inside $T$ only when $p^\ast$ is
small. The exact minimisation nevertheless returns the consistent root. Under
$H_0$ and \ref{A:theta}--\ref{A:ident}, the empirical contrast at the root
near $t^\ast$ is $O_P(n^{-1})$ (Corollary~\ref{thm:null}). By
\eqref{eq:lowerbound}, every competitor at a fixed distance from $t^\ast$
has a value bounded away from zero: in particular
$\delta_n(t_l)\to\delta(t_l)=d(\theta_u)\ge c_G^2p^{\ast2}(t^\ast-t_l)^2>0$ at the endpoint
$p=\theta_u$, while $\delta_n(t_u)$ is even larger since the quartic grows like
$\nG{C_n}^2t^4$. Moreover, $\ddot d(p^\ast)>0$ (Lemma~\ref{lem:contrast})
and the coefficients of the cubic converge almost surely
(Theorem~\ref{thm:consistency}). Hence, with probability tending to one, the
cubic has exactly one root in a fixed neighbourhood of $t^\ast$, and
$\that=1/\ph$ is that root. In finite samples the only serious competitor is
the endpoint $p=\theta_u$, beyond the local maximum. It wins when
$n\,d(\theta_u)$ is not large relative to the sampling noise, which explains
the boundary estimates and the size distortion observed in Scenario~B with
$p^\ast=0.3$ and $n=250$ (Section~\ref{sec:simu}). A boundary value
$\ph=\theta_u$ should therefore be reported as a warning sign of weak
identification.
\end{remark}

\subsection{Almost sure behaviour}

\begin{theorem}\label{thm:consistency}
Assume $\theta_l>0$. Then, under both $H_0$ and $H_1$,
\begin{equation}\label{eq:unifrate}
\begin{gathered}
  \sup_{p\in\Theta}\nI{D(\cdot\,;\Ln,p)-D(\cdot\,;L,p)}
  =O_{a.s.}\Bigl(\sqrt{\log\log n/n}\Bigr),\\
  \sup_{p\in\Theta}|d_n(p)-d(p)|=O_{a.s.}\Bigl(\sqrt{\log\log n/n}\Bigr).
\end{gathered}
\end{equation}
Consequently:
\begin{enumerate}[label=(\roman*),leftmargin=*]
\item under $H_0$ and \ref{A:ident}, $|\ph-p^\ast|=O_{a.s.}(\sqrt{\log\log n/n})$;
\item under $H_1$, $d_n(\ph)\to d_{\min}$ a.s.; if $d_{\min}>0$ then
  $\mathsf T_n\to\infty$ a.s.; if moreover $d$ has a unique minimiser $p^c$ on
  $\Theta$, then $\ph\to p^c$ a.s.
\end{enumerate}
\end{theorem}

The rate in~\eqref{eq:unifrate} comes from the law of the iterated logarithm
for the bivariate empirical cdf,
$\nI{\Ln-L}=O_{a.s.}(\sqrt{\log\log n/n})$ \citep{Kiefer1961,Alexander1984}.
It holds for every $L$, continuous or not, because lower-left quadrants form a
Vapnik--\v{C}ervonenkis class. Since the marginals of $\Ln$ are the
empirical marginal cdfs, the same bound holds for
$\max_j\lVert\widehat L_{n,j}-L_j\rVert_\infty$, exactly as in the univariate
case \citep{ShorackWellner1986}.

\subsection{Asymptotic theory under both hypotheses}\label{sec:unified}

The behaviour of $\mathsf T_n$ under $H_0$ and under $H_1$ follows from a
single functional central limit theorem for the pair formed by the estimated
weight and the estimated discrepancy, centred at the \emph{pseudo-true} weight
\[
  p^c\in\argmin_{p\in\Theta}d(p),\qquad t^c=1/p^c,\qquad d_{\min}=d(p^c)=\delta(t^c),
\]
in the spirit of \citet{MPSV2024}. Under $H_0$ and~\ref{A:ident}, $d$ vanishes
only at $p^\ast$, so $p^c=p^\ast$, $t^c=t^\ast$ and $d_{\min}=0$.

\paragraph{Population quantities.} Write $D(t)=D(\cdot\,;L,1/t)=K_0+tB^\ast-t^2C^\ast$
for the population discrepancy as a function of $t$, and set
\[
  D^c:=D(t^c),\qquad D'_c:=D'(t^c)=B^\ast-2t^cC^\ast,\qquad
  J^c:=\delta''(t^c)=2\nG{D'_c}^2+2\ip{D^c}{D''},
\]
with $D''=-2C^\ast$. The first-order condition
\begin{equation}\label{eq:orth}
  \delta'(t^c)=2\ip{D^c}{D'_c}=0
\end{equation}
says that the residual $D^c$ is orthogonal in $L^2(G)$ to the tangent $D'_c$ of the
curve $t\mapsto D(t)$. By the chain rule~\eqref{eq:chain} and $\dot d(p^c)=0$,
the curvatures in the two parametrisations are linked by
$J^c=p^{c4}\,\ddot d(p^c)$.

In terms of the sources, let $s_c=p^\ast t^c=p^\ast/p^c$. By
Lemma~\ref{lem:algebra}(iii), $B^\ast=p^\ast(K+\Delta-K_0)$ and
$C^\ast=p^{\ast2}\Delta$, hence
\begin{equation}\label{eq:Dc-sources}
  D^c=(1-s_c)K_0+s_c(1-s_c)\Delta+s_cK,\qquad
  D'_c=B^\ast-2t^cC^\ast=p^\ast\bigl[K-K_0+(1-2s_c)\Delta\bigr].
\end{equation}
The first identity is~\eqref{eq:structure} at $p=p^c$. The second is the
derivative of~\eqref{eq:structure} with respect to $s$ multiplied by
$ds/dt=p^\ast$. Under $H_0$, $K=0$ and $p^c=p^\ast$, so $s_c=1$. Then
$D^c=0$ and~\eqref{eq:Dc-sources} reduces to
\[
  D'_c=\Dd=-p^\ast(K_0+\Delta)=-p^\ast\Psi,\qquad
  \Psi:=K_0+\Delta=p^\ast\dot D(\cdot\,;L,p^\ast)
\]
(see~\eqref{eq:Ddot-link} in Lemma~\ref{lem:contrast}(ii)). The direction
$\Psi$ is therefore a null object. Under $H_1$, the tangent $D'_c$ involves $K$
and is evaluated at $s_c\neq1$ in general. Even at the true weight, $s=1$, it
equals $-p^\ast(K_0+\Delta-K)$ rather than $-p^\ast\Psi$. Under $H_0$,
\eqref{eq:ddot} gives
$J^c=2p^{\ast2}\nG\Psi^2\ge2p^{\ast2}c_G^2>0$ under~\ref{A:ident}.

\paragraph{Function spaces and the projection.} Let $L^2(G)$ denote the
Hilbert space of (equivalence classes of) Borel functions $f:\R^2\to\R$ with
$\nG f^2=\int f^2\,dG<\infty$, equipped with the inner product $\ip\cdot\cdot$.
Bounded Borel functions belong to $L^2(G)$, since $G$ is a probability
measure. Let $\Pi:L^2(G)\to L^2(G)$ be the orthogonal projection onto
$\operatorname{span}\{\Psi\}=\operatorname{span}\{\Dd\}$,
\[
  \Pi f=\frac{\ip f\Psi}{\nG\Psi^2}\,\Psi,\qquad f\in L^2(G).
\]
It is linear, self-adjoint and idempotent, $\lVert\Pi\rVert_{\mathrm{op}}=\lVert I-\Pi\rVert_{\mathrm{op}}=1$, and
\begin{equation}\label{eq:L2-linf}
  \nG{(I-\Pi)z}\le\nG z\le\nI z
\end{equation}
for every $z$ in the closed subspace $\mathcal B(\R^2)\subset\ell^\infty(\Rb^2)$ of
bounded Borel functions, on which $\Pi$ acts through $G$-equivalence classes.

\paragraph{Empirical process.} For $\bx\in\R^2$, let
\[
  \mathsf e_\bx(\bv)=\1\{\bx\le\bv\},\qquad \bv\in\Rb^2,
\]
be the cdf of the point mass at $\bx$, a function of $\bv$ with $\bx$ fixed.
Then, for every $\bv\in\Rb^2$,
\[
  \Ln(\bv)=\frac1n\sum_{i=1}^n\mathsf e_{\bm X_i}(\bv),\qquad
  L(\bv)=\EE\,\mathsf e_{\bm X}(\bv).
\]
When a linear map is applied to $\mathsf e_\bx$, it acts on the function
$\bv\mapsto\1\{\bx\le\bv\}$, and the result is then evaluated at a point
$\bu$. Let $\G_n=\sqrt n(\Ln-L)$ and let $\G$ be an
$L$-Brownian bridge on $\Rb^2$. It is the centred Gaussian process indexed
by $\bu=(u_1,u_2)\in\Rb^2$ whose covariance function is that of the indicator
$\1\{\bm X\le\bu\}$:
\begin{equation}\label{eq:bridge-cov}
  \EE\bigl[\G(\bu)\G(\bv)\bigr]
  =\PP(\bm X\le\bu,\ \bm X\le\bv)-L(\bu)L(\bv)
  =L(\bu\wedge\bv)-L(\bu)L(\bv),
\end{equation}
where $\bu\wedge\bv=(u_1\wedge v_1,\,u_2\wedge v_2)$ is the componentwise
minimum and, in the model, $L(\bu\wedge\bv)=(1-p^\ast)F_0(\bu\wedge\bv)+p^\ast F(\bu\wedge\bv)$.
Equivalently, for any points $\bu_1,\dots,\bu_k\in\Rb^2$, the vector
$(\G(\bu_1),\dots,\G(\bu_k))$ is $\mathcal N_k(0,\Sigma)$ with
\[
  \Sigma_{ij}=L\bigl(u_{i1}\wedge u_{j1},\,u_{i2}\wedge u_{j2}\bigr)-L(\bu_i)L(\bu_j),
  \qquad 1\le i,j\le k ,
\]
and in particular $\operatorname{Var}\G(\bu)=L(\bu)(1-L(\bu))$. The marginal
bridges $\G^{(1)}(x)=\G(x,+\infty)$ and $\G^{(2)}(y)=\G(+\infty,y)$ are
$L_1$- and $L_2$-Brownian bridges. Their cross-covariances with $\G$ follow
from~\eqref{eq:bridge-cov}:
\[
  \EE\bigl[\G(\bu)\G^{(1)}(x)\bigr]=L(u_1\wedge x,u_2)-L(\bu)L_1(x),\qquad
  \EE\bigl[\G^{(1)}(x)\G^{(2)}(y)\bigr]=L(x,y)-L_1(x)L_2(y).
\]
The process $\G_n$ has the same covariance~\eqref{eq:bridge-cov} for every
$n$, and $\G$ is its weak limit in $\ell^\infty(\Rb^2)$ \citep[Ch.~2.5]{vdVW1996}. Every process used below is
obtained from $\G_n$ by integrating it against a family of finite signed
measures, and its asymptotic behaviour follows from a single general fact.

\begin{lemma}[Linear images of the empirical process]\label{lem:linear}
Let $\mathcal T$ be an index set and, for each $\tau\in\mathcal T$, let
$\mu_\tau$ be a finite signed Borel measure on $\Rb^2$ with
$m_\mu:=\sup_\tau|\mu_\tau|(\Rb^2)<\infty$. For bounded Borel $w$ on $\Rb^2$, set
\[
  \mathcal M(w)(\tau)=\int w\,d\mu_\tau,\qquad
  f_\tau(\bx)=\mathcal M(\mathsf e_\bx)(\tau)=\int\mathsf e_\bx(\bv)\,d\mu_\tau(\bv)
  =\int\1\{\bx\le\bv\}\,d\mu_\tau(\bv).
\]
\begin{enumerate}[label=(\roman*),leftmargin=*]
\item Each $f_\tau$ is Borel with $|f_\tau|\le m_\mu$, and, exactly for every $n$,
\[
  \mathcal M(\G_n)(\tau)=\frac1{\sqrt n}\sum_{i=1}^n\bigl(f_\tau(\bm X_i)-\EE f_\tau(\bm X)\bigr).
\]
\item $\mathcal M(\G_n)\dto\mathcal M(\G)$ in $\ell^\infty(\mathcal T)$.
\item $\mathcal M(\G)$ is a centred Gaussian process with covariance
\begin{equation}\label{eq:lin-cov}
\begin{aligned}
  \EE\bigl[\mathcal M(\G)(\tau)\,\mathcal M(\G)(\tau')\bigr]
  &=\operatorname{Cov}\bigl(f_\tau(\bm X),f_{\tau'}(\bm X)\bigr)\\
  &=\iint\bigl[L(\bm a\wedge\bm b)-L(\bm a)L(\bm b)\bigr]\,d\mu_\tau(\bm a)\,d\mu_{\tau'}(\bm b).
\end{aligned}
\end{equation}
\end{enumerate}
\end{lemma}

In words, a bounded linear image of the bridge is the limit of the
normalised i.i.d.\ sums of its \emph{influence functions} $f_\tau$, the images
of the functions $\mathsf e_\bx$. Its covariance is the covariance of these
influence functions under $L$, which is~\eqref{eq:bridge-cov} integrated
against $\mu_\tau\otimes\mu_{\tau'}$. Joint statements for several such images follow by taking
the disjoint union of their index sets. Lemma~\ref{lem:linear} is a direct consequence of Donsker's theorem for
the empirical cdf and of the continuous mapping theorem. Equivalently, each
$f_\tau$ lies in $m_\mu$ times the closed symmetric convex hull of the Donsker
class $\{\1\{\cdot\le\bv\}:\bv\in\Rb^2\}$, so $\{f_\tau:\tau\in\mathcal T\}$ is
itself Donsker \citetext{\citealp[Thm.~2.10.3]{vdVW1996}; \citealp[Ch.~10]{Dudley1999}}, and
$\mathcal M(\G)$ is the $L$-Brownian bridge indexed by these functions
\citep[see also][Ch.~3]{ShorackWellner1986}.

\paragraph{Linear maps.} The empirical process enters through bounded linear
maps on $\ell^\infty(\Rb^2)$. With $\beta,\gamma^\ast$ as
in~\eqref{eq:operators}, let
\begin{equation}\label{eq:Lc}
  \mathcal L_c(w)=t^c\beta(w)-t^{c2}\gamma^\ast(w),\qquad
  \mathcal L_c'(w)=\beta(w)-2t^c\gamma^\ast(w),
\end{equation}
the derivatives of $L\mapsto D(t^c)$ and $L\mapsto D'(t^c)$ in the direction $w$
(Lemma~\ref{lem:decomp}). They are combined into the real linear form and the
function-valued linear map
\begin{equation}\label{eq:elln}
  \ell_c(w)=\underbrace{\ip{\mathcal L_c(w)}{D'_c}}_{\text{noise along the tangent}}
  +\underbrace{\ip{D^c}{\mathcal L_c'(w)}}_{\text{tilt of the tangent}},\qquad
  \Phi_c(w)=\underbrace{\mathcal L_c(w)}_{\text{weight fixed}}
  -\underbrace{\frac{2\ell_c(w)}{J^c}\,D'_c}_{\text{weight re-fitted}} .
\end{equation}
The form $\ell_c(\G_n)$ is the leading term of half the empirical score
$\sqrt n\,\delta_n'(t^c)$, and $\Phi_c$ propagates a perturbation $w$ of $L$
to the discrepancy, once the weight has been re-estimated.

These maps are of the form of Lemma~\ref{lem:linear}. Since
$w^{(1)}(u_1)=w(u_1,+\infty)$ and $w^{(2)}(u_2)=w(+\infty,u_2)$,
\eqref{eq:operators} gives $\mathcal L_c(w)(\bu)=\int w\,d\nu_\bu$ and
$\mathcal L'_c(w)(\bu)=\int w\,d\nu'_\bu$, with
\begin{equation}\label{eq:nu}
\begin{aligned}
  \nu_\bu&=t^c\bigl[\delta_\bu-F_{0,2}(u_2)\,\delta_{(u_1,+\infty)}-F_{0,1}(u_1)\,\delta_{(+\infty,u_2)}\bigr]
  -t^{c2}\bigl[A_2(u_2)\,\delta_{(u_1,+\infty)}+A_1(u_1)\,\delta_{(+\infty,u_2)}\bigr],\\
  \nu'_\bu&=\bigl[\delta_\bu-F_{0,2}(u_2)\,\delta_{(u_1,+\infty)}-F_{0,1}(u_1)\,\delta_{(+\infty,u_2)}\bigr]
  -2t^c\bigl[A_2(u_2)\,\delta_{(u_1,+\infty)}+A_1(u_1)\,\delta_{(+\infty,u_2)}\bigr].
\end{aligned}
\end{equation}
The real form $\ell_c$ is obtained by integrating these coordinates against
$G$. Inserting~\eqref{eq:nu} into~\eqref{eq:elln} and collecting the
coefficients of $w(\bv)$, $w(v_1,+\infty)$ and $w(+\infty,v_2)$ gives, with the
bounded functions
\[
  a_c(\bv)=t^cD'_c(\bv)+D^c(\bv),\qquad b_c(\bv)=t^{c2}D'_c(\bv)+2t^cD^c(\bv),
\]
the explicit expression
\begin{equation}\label{eq:ell-explicit}
\begin{aligned}
  \ell_c(w)=\int_{\R^2}\Bigl\{&a_c(\bv)\,w(\bv)
  -\bigl[a_c(\bv)F_{0,2}(v_2)+b_c(\bv)A_2(v_2)\bigr]\,w(v_1,+\infty)\\
  &-\bigl[a_c(\bv)F_{0,1}(v_1)+b_c(\bv)A_1(v_1)\bigr]\,w(+\infty,v_2)\Bigr\}\,dG(\bv).
\end{aligned}
\end{equation}
In the language of Lemma~\ref{lem:linear}, $\ell_c(w)=\int w\,d\kappa_c$, where
$\kappa_c$ is the finite signed Borel measure on $\Rb^2$ that assigns to a
Borel set $B$ the mass
\[
  \kappa_c(B)=\int_{\R^2}\bigl[D'_c(\bv)\,\nu_\bv(B)+D^c(\bv)\,\nu'_\bv(B)\bigr]\,dG(\bv),
\]
a mixture over $\bv$ of the measures~\eqref{eq:nu} with weights $D'_c(\bv)$ and
$D^c(\bv)$. It is a single measure, not indexed by any argument, because $\ell_c$
is real-valued. In contrast, $\Phi_c(w)$ is a function of $\bu$, and its
coordinate at a fixed $\bu$ is $\Phi_c(w)(\bu)=\int w\,d\nu^\Phi_\bu$ with
$\nu^\Phi_\bu=\nu_\bu-(2D'_c(\bu)/J^c)\,\kappa_c$, where $2D'_c(\bu)/J^c$ is a
number. All these measures have uniformly bounded total variation, because
$|F_{0,j}|,|A_j|\le1$ and $D^c$, $D'_c$ are bounded. Integrating $\bv\mapsto\1\{\bx\le\bv\}$ against
the atoms in~\eqref{eq:nu}, as $\1\{\bx\le(u_1,+\infty)\}=\1\{x_1\le u_1\}$
and $\1\{\bx\le(+\infty,u_2)\}=\1\{x_2\le u_2\}$, gives the influence function
of $\mathcal L_c$,
\[
\begin{aligned}
  \mathcal L_c(\mathsf e_\bx)(\bu)={}&t^c\bigl[\1\{\bx\le\bu\}-\1\{x_1\le u_1\}F_{0,2}(u_2)-F_{0,1}(u_1)\1\{x_2\le u_2\}\bigr]\\
  &-t^{c2}\bigl[\1\{x_1\le u_1\}A_2(u_2)+A_1(u_1)\1\{x_2\le u_2\}\bigr],
\end{aligned}
\]
and similarly for $\mathcal L'_c$.

The empirical process enters these maps as a normalised sum. For every
$\bv\in\Rb^2$,
\[
  \G_n(\bv)=\sqrt n\,\bigl(\Ln(\bv)-L(\bv)\bigr)
  =\sqrt n\cdot\frac1n\sum_{i=1}^n\bigl(\1\{\bm X_i\le\bv\}-L(\bv)\bigr)
  =\frac1{\sqrt n}\sum_{i=1}^n\bigl(\mathsf e_{\bm X_i}(\bv)-L(\bv)\bigr).
\]
As functions of $\bv$, $\G_n=n^{-1/2}\sum_i(\mathsf e_{\bm X_i}-L)$ is a sum of
i.i.d.\ bounded random functions, centred because
$\EE\,\mathsf e_{\bm X}(\bv)=L(\bv)$. By linearity,
$\ell_c(\G_n)=n^{-1/2}\sum_i\bigl(\ell_c(\mathsf e_{\bm X_i})-\ell_c(L)\bigr)$.
Here $\ell_c(\mathsf e_{\bm X_i})=\int\1\{\bm X_i\le\bv\}\,d\kappa_c(\bv)=\lambda_c(\bm X_i)$
is the influence function below, evaluated at the $i$-th observation. The centring constant is identified through the
measure representation. A general bounded linear functional on $\ell^\infty$
need not commute with expectations, but $\ell_c(w)=\int w\,d\kappa_c$, so
Fubini's theorem gives
\[
  \ell_c(L)=\int\EE\,\1\{\bm X\le\bv\}\,d\kappa_c(\bv)
  =\EE\int\1\{\bm X\le\bv\}\,d\kappa_c(\bv)=\EE\lambda_c(\bm X).
\]
The same argument applies to $\Phi_c(\cdot)(\bu)$, for each $\bu$, with the
measure $\nu_\bu-2D'_c(\bu)\kappa_c/J^c$. This is Lemma~\ref{lem:linear}(i), which
yields the i.i.d.\ representations
\begin{equation}\label{eq:iid-rep}
  \ell_c(\G_n)=\frac1{\sqrt n}\sum_{i=1}^n\bigl(\lambda_c(\bm X_i)-\EE\lambda_c(\bm X)\bigr),\qquad
  \Phi_c(\G_n)(\bu)=\frac1{\sqrt n}\sum_{i=1}^n\bigl(\phi_{c,\bu}(\bm X_i)-\EE\phi_{c,\bu}(\bm X)\bigr),
\end{equation}
with the influence functions
\begin{equation}\label{eq:infl-c}
\begin{gathered}
  \lambda_c(\bx)=\ell_c(\mathsf e_\bx)=\int_{\R^2}\Bigl[\mathcal L_c(\mathsf e_\bx)(\bv)\,D'_c(\bv)
  +D^c(\bv)\,\mathcal L'_c(\mathsf e_\bx)(\bv)\Bigr]dG(\bv),\\
  \phi_{c,\bu}(\bx)=\Phi_c(\mathsf e_\bx)(\bu)=\mathcal L_c(\mathsf e_\bx)(\bu)-\frac{2\lambda_c(\bx)}{J^c}\,D'_c(\bu).
\end{gathered}
\end{equation}
Taking $w=\mathsf e_\bx$ in~\eqref{eq:ell-explicit} gives the score influence
function explicitly,
\[
\begin{aligned}
  \lambda_c(\bx)=\int_{\R^2}\Bigl\{&a_c(\bv)\1\{\bx\le\bv\}
  -\bigl[a_cF_{0,2}+b_cA_2\bigr](\bv)\,\1\{x_1\le v_1\}\\
  &-\bigl[a_cF_{0,1}+b_cA_1\bigr](\bv)\,\1\{x_2\le v_2\}\Bigr\}\,dG(\bv),
\end{aligned}
\]
where $[a_cF_{0,2}+b_cA_2](\bv)=a_c(\bv)F_{0,2}(v_2)+b_c(\bv)A_2(v_2)$, and similarly for the other bracket.
The identities~\eqref{eq:iid-rep} are exact for every $n$. Their summands are
i.i.d., centred and bounded uniformly in $(\bx,\bu)$. The central limit theorem
therefore applies directly, and Lemma~\ref{lem:linear}(ii)--(iii) identifies the
joint Gaussian limit of $(\ell_c(\G_n),\Phi_c(\G_n))$, whose covariance is that
of the influence functions.

\begin{theorem}[Joint central limit theorem]\label{thm:unified}
Assume either
\begin{enumerate}[label=(\alph*),leftmargin=*]
\item $H_0$ and~\ref{A:theta}--\ref{A:ident} (then $p^c=p^\ast$); or
\item $H_1$, $\theta_l>0$, and $d$ has a unique minimiser $p^c$ on $\Theta$, which
  lies in the interior of $\Theta$ and satisfies $\ddot d(p^c)>0$ (see
  Lemma~\ref{lem:contrast}(iii) for sufficient conditions).
\end{enumerate}
Then $J^c>0$ and, with $\mathbb W_n:=\sqrt n\bigl(D(\cdot\,;\Ln,\ph)-D(\cdot\,;L,p^c)\bigr)$,
\begin{align}
  \sqrt n(\ph-p^c)&=\frac{2p^{c2}}{J^c}\,\ell_c(\G_n)+O_P(n^{-1/2}),\label{eq:unif-p}\\
  \mathbb W_n&=\Phi_c(\G_n)+O_P(n^{-1/2})\quad\text{in }\ell^\infty(\Rb^2).\label{eq:unif-D}
\end{align}
Consequently, jointly in $\R\times\ell^\infty(\Rb^2)$,
\begin{equation}\label{eq:joint-clt}
  \Bigl(\sqrt n(\ph-p^c),\ \mathbb W_n\Bigr)\dto
  \Bigl(\frac{2p^{c2}}{J^c}\,\ell_c(\G),\ \Phi_c(\G)\Bigr),
\end{equation}
a centred Gaussian element. By Lemma~\ref{lem:linear}, its covariance is
that of the vector
\[
  \Bigl(\frac{2p^{c2}}{J^c}\,\lambda_c(\bm X),\ \bigl(\phi_{c,\bu}(\bm X)\bigr)_{\bu\in\Rb^2}\Bigr).
\]
\end{theorem}

The two terms of $\Phi_c$ in~\eqref{eq:elln} explain the structure of the limit. If the
weight were known, $\mathbb W_n$ would reduce to $\mathcal L_c(\G_n)$. Estimating
it moves the discrepancy along the tangent $D'_c$ by the amount
$\sqrt n(\that-t^c)\approx-2\ell_c(\G_n)/J^c$, i.e.\ score over curvature. The
test statistic is a fixed quadratic functional of $\mathbb W_n$.

\begin{corollary}[Representation of the statistic]\label{cor:repr}
Under the assumptions of Theorem~\ref{thm:unified}, exactly,
\begin{equation}\label{eq:repr}
  \mathsf T_n=\bigl\lVert\underbrace{\sqrt n\,D^c}_{\text{signal}}+\underbrace{\mathbb W_n}_{\text{noise}}\bigr\rVert_G^2
  =n\,d_{\min}+2\sqrt n\,\ip{D^c}{\mathbb W_n}+\nG{\mathbb W_n}^2 ,
\end{equation}
and, by~\eqref{eq:orth}, $\ip{D^c}{\Phi_c(w)}=\ip{D^c}{\mathcal L_c(w)}$ for every $w$.
Moreover, under $H_0$ as well as under $H_1$, the noise converges in law,
\begin{equation}\label{eq:noise-limit}
  \nG{\mathbb W_n}^2\dto\mathsf T^c_\infty:=\nG{\Phi_c(\G)}^2
  =\int_{\R^2}\bigl[\Phi_c(\G)(\bu)\bigr]^2\,dG(\bu),
\end{equation}
where $\Phi_c(\G)$ is the centred Gaussian process with covariance
$\operatorname{Cov}(\phi_{c,\bu}(\bm X),\phi_{c,\bv}(\bm X))$ (Lemma~\ref{lem:linear}).
In particular
$\EE\,\mathsf T^c_\infty=\int\operatorname{Var}\phi_{c,\bu}(\bm X)\,dG(\bu)$.
\end{corollary}

\begin{remark}[Structure of the reference law]\label{rem:Vstat}
With the centred influence functions
$\bar\phi_{c,\bu}=\phi_{c,\bu}-\EE\phi_{c,\bu}(\bm X)$, define the bounded,
symmetric kernel
\begin{equation}\label{eq:kernel}
  \bar h_c(\bx,\by)=\int_{\R^2}\bar\phi_{c,\bu}(\bx)\,\bar\phi_{c,\bu}(\by)\,dG(\bu).
\end{equation}
By~\eqref{eq:iid-rep} and Fubini's theorem, the leading term of the noise is exactly a $V$-statistic,
\[
  \nG{\Phi_c(\G_n)}^2=\frac1n\sum_{i,j=1}^n\bar h_c(\bm X_i,\bm X_j),
\]
as for the classical Cram\'er--von Mises and BKR statistics. The kernel is
degenerate, in the sense that $\int\bar h_c(\bx,\by)\,dL(\by)=0$ for every $\bx$,
which is why the limit is not Gaussian. Equivalently, $\mathsf T^c_\infty$ is a double Wiener--It\^o integral of
$\bar h_c$ with respect to $\G$, shifted by the trace
$\int\bar h_c(\bx,\bx)\,dL(\bx)=\EE\,\mathsf T^c_\infty$,
or a weighted series $\sum_k\mu_k\chi^2_{1,k}$ whose weights are the eigenvalues of
$\bar h_c$ \citep{DynkinMandelbaum1983}. The kernel depends on the unknown
$p^\ast$ and $F$, and its eigenvalues have no closed form, so neither
representation is used for calibration. The reference law is simulated instead
(Section~\ref{sec:bootstrap}).
\end{remark}

The asymptotic behaviour of the test statistic is read off~\eqref{eq:repr}
according to the size of the signal $\sqrt n\,D^c$: it vanishes under $H_0$,
is of order $\sqrt n$ under a fixed alternative, and is of order one under
local alternatives, while the noise $\nG{\mathbb W_n}^2$ remains of order one
in all cases. This suggests using the asymptotic law of the noise,
$\mathsf T^c_\infty$ in~\eqref{eq:noise-limit}, as the \emph{reference
distribution} of the test, the one whose quantiles must be tabulated, and
rejecting $H_0$ when $\mathsf T_n$ exceeds a high quantile of it. The same
reference is used whichever hypothesis holds: under $H_0$,
$\mathsf T_n=\nG{\mathbb W_n}^2$ exactly, so the test has asymptotic level
$\alpha$; under a fixed alternative, $\mathsf T_n-\nG{\mathbb W_n}^2=n\,d_{\min}+O_P(\sqrt n)\to\infty$
while the noise stays tight, so the power tends to one. How to approximate this
law from the data is the object of Section~\ref{sec:bootstrap}. The two
hypotheses are now made explicit, starting with the null.

\paragraph{The null process.} Theorem~\ref{thm:unified} involves no
process other than $\Phi_c(\G_n)$. Under $H_0$, however, this noise takes an
interpretable form. There $p^c=p^\ast$, so $t^c=t^\ast$, and
$t^\ast A_j=F_j-F_{0,j}$, so the measures~\eqref{eq:nu} reduce to $\nu_\bu=t^\ast\mu_\bu$ with
$\mu_\bu=\delta_\bu-F_2(u_2)\,\delta_{(u_1,+\infty)}-F_1(u_1)\,\delta_{(+\infty,u_2)}$
(Lemma~\ref{lem:decomp}(ii)). Equivalently, $\mathcal L_c(\G_n)=t^\ast\xi_n$ and
$\mathcal L_c(\G)=t^\ast\xi$, where
\begin{equation}\label{eq:xi}
  \xi_n=\G_n-\G_n^{(1)}\otimes F_2-F_1\otimes\G_n^{(2)},\qquad
  \xi=\G-\G^{(1)}\otimes F_2-F_1\otimes\G^{(2)}
\end{equation}
are the BKR independence process of the \emph{unobserved} component and its
limit. Thus $\xi$ is the image of $\G$ under the family $(\mu_\bu)$, an
instance of Lemma~\ref{lem:linear} with $\mathcal T=\Rb^2$ and $m_\mu\le3$, because $\G^{(1)}(u_1)=\G(u_1,+\infty)$ and $\G^{(2)}(u_2)=\G(+\infty,u_2)$.
Its influence function is obtained by integrating $\bv\mapsto\1\{\bx\le\bv\}$
against the three atoms of $\mu_\bu$. Since $\int g\,d\delta_{\bm a}=g(\bm a)$
and $x_1,x_2$ are finite, so that $\1\{\bx\le(u_1,+\infty)\}=\1\{x_1\le u_1\}$ and
$\1\{\bx\le(+\infty,u_2)\}=\1\{x_2\le u_2\}$,
\begin{equation}\label{eq:phi}
\begin{aligned}
  \varphi_{\bu}(\bx)=\int\1\{\bx\le\bv\}\,d\mu_\bu(\bv)
  &=\1\{\bx\le\bu\}-F_2(u_2)\,\1\{\bx\le(u_1,+\infty)\}-F_1(u_1)\,\1\{\bx\le(+\infty,u_2)\}\\
  &=\1\{\bx\le\bu\}-\1\{x_1\le u_1\}F_2(u_2)-F_1(u_1)\1\{x_2\le u_2\}.
\end{aligned}
\end{equation}
Conversely, $\int w\,d\mu_\bu=w(\bu)-w^{(1)}(u_1)F_2(u_2)-F_1(u_1)w^{(2)}(u_2)$
for every bounded $w$, which is~\eqref{eq:xi} when $w=\G_n$ or $w=\G$. In particular, taking $w=L$,
\[
  \int L\,d\mu_\bu=L(\bu)-L_1(u_1)F_2(u_2)-F_1(u_1)L_2(u_2)=\EE\varphi_\bu(\bm X).
\]
Inserting $\G_n(\bv)=n^{-1/2}\sum_i(\1\{\bm X_i\le\bv\}-L(\bv))$ and
integrating term by term against $\mu_\bu$ therefore gives
\[
  \xi_n(\bu)=\int\G_n\,d\mu_\bu
  =\frac1{\sqrt n}\sum_{i=1}^n\Bigl(\int\1\{\bm X_i\le\bv\}\,d\mu_\bu(\bv)-\int L\,d\mu_\bu\Bigr)
  =\frac1{\sqrt n}\sum_{i=1}^n\bigl(\varphi_\bu(\bm X_i)-\EE\varphi_\bu(\bm X)\bigr).
\]
Lemma~\ref{lem:linear} then gives $\xi_n\dto\xi$ in $\ell^\infty(\Rb^2)$,
where $\xi$ is centred Gaussian with covariance
\begin{equation}\label{eq:Gamma}
  \Gamma(\bu,\bv)=\operatorname{Cov}\bigl(\varphi_{\bu}(\bm X),\varphi_{\bv}(\bm X)\bigr).
\end{equation}
Expanding~\eqref{eq:lin-cov}, $\Gamma(\bu,\bv)$ is the sum of the nine terms
$c\,c'\,[L(\bm a\wedge\bm b)-L(\bm a)L(\bm b)]$, where $(c,\bm a)$ and
$(c',\bm b)$ range over the three atoms of $\mu_\bu$ and $\mu_\bv$. For instance,
the first term is $L(\bu\wedge\bv)-L(\bu)L(\bv)$, and the term pairing
$-F_2(u_2)\delta_{(u_1,+\infty)}$ with $-F_1(v_1)\delta_{(+\infty,v_2)}$ is
\[
  F_2(u_2)F_1(v_1)\,\bigl[L(u_1,v_2)-L_1(u_1)L_2(v_2)\bigr].
\]
The real-valued functional $\ip{\xi_n}\Psi$ is covered too, with
$\mu=\int\mu_\bv\,\Psi(\bv)\,dG(\bv)$ and influence function
$\int\varphi_\bv\Psi(\bv)\,dG(\bv)$. Combined with $\Phi_c=(I-\Pi)\mathcal L_c$
under $H_0$ (Corollary~\ref{thm:null} below), this shows that the null noise is
$\mathbb W_n=t^\ast(I-\Pi)\xi_n+O_P(n^{-1/2})$.
The null limit is non-degenerate under the following condition.

\begin{assumption}\label{A:nondeg}
The process $(I-\Pi)\xi$ is not $G$-a.e.\ degenerate, i.e.\
$\EE\nG{(I-\Pi)\xi}^2>0$.
\end{assumption}

\begin{corollary}[Null hypothesis]\label{thm:null}
Assume $H_0$ and \ref{A:theta}--\ref{A:ident}. Then $\mathcal L_c(\G_n)=t^\ast\xi_n$,
$\Phi_c(w)=(I-\Pi)\mathcal L_c(w)$ and $\mathsf T_n=\nG{\mathbb W_n}^2$, so that
\begin{align}
  \sqrt n(\ph-p^\ast) &= -\frac{\ip{\xi_n}{\Psi}}{\nG\Psi^2}+O_P(n^{-1/2}), \label{eq:phat-lin}\\
  \mathsf T_n &= \frac{1}{p^{\ast2}}\,\nG{(I-\Pi)\xi_n}^2+O_P(n^{-1/2}). \label{eq:T-lin}
\end{align}
Consequently, jointly,
\begin{equation}\label{eq:limits}
\begin{gathered}
  \sqrt n(\ph-p^\ast)\dto\mathcal N(0,\sigma^2),\qquad
  \mathsf T_n=\nG{\mathbb W_n}^2\dto \mathsf T_\infty,\\
  \mathsf T_\infty:=\mathsf T^c_\infty=\nG{\Phi_c(\G)}^2=\frac{1}{p^{\ast2}}\int_{\R^2}\bigl[(I-\Pi)\xi\bigr]^2(\bu)\,dG(\bu),
\end{gathered}
\end{equation}
with $\sigma^2=\nG\Psi^{-4}\iint\Gamma(\bu,\bv)\Psi(\bu)\Psi(\bv)\,dG(\bu)dG(\bv)$.
The kernel of Remark~\ref{rem:Vstat} is, with $\bar\varphi_\bu=\varphi_\bu-\EE\varphi_\bu(\bm X)$,
\begin{equation}\label{eq:null-kernel}
  \bar h_0(\bx,\by)=\frac1{p^{\ast2}}\int_{\R^2}\bigl[(I-\Pi)\bar\varphi_\cdot(\bx)\bigr](\bu)\,
  \bigl[(I-\Pi)\bar\varphi_\cdot(\by)\bigr](\bu)\,dG(\bu),
\end{equation}
where $(I-\Pi)$ acts on $\bu\mapsto\bar\varphi_\bu(\bx)$, and
$\EE\,\mathsf T_\infty=\int\bar h_0(\bx,\bx)\,dL(\bx)$. Under~\ref{A:nondeg},
$\mathsf T_\infty$ has a continuous and strictly increasing cdf on $[0,\infty)$.
\end{corollary}

\begin{corollary}[Fixed alternative]\label{thm:alt}
Assume case~(b) of Theorem~\ref{thm:unified} and $d_{\min}>0$. Then
$\ip{D^c}{\Phi_c(w)}=\ip{D^c}{\mathcal L_c(w)}$ for every $w$, and the statistic
decomposes exactly as
\begin{equation}\label{eq:alt-decomp}
  \mathsf T_n=\underbrace{n\,d_{\min}}_{\text{signal}}
  +\underbrace{2\sqrt n\,\ip{D^c}{\mathbb W_n}}_{\text{cross term}}
  +\underbrace{\nG{\mathbb W_n}^2}_{\text{noise}} ,
\end{equation}
whose three terms are of order $n$, $\sqrt n$ and $1$. Consequently
\begin{align}
  \sqrt n(\ph-p^c) &= \frac{2p^{c2}}{J^c}\,\frac1{\sqrt n}\sum_{i=1}^n\bigl(\lambda_c(\bm X_i)-\EE\lambda_c(\bm X)\bigr)+O_P(n^{-1/2}), \label{eq:alt-p}\\
  \mathsf T_n &= n\,d_{\min}+2\sum_{i=1}^n\bigl(\chi_c(\bm X_i)-\EE\chi_c(\bm X)\bigr)+O_P(1), \label{eq:alt-d}
\end{align}
with $\chi_c(\bx)=\ip{D^c}{\mathcal L_c(\mathsf e_\bx)}$. Hence, jointly,
\begin{equation}\label{eq:alt-limits}
  \sqrt n(\ph-p^c)\dto\mathcal N(0,\sigma_c^2),\qquad
  \frac{\mathsf T_n-n\,d_{\min}}{\sqrt n}=\sqrt n\bigl(d_n(\ph)-d_{\min}\bigr)\dto\mathcal N(0,\tau^2),
\end{equation}
with
\begin{equation}\label{eq:infl-alt}
  \sigma_c^2=\frac{4p^{c4}}{J^{c2}}\operatorname{Var}\lambda_c(\bm X),\qquad
  \tau^2=4\operatorname{Var}\chi_c(\bm X),\qquad
  \text{covariance }\frac{4p^{c2}}{J^c}\operatorname{Cov}\bigl(\lambda_c(\bm X),\chi_c(\bm X)\bigr).
\end{equation}
In particular $\mathsf T_n=n\,d_{\min}+O_P(\sqrt n)\to\infty$, whereas the
noise converges in law, as under $H_0$, to the reference distribution
\begin{equation}\label{eq:noise-alt}
  \nG{\mathbb W_n}^2\dto\mathsf T^c_\infty=\nG{\Phi_c(\G)}^2,\qquad
  \Phi_c(w)=\mathcal L_c(w)-\frac{2}{J^c}\Bigl(\ip{\mathcal L_c(w)}{D'_c}+\ip{D^c}{\mathcal L_c'(w)}\Bigr)D'_c ,
\end{equation}
where, unlike under $H_0$, the tilt term $\ip{D^c}{\mathcal L_c'(w)}$ and the
term $2\ip{D^c}{D''}$ in $J^c$ do not vanish. The kernel~\eqref{eq:kernel}
of Remark~\ref{rem:Vstat} is built from these alternative influence functions.
\end{corollary}

\begin{remark}\label{rem:interpretation}
Four features are worth stressing. First, under $H_0$ the noise $\mathbb W_n$
is, up to $O_P(n^{-1/2})$, the BKR process $\xi_n$ of the \emph{unobserved}
component, written in terms of the observed empirical process and multiplied by
$t^\ast=1/p^\ast$: the statistic behaves as if only the $np^\ast$ observations from
the unknown component carried information. Second, the projection $I-\Pi$
is the price paid for estimating the weight: exactly one direction, $\Psi$, is
lost, as in Cram\'er--von Mises tests with estimated parameters. Third, under
a fixed alternative the estimation of the weight does not contribute to the
fluctuations of the statistic at order $\sqrt n$, because the re-fitting part
of $\Phi_c$ is orthogonal to $D^c$; and $\tau=0$ under $H_0$ explains why the
statistic must then be rescaled by $n$ rather than $\sqrt n$. Fourth, the null
limit depends on the unknown $(p^\ast,F_1,F_2)$, so the test is not
asymptotically distribution free and must be calibrated; see
Section~\ref{sec:bootstrap}.
\end{remark}

\begin{remark}[Local alternatives]\label{rem:local}
Consider alternatives $F=F_n$ with fixed marginals $F_1,F_2$ and dependence
$K_n=\kappa/\sqrt n$, $\kappa\in\ell^\infty(\Rb^2)$ (with zero marginals).
By~\eqref{eq:structure}, $\sqrt n D(\cdot\,;L_n,1/(t^\ast+h/\sqrt n))=h\Dd+\kappa+o(1)$
uniformly on compact sets of $h$. This is the intermediate regime
of~\eqref{eq:repr}, where signal and noise are of the same order. The argument
of Theorem~\ref{thm:unified} then gives, along the same lines (using a
triangular-array Donsker theorem),
$\mathsf T_n\dto\nG{(I-\Pi)(p^{\ast-1}\xi+\kappa)}^2$. The test therefore has
non-trivial local power against every direction $\kappa$ with
$(I-\Pi)\kappa\neq0$, the only undetectable local direction being $\Psi$. This
direction is the tangent at $\lambda=1$ of the curve of absorbed alternatives
constructed in Appendix~\ref{app:example}.
\end{remark}

\begin{remark}[Power approximation]\label{rem:power}
Assume $\tau>0$. For any constant $c>0$, Corollary~\ref{thm:alt} gives
\[
  \PP(\mathsf T_n>c)=\PP\Bigl(\sqrt n\bigl(d_n(\ph)-d_{\min}\bigr)>\frac{c}{\sqrt n}-\sqrt n\,d_{\min}\Bigr)
  \approx\Phi_{\mathcal N}\Bigl(\frac{\sqrt n\,d_{\min}-c/\sqrt n}{\tau}\Bigr),
\]
where $\Phi_{\mathcal N}$ is the standard normal cdf. With $c$ the critical value
$\widehat q_{1-\alpha}$, which is $O_P(1)$ by Theorem~\ref{thm:bootstrap}(ii), this
describes how the power approaches one. The signal-to-noise ratio that governs
it is $\sqrt n\,d_{\min}/\tau$. Both $d_{\min}$ and $\tau^2$ can be estimated by
plug-in, replacing $(L,p^c)$ by $(\Ln,\ph)$ in~\eqref{eq:infl-alt}.
\end{remark}

\subsection{Calibration by multiplier bootstrap}\label{sec:bootstrap}

\paragraph{What has to be calibrated.} By Corollary~\ref{cor:repr}, the
statistic is the squared norm of a signal plus a noise,
\[
  \mathsf T_n=\bigl\lVert\sqrt n\,D^c+\mathbb W_n\bigr\rVert_G^2,\qquad
  \mathbb W_n=\sqrt n\bigl(D(\cdot\,;\Ln,\ph)-D(\cdot\,;L,p^c)\bigr).
\]
The \emph{reference distribution} of the test is the asymptotic law of the
noise $\nG{\mathbb W_n}^2$. By Corollaries~\ref{cor:repr}--\ref{thm:alt}, it
exists under $H_0$ as well as under a fixed alternative:
\begin{equation}\label{eq:reference}
  \nG{\mathbb W_n}^2\dto\mathsf T^c_\infty=\nG{\Phi_c(\G)}^2,
  \qquad\text{with}\qquad
  \mathsf T^c_\infty=\mathsf T_\infty=\frac1{p^{\ast2}}\,\nG{(I-\Pi)\xi}^2\quad\text{under }H_0 .
\end{equation}
The test rejects $H_0$ when $\mathsf T_n$ exceeds a high quantile of this law,
and the same rule applies independently of the hypothesis in force:
\begin{itemize}[leftmargin=*]
\item under $H_0$ the signal vanishes and $\mathsf T_n=\nG{\mathbb W_n}^2$
  exactly, so $\mathsf T_n$ follows the reference law asymptotically and the
  level is $\alpha$;
\item under a fixed alternative, by~\eqref{eq:alt-decomp}, $\mathsf T_n$ exceeds
  the noise $\nG{\mathbb W_n}^2$, whose law is still close to the reference law,
  by the signal $n\,d_{\min}$ and the cross term $2\sqrt n\ip{D^c}{\mathbb W_n}$,
  which grow with $n$ and make the power tend to one
  (Corollary~\ref{thm:alt}).
\end{itemize}
Calibration thus amounts to approximating the law of the noise
$\nG{\mathbb W_n}^2$ from the data, without knowing which hypothesis holds.

\paragraph{Why a multiplier bootstrap.} The law~\eqref{eq:reference} depends
on the unknown quantities $p^\ast$, $F_1$ and $F_2$, both through the
covariance $\Gamma$ of $\xi$ and through the projection $\Pi$ that accounts
for the estimation of the weight. As for Cram\'er--von Mises statistics with
estimated parameters \citep{Durbin1973}, the test is therefore not
asymptotically distribution free, and its critical values must be estimated.
Three routes are available. (a) Tabulating the law of Corollary~\ref{thm:null}
through its weighted $\chi^2_1$ series (Remark~\ref{rem:Vstat}) requires the
eigenvalues of the unknown kernel $\bar h_0$, which must themselves be
estimated and truncated before an inversion formula such as
\citet{Imhof1961} can be applied. (b) The resampling bootstrap
\citep{GineZinn1990} or the parametric bootstrap of \citet{Stute1993}
requires redrawing samples from a distribution that satisfies $H_0$ and
re-estimating the weight on each of them. Here this means building a
contamination model whose unknown component has independent coordinates. The plug-in
candidate $(1-\ph)F_0+{}$\allowbreak$\ph\,\widehat F_{1}\otimes\widehat F_{2}$ need not be a
proper cdf, because the inverted marginals $\widehat F_j$ need not be
monotone. (c) The multiplier bootstrap avoids both difficulties: it simulates
copies of the noise $\mathbb W_n$ directly. The data are kept fixed and, in the
linear representation $\mathbb W_n=\Phi_c(\G_n)+O_P(n^{-1/2})$ of
Theorem~\ref{thm:unified}, the empirical process $\G_n$ is replaced by a
multiplier process and the unknown map $\Phi_c$ by a plug-in estimate. Each
replication then simulates the reference law~\eqref{eq:reference} directly,
under $H_0$ as well as under $H_1$,
with no resampling and no re-estimation of the weight. Its validity rests on
the conditional multiplier central limit theorem for Donsker classes; see
Section~2.9 of \citet{vdVW1996} and Chapter~10 of \citet{Kosorok2008}. It has
become the standard calibration tool for tests built on empirical copula
processes \citep{RemillardScaillet2009,KYH2011}, where it is much faster than
the parametric bootstrap at comparable accuracy \citep{KYH2011}, and it
compares favourably with resampling schemes in finite samples
\citep{BucherDette2010}. In our setting it is the natural choice because the
noise is, up to $O_P(n^{-1/2})$, a continuous linear map of the empirical
process. Moreover, by Theorem~\ref{thm:bootstrap}(ii), the bootstrap noise
still reproduces the law of the noise under the alternative, and in
particular remains bounded, so the test retains its power.

\paragraph{The procedure.} The bootstrap copy of $\mathbb W_n$ is built by
mimicking the representation $\mathbb W_n\approx\Phi_c(\G_n)$ in its two
ingredients.
\begin{itemize}[leftmargin=*]
\item \emph{The empirical process.} Let $\zeta_1,\dots,\zeta_n$ be i.i.d.\
  multipliers, independent of the data, with mean $0$, variance $1$ and finite
  $L^{2,1}$-norm,
  $\lVert\zeta_1\rVert_{2,1}=\int_0^\infty\sqrt{\PP(|\zeta_1|>x)}\,dx<\infty$
  (e.g.\ standard normal multipliers). The multiplier process
  \[
    \G_n^\zeta(\bu)=\frac1{\sqrt n}\sum_{i=1}^n\zeta_i\bigl(\1\{\bm X_i\le\bu\}-\Ln(\bu)\bigr)
  \]
  has, conditionally on the data, the same limit $\G$ as $\G_n$.
\item \emph{The map $\Phi_c$.} Its plug-in estimate replaces
  $(t^c,L)$ by $(\that,\Ln)$ in~\eqref{eq:Lc}--\eqref{eq:elln}. Let
  $\widehat D_n=D_n(\that)$, $\widehat D{}'_n=B_n-2\that C_n$ (the derivative of
  $t\mapsto D_n(\cdot\,;t)$ at $\that$; equivalently
  $\widehat D{}'_n=-\ph^{\,2}\,\dot D(\cdot\,;\Ln,\ph)$ by the chain
  rule~\eqref{eq:chain}), $\widehat J_n=\delta_n''(\that)=2\nG{\widehat D{}'_n}^2-4\ip{\widehat D_n}{C_n}$,
  $\widehat\gamma_n(w)=w^{(1)}\otimes A_{n,2}+A_{n,1}\otimes w^{(2)}$, and
  \[
  \begin{gathered}
    \widehat{\mathcal L}_n(w)=\that\beta(w)-\that^{\,2}\widehat\gamma_n(w)
    =\that\bigl(w-w^{(1)}\otimes\widehat F_2-\widehat F_1\otimes w^{(2)}\bigr),\qquad
    \widehat{\mathcal L}_n'(w)=\beta(w)-2\that\,\widehat\gamma_n(w),\\
    \widehat\Phi_n(w)=\widehat{\mathcal L}_n(w)
    -\frac{2}{\widehat J_n}\Bigl(\ip{\widehat{\mathcal L}_n(w)}{\widehat D{}'_n}+\ip{\widehat D_n}{\widehat{\mathcal L}_n'(w)}\Bigr)\widehat D{}'_n ,
  \end{gathered}
  \]
  where $\widehat F_j=\widehat F_{n,j}(\cdot\,;\ph)=\that A_{n,j}+F_{0,j}$ are the
  plug-in marginals of the unknown component. The map $\widehat\Phi_n$ estimates
  $\Phi_c$ under both hypotheses. Under $H_0$,
  $\widehat D_n=O_P(n^{-1/2})$ and $\widehat J_n=2\nG{\widehat D{}'_n}^2+O_P(n^{-1/2})$,
  so $\widehat\Phi_n$ reduces, up to $O_P(n^{-1/2})$ in operator norm, to its
  null form $(I-\widehat\Pi)\widehat{\mathcal L}_n$, where $\widehat\Pi$ is the
  $L^2(G)$-projection onto $\widehat D{}'_n$.
\end{itemize}
The bootstrap noise and the bootstrap statistic are
\begin{equation}\label{eq:bootstrap}
  \mathbb W_n^\zeta=\widehat\Phi_n(\G_n^\zeta),\qquad
  \mathsf T_n^\zeta=\nG{\mathbb W_n^\zeta}^2 .
\end{equation}
Its null form is $\that\,(I-\widehat\Pi)\,\xi_n^\zeta$, where
$\xi_n^\zeta=\widehat{\mathcal L}_n(\G_n^\zeta)/\that=\G_n^\zeta-\G_n^{\zeta(1)}\otimes\widehat F_2-\widehat F_1\otimes\G_n^{\zeta(2)}$
can be computed as
$\xi_n^\zeta(\bu)=n^{-1/2}\sum_{i}\zeta_i\bigl(\widehat\varphi_{\bu}(\bm X_i)-n^{-1}\sum_{k}\widehat\varphi_{\bu}(\bm X_k)\bigr)$,
with
$\widehat\varphi_{\bu}(\bx)=\1\{\bx\le\bu\}-\1\{x_1\le u_1\}\widehat F_2(u_2)-\widehat F_1(u_1)\1\{x_2\le u_2\}$.
This simpler version, used in the simulations of Section~\ref{sec:simu}, has
the same asymptotic level; the full version~\eqref{eq:bootstrap} also
reproduces the reference law under the alternative. The correspondence
between the noise and its bootstrap version is thus
\[
  \nG{\mathbb W_n}^2,\ \ \mathbb W_n=\Phi_c(\G_n)+O_P(n^{-1/2})
  \qquad\longleftrightarrow\qquad
  \mathsf T_n^\zeta=\nG{\mathbb W_n^\zeta}^2,\ \ \mathbb W_n^\zeta=\widehat\Phi_n(\G_n^\zeta),
\]
under $H_0$ as well as under $H_1$; under $H_0$, moreover, $\mathsf T_n=\nG{\mathbb W_n}^2$.
Let $\widehat q_{1-\alpha}$ be the conditional $(1-\alpha)$-quantile of
$\mathsf T_n^\zeta$ given the data, i.e.\ with respect to the law of the
multipliers only. The test is
\begin{equation}\label{eq:rule}
  \text{reject } H_0 \iff \mathsf T_n>\widehat q_{1-\alpha}.
\end{equation}

\paragraph{Monte Carlo computation.} The conditional quantile
$\widehat q_{1-\alpha}$ has no closed form, and it is computed by simulation, the
data being kept fixed. For $b=1,\dots,N$, draw a new vector of multipliers
$(\zeta_1^{(b)},\dots,\zeta_n^{(b)})$, independent of the data and of the other
vectors, and compute the corresponding bootstrap statistic
$\mathsf T_n^{\zeta(b)}$ by~\eqref{eq:bootstrap}. The plug-in quantities
$\that,\widehat F_1,\widehat F_2,\widehat D{}'_n,\widehat\Pi$ (and $\widehat D_n,\widehat J_n$ for the full
version of $\widehat\Phi_n$) are computed once: they depend only on the empirical
cdf $\Ln$ and on the known source $F_0$, through the coefficients $B_n,C_n$ and
the estimate $\that$, and are the same for all $b$. Only the multipliers change
from one replication to the next; in particular the weight is never
re-estimated. Then $\widehat q_{1-\alpha}$ is approximated by the
empirical $(1-\alpha)$-quantile of $\mathsf T_n^{\zeta(1)},\dots,\mathsf T_n^{\zeta(N)}$, and the
associated $p$-value is
\begin{equation}\label{eq:pvalue}
  \widehat\pi_N=\frac{1+\#\{b\le N:\ \mathsf T_n^{\zeta(b)}\ge\mathsf T_n\}}{N+1}.
\end{equation}
These $N$ bootstrap replications are computed from the same data set; they are
not new samples. The resulting Monte Carlo error depends only on $N$, not on $n$,
and vanishes as $N\to\infty$; the smallest attainable $p$-value is $1/(N+1)$. Each
replication costs one product of an $n\times n$ matrix with a vector when
$G=\Ln$, so that $N$ in the thousands is affordable. The values of $N$ used in
Sections~\ref{sec:simu}--\ref{sec:real} are indicated there.

\begin{theorem}\label{thm:bootstrap}
\begin{enumerate}[label=(\roman*),leftmargin=*]
\item Under $H_0$ and \ref{A:theta}--\ref{A:nondeg}, the conditional law of
the bootstrap noise $\nG{\mathbb W_n^\zeta}^2$ consistently estimates the
reference law~\eqref{eq:reference}:
$\sup_{x\ge0}\bigl|\PP(\mathsf T^\zeta_n\le x\mid\bm X_1,\dots,\bm X_n)-\PP(\mathsf T_\infty\le x)\bigr|\to0$
in probability. Consequently $\widehat q_{1-\alpha}\to q_{1-\alpha}$ in probability, where
$q_{1-\alpha}$ is the $(1-\alpha)$-quantile of $\mathsf T_\infty$, and
$\PP(\mathsf T_n>\widehat q_{1-\alpha})\to\alpha$.
\item Under case~(b) of Theorem~\ref{thm:unified} with $d_{\min}>0$, the conditional
law of $\nG{\mathbb W_n^\zeta}^2$ converges weakly, in probability, to the
reference law $\mathsf T^c_\infty=\nG{\Phi_c(\G)}^2$ of~\eqref{eq:reference}. In
particular $\mathsf T_n^\zeta=O_P(1)$ and $\widehat q_{1-\alpha}=O_P(1)$, whereas
$\mathsf T_n=n\,d_{\min}+O_P(\sqrt n)$, so $\PP(\mathsf T_n>\widehat q_{1-\alpha})\to1$.
For the null form of the bootstrap, $\widehat q_{1-\alpha}=O_P(1)$ and
$\PP(\mathsf T_n>\widehat q_{1-\alpha})\to1$ under any alternative with $d_{\min}>0$.
\end{enumerate}
\end{theorem}

\subsection{Data-driven integrating measure}

The natural Cram\'er--von Mises choice is $G=\Ln$, in which case all integrals
reduce to averages over the sample points and the whole procedure costs
$O(n^2)$ operations. The following result shows that this choice leads to the
same asymptotic theory, with $G$ replaced by $L$.

\begin{proposition}\label{prop:empiricalG}
Assume that $F_0$ and $L$ are continuous, and that \ref{A:theta}--\ref{A:nondeg} hold with $G=L$.
Then Theorem~\ref{thm:consistency}(i)--(ii), Corollary~\ref{thm:null} and Theorem~\ref{thm:bootstrap}
remain valid when $G$ is replaced by $\Ln$ in~\eqref{eq:dn} and~\eqref{eq:bootstrap},
with $G=L$ in the limiting quantities.
\end{proposition}

\section{Simulation study}\label{sec:simu}

\input{simulation_section}

\input{realdata_section}

\section{Discussion}\label{sec:discussion}

The quadratic structure of Lemma~\ref{lem:algebra} is not specific to the
bivariate case. For a $d$-variate contamination model and the hypothesis of mutual
independence, $F(\cdot\,;p)-\bigotimes_{j=1}^dF_j(\cdot\,;p)$ is a polynomial of
degree $d$ in $t=1/p$, and the arguments of the Appendix carry over with
$\Psi$ replaced by the derivative of this polynomial at $t^\ast$; the
identifiability condition becomes a non-vanishing condition for a polynomial
curve in $L^2(G)$. Other natural extensions include testing a parametric
copula for the unknown component (the product copula being the simplest
case), a known-up-to-a-parameter known source $F_0$ in the spirit of
\citet{MPSV2024}, and $K$-sample versions of the problem. Finally,
Appendix~\ref{app:example} and Remark~\ref{rem:local} show that the
alternatives that cannot be detected are those in which the unknown component
contains a fraction of the known source, i.e.\ a one-parameter curve of
dependence functions in $\operatorname{span}\{K_0,\Delta\}$ whose tangent at the
null is $\Psi$. The choice of $G$ could be optimised to increase power in
directions of practical interest.

\appendix
\section{Proofs}\label{app:proofs}

\subsection{Key quantities and two auxiliary lemmas}\label{app:prelim}

All proofs are organised around a small number of quantities, collected in
Table~\ref{tab:key}. Every approximation below is obtained from an
\emph{exact} identity, and the remainder is bounded explicitly in terms of
the single random variable
\[
  \kappa_n:=\nI{\G_n}=\sqrt n\,\nI{\Ln-L},
\]
which is $O_P(1)$ because the class of lower-left quadrants is $L$-Donsker
for every $L$ \citep[Ch.~2.5]{vdVW1996}, and which satisfies
$\kappa_n=O_{a.s.}(\sqrt{\log\log n})$ by the law of the iterated logarithm
\citep{Kiefer1961,Alexander1984}. Since the marginals of $\Ln$ are the
empirical marginal cdfs, $\lVert\G_n^{(j)}\rVert_\infty\le\kappa_n$.

\begin{table}[h]
\centering\small
\caption{Key quantities used in the proofs.}\label{tab:key}
\begin{tabular}{p{0.32\textwidth}p{0.6\textwidth}}
\toprule
Quantity & Role \\
\midrule
$E_n=\Ln-L$, $\G_n=\sqrt nE_n$ & empirical error and empirical process \\
$\kappa_n=\nI{\G_n}$ & size of $\G_n$; all remainders are polynomials in $\kappa_n$ \\
$t=1/p$, $t^c=1/p^c$, $T=[t_l,t_u]$ &
  inverse weight, in which $D$ is a quadratic polynomial; $t^c=t^\ast$ under $H_0$ \\
$\widehat h_n=\sqrt n(\that-t^c)$ & rescaled estimation error \\
$S_n=\sqrt n\,\delta_n'(t^c)$, $J_n$, $J^c=\delta''(t^c)$ & score, empirical and population curvature \\
$\ell_c$, $\Phi_c$ & linear maps giving the limits of the weight and of the discrepancy \\
$\mathbb W_n=\sqrt n\bigl(D(\cdot\,;\Ln,\ph)-D^c\bigr)$ & estimated discrepancy, centred and rescaled \\
$\xi_n=\G_n-\G_n^{(1)}\otimes F_2-F_1\otimes\G_n^{(2)}$ &
  BKR process of the unobserved component; $\nI{\xi_n}\le3\kappa_n$ \\
$\Psi=K_0+\Delta$, $\Dd=-p^\ast\Psi$ &
  direction lost by estimating $p$; $\nI\Psi\le2$, $\nG\Psi\ge c_G$ \\
$\Pi$ & $L^2(G)$-projection onto $\operatorname{span}\{\Psi\}$ \\
$r_T=\min(t^\ast-t_l,\,t_u-t^\ast)$ & distance to the boundary, $>0$ under \ref{A:theta} \\
\bottomrule
\end{tabular}
\end{table}

For a bounded function $w\in\ell^\infty(\Rb^2)$ define the bounded linear operators
\begin{equation}\label{eq:operators}
  \beta(w)=w-w^{(1)}\otimes F_{0,2}-F_{0,1}\otimes w^{(2)},\qquad
  \gamma^\ast(w)=w^{(1)}\otimes A_2+A_1\otimes w^{(2)} ,
\end{equation}
for which $\nI{\beta(w)}\le3\nI w$ and $\nI{\gamma^\ast(w)}\le2\nI w$, since
$|A_j|\le1$. The operator $\beta$ involves only the known marginals
$F_{0,1},F_{0,2}$ and is therefore free from $p^\ast$. In contrast, $\gamma^\ast$
involves $A_j=L_j-F_{0,j}=p^\ast(F_j-F_{0,j})$. It depends on the true model, as
the star indicates (the same convention as for $B^\ast$ and $C^\ast$), and can be written
\[
  \gamma^\ast(w)=p^\ast\bigl[w^{(1)}\otimes(F_2-F_{0,2})+(F_1-F_{0,1})\otimes w^{(2)}\bigr].
\]
We also use $|K_0|\le1$, $|B^\ast|\le3$, $|C^\ast|\le1$, and the same bounds
for $B_n,C_n$, so that, as $t_u\ge1$,
\begin{equation}\label{eq:Dbound}
  \sup_{t\in T}\bigl(\nI{D(\cdot\,;1/t)}\vee\nI{D_n(\cdot\,;t)}\bigr)\le 1+3t_u+t_u^2\le5t_u^2 .
\end{equation}
Throughout we abbreviate $D_n(t)=D_n(\cdot\,;t)$ and $D(t)=D(\cdot\,;L,1/t)$.
Generic bivariate functions, such as the arguments of $\beta$ and
$\gamma^\ast$, are denoted by $w$.

\begin{lemma}[Exact error decomposition]\label{lem:decomp}
\begin{enumerate}[label=(\roman*),leftmargin=*]
\item For every $t>0$, under $H_0$ or $H_1$,
\begin{equation}\label{eq:Dn-D}
  D_n(t)-D(t)=t\,\beta(E_n)-t^2\gamma^\ast(E_n)-t^2E_n^{(1)}\otimes E_n^{(2)} .
\end{equation}
\item Under $H_0$, for every $w\in\ell^\infty(\Rb^2)$,
$t^\ast\beta(w)-t^{\ast2}\gamma^\ast(w)=t^\ast\bigl(w-w^{(1)}\otimes F_2-F_1\otimes w^{(2)}\bigr)$.
In particular
$\sqrt n\,D_n(t^\ast)=t^\ast\xi_n-t^{\ast2}n^{-1/2}\G_n^{(1)}\otimes\G_n^{(2)}$.
\end{enumerate}
\end{lemma}

\begin{proof}
(i) By~\eqref{eq:poly} and~\eqref{eq:Dn}, $D_n(t)-D(t)=t(B_n-B^\ast)-t^2(C_n-C^\ast)$.
Since $A_n-A=E_n$ and $A_{n,j}-A_j=E_n^{(j)}$, the two coefficient errors are
\[
  B_n-B^\ast=\underbrace{E_n-E_n^{(1)}\otimes F_{0,2}-F_{0,1}\otimes E_n^{(2)}}_{=\,\beta(E_n)},
  \qquad
  C_n-C^\ast=\underbrace{E_n^{(1)}\otimes A_2+A_1\otimes E_n^{(2)}}_{=\,\gamma^\ast(E_n)}
  +E_n^{(1)}\otimes E_n^{(2)},
\]
the second one following from expanding
$(A_1+E_n^{(1)})\otimes(A_2+E_n^{(2)})-A_1\otimes A_2$.

(ii) Under $H_0$, $t^\ast A_j=F_j-F_{0,j}$, so
$t^\ast\beta(w)-t^{\ast2}\gamma^\ast(w)
=t^\ast\bigl(w-w^{(1)}\otimes(F_{0,2}+t^\ast A_2)-(F_{0,1}+t^\ast A_1)\otimes w^{(2)}\bigr)$,
which is the claim. Multiplying~\eqref{eq:Dn-D} by $\sqrt n$ at $t=t^\ast$, where
$D(t^\ast)=0$, and taking $w=\G_n$ makes the BKR process of the unobserved
component appear:
\[
  \sqrt n\,D_n(t^\ast)=\underbrace{t^\ast\beta(\G_n)-t^{\ast2}\gamma^\ast(\G_n)}_{=\,t^\ast\xi_n}
  -\;t^{\ast2}n^{-1/2}\G_n^{(1)}\otimes\G_n^{(2)} . \qedhere
\]
\end{proof}

\begin{lemma}[Perturbation of projected quadratic forms]\label{lem:proj}
Let $\Pi_\psi$ denote the $L^2(G)$-projection onto $\operatorname{span}\{\psi\}$.
For $\psi,\phi\in L^2(G)\setminus\{0\}$,
$\lVert\Pi_\psi-\Pi_\phi\rVert_{\mathrm{op}}\le4\nG{\psi-\phi}/\nG\phi$.
For orthogonal projections $P,P'$ and $z,z'\in L^2(G)$,
\[
  \bigl|\nG{Pz}^2-\nG{P'z'}^2\bigr|\le(\nG z+\nG{z'})\bigl(\lVert P-P'\rVert_{\mathrm{op}}\nG z+\nG{z-z'}\bigr).
\]
\end{lemma}

\begin{proof}
For unit vectors $e,f$, $\lVert e\langle e,\cdot\rangle-f\langle f,\cdot\rangle\rVert_{\mathrm{op}}\le2\nG{e-f}$, and
$\nG{\psi/\nG\psi-\phi/\nG\phi}\le\nG{\psi-\phi}/\nG\phi+|\nG\phi-\nG\psi|/\nG\phi\le2\nG{\psi-\phi}/\nG\phi$.
For the second claim, $|\nG{Pz}-\nG{P'z'}|\le\nG{(P-P')z}+\nG{P'(z-z')}$ and
$\nG{Pz}+\nG{P'z'}\le\nG z+\nG{z'}$.
\end{proof}

\subsection{Proof of Lemma~\ref{lem:algebra}}

(i) By~\eqref{eq:inverted-explicit},
\[
  p^2F_{1\otimes2}(\cdot\,;p)=(p-p^\ast)^2F_{0,1}\otimes F_{0,2}
  +p^\ast(p-p^\ast)\bigl(F_{0,1}\otimes F_2+F_1\otimes F_{0,2}\bigr)+p^{\ast2}F_1\otimes F_2,
\]
while $p^2F(\cdot\,;p)=p(p-p^\ast)F_0+pp^\ast F$. Subtracting gives~\eqref{eq:expansion}.

(iii) By~\eqref{eq:family1}--\eqref{eq:family2}, $pF(\cdot\,;p)=A+pF_0$ and
$pF_j(\cdot\,;p)=A_j+pF_{0,j}$. Hence
\[
\begin{aligned}
  p^2D(\cdot\,;L,p)&=p(A+pF_0)-(A_1+pF_{0,1})\otimes(A_2+pF_{0,2})\\
  &=p^2\underbrace{\bigl(F_0-F_{0,1}\otimes F_{0,2}\bigr)}_{=\,K_0}
   +p\,\underbrace{\bigl(A-A_1\otimes F_{0,2}-F_{0,1}\otimes A_2\bigr)}_{=\,B^\ast}
   -\underbrace{A_1\otimes A_2}_{=\,C^\ast},
\end{aligned}
\]
which is~\eqref{eq:poly} after division by $p^2$. Next, \eqref{eq:model} and
its marginal version give $A=p^\ast(F-F_0)$ and $A_j=p^\ast(F_j-F_{0,j})$. Hence
$C^\ast=A_1\otimes A_2=p^{\ast2}\Delta$ and
\[
  B^\ast=p^\ast\bigl[F-F_0-(F_1-F_{0,1})\otimes F_{0,2}-F_{0,1}\otimes(F_2-F_{0,2})\bigr]
   =p^\ast\,(K+\Delta-K_0),
\]
as is checked by expanding $K$, $\Delta$ and $K_0$. This is~\eqref{eq:BC}.

(ii) We derive~\eqref{eq:structure} from~\eqref{eq:expansion} in three
steps. Write $a=p-p^\ast$, so that~\eqref{eq:expansion} reads
\[
  p^2D=pp^\ast F-p^{\ast2}F_1\otimes F_2
  +a\bigl[pF_0-p^\ast\bigl(F_{0,1}\otimes F_2+F_1\otimes F_{0,2}\bigr)\bigr]
  -a^2F_{0,1}\otimes F_{0,2}.
\]
\emph{Step 1: extract $K$ from the term in $F$.} Since $F=K+F_1\otimes F_2$,
$pp^\ast F=pp^\ast K+pp^\ast F_1\otimes F_2$, and the coefficient of
$F_1\otimes F_2$ becomes $pp^\ast-p^{\ast2}=p^\ast(p-p^\ast)=p^\ast a$, recalling that $a=p-p^\ast$:
\[
  p^2D=pp^\ast K+p^\ast a\,F_1\otimes F_2+pa\,F_0
  -p^\ast a\bigl(F_{0,1}\otimes F_2+F_1\otimes F_{0,2}\bigr)-a^2F_{0,1}\otimes F_{0,2}.
\]
\emph{Step 2: split the coefficient of $F_{0,1}\otimes F_{0,2}$.} Writing
$-a^2=-a(p-p^\ast)=-pa+p^\ast a$ and assigning $-pa$ to the group of $F_0$
and $+p^\ast a$ to the group of $F_1\otimes F_2$,
\[
  p^2D=pp^\ast K
  +pa\,\underbrace{\bigl[F_0-F_{0,1}\otimes F_{0,2}\bigr]}_{=\,K_0}
  +p^\ast a\,\underbrace{\bigl[F_1\otimes F_2-F_{0,1}\otimes F_2-F_1\otimes F_{0,2}+F_{0,1}\otimes F_{0,2}\bigr]}_{=\,(F_1-F_{0,1})\otimes(F_2-F_{0,2})\,=\,\Delta}.
\]
This is the bridge identity~\eqref{eq:bridge}:
$p^2D=p(p-p^\ast)K_0+p^\ast(p-p^\ast)\Delta+pp^\ast K$.

\emph{Step 3: divide by $p^2$.} With $s=p^\ast/p$, $p(p-p^\ast)/p^2=1-s$,
$p^\ast(p-p^\ast)/p^2=s(1-s)$ and $pp^\ast/p^2=s$, which
gives~\eqref{eq:structure}.

The splitting in Step~2 explains why $K_0$ does not appear as a block
in~\eqref{eq:expansion}: the single coefficient $-(p-p^\ast)^2$ of
$F_{0,1}\otimes F_{0,2}$ is the sum of the contribution $-p(p-p^\ast)$ of
$K_0$ and the contribution $+p^\ast(p-p^\ast)$ of $\Delta$.

Alternatively, substituting~\eqref{eq:BC} into~\eqref{eq:poly} with
$s=p^\ast t$ gives $D=K_0+s(K+\Delta-K_0)-s^2\Delta$, which is~\eqref{eq:structure}
again.

Under $H_0$, $K=0$ and $D=(1-s)(K_0+s\Delta)$. Writing $1-s=-p^\ast(t-t^\ast)$ and
$K_0+s\Delta=(K_0+\Delta)+p^\ast(t-t^\ast)\Delta$ gives the last expression
in~\eqref{eq:structureH0}. \qed

\subsection{Proof of Proposition~\ref{prop:ident}}

(i) Expanding $\Delta$,
\[
  K_0+s\Delta=s\,F_1\otimes F_2+F_0-s\bigl(F_{0,1}\otimes F_2+F_1\otimes F_{0,2}\bigr)+(s-1)\,F_{0,1}\otimes F_{0,2}.
\]
The coefficient of $F_0$ equals $1$, so under~(I) the function $K_0+s\Delta$
does not vanish $G$-a.e., for any $s\in\R$. The map
$s\mapsto\nG{K_0+s\Delta}$ is continuous and positive on the compact set $S$,
hence bounded below by a positive constant, which is~\ref{A:ident}.

(ii) By~\eqref{eq:structureH0},
$\nG{D(\cdot\,;L,p)}=|1-s|\,\nG{K_0+s\Delta}\ge p^\ast|t-t^\ast|\,c_G$
for $s=p^\ast t\in S$.

(iii) $d(p)=0$ iff $F(\cdot\,;p)=F_{1\otimes2}(\cdot\,;p)$ $G$-a.e. Since $d$ is
continuous (it is a polynomial in $1/p$) and $\Theta$ is compact,
$d_{\min}>0$ iff $d(p)>0$ for all $p\in\Theta$. Under~(II), the coefficient
of $F$ in~\eqref{eq:expansion} is $pp^\ast\neq0$, so $D(\cdot\,;L,p)$ does
not vanish $G$-a.e.\ for any $p\in\Theta$, i.e.\ $d(p)>0$ for all $p$.
Finally, if $d(p)=0$ for some $p$ then, by~\eqref{eq:structure},
$sK=-(1-s)K_0-s(1-s)\Delta$ in $L^2(G)$ with $s>0$, i.e.\
$K\in\operatorname{span}\{K_0,\Delta\}$. \qed

\subsection{Proof of Lemma~\ref{lem:contrast}}

(i) The derivatives follow from~\eqref{eq:poly}, since $K_0,B^\ast,C^\ast$ do not
depend on $p$. Because $|K_0|\le1$, $|B^\ast|\le3$ and $|C^\ast|\le1$ (see
\eqref{eq:Dbound}), for $p\in\Theta$ we have
$\nI{D(\cdot\,;L,p)}\le1+3/p+1/p^2\le5\theta_l^{-2}$ and
$\nI{\dot D(\cdot\,;L,p)}\le3/p^2+2/p^3\le5\theta_l^{-3}$. Hence, by
dominated differentiation,
$|\dot d(p)|=2|\ip{D}{\dot D}|\le50\,\theta_l^{-5}$ on $\Theta$, and the mean
value theorem gives the Lipschitz bound.

(ii) By~\eqref{eq:structureH0}, $D(\cdot\,;L,p^\ast)=0$ as a function, not only
in $L^2(G)$. Hence $d(p^\ast)=0$ and
$\dot d(p^\ast)=2\ip{D(\cdot\,;L,p^\ast)}{\dot D(\cdot\,;L,p^\ast)}=0$. By dominated
differentiation, $\ddot d=2\int(\dot D^2+D\ddot D)\,dG$, and the second term
vanishes at $p^\ast$ for the same reason. Hence $\ddot d(p^\ast)$ is the squared
norm $2\nG{\dot D(\cdot\,;L,p^\ast)}^2$. To compute it, insert the
coefficients~\eqref{eq:BC} of Lemma~\ref{lem:algebra}(iii),
$B^\ast=p^\ast(K+\Delta-K_0)$ and $C^\ast=p^{\ast2}\Delta$, into the formula of
part~(i). This is valid under $H_0$ or $H_1$ and gives, at $p=p^\ast$,
\begin{equation}\label{eq:Ddot-general}
  \dot D(\cdot\,;L,p^\ast)=-\frac{B^\ast}{p^{\ast2}}+\frac{2C^\ast}{p^{\ast3}}
  =-\frac{K+\Delta-K_0}{p^\ast}+\frac{2\Delta}{p^\ast}
  =\frac{K_0+\Delta-K}{p^\ast}.
\end{equation}
The dependence function $K$ of the unknown component thus enters
$\dot D(\cdot\,;L,p^\ast)$ with coefficient $-1/p^\ast$. Under $H_0$, $K=0$,
and~\eqref{eq:Ddot-general} reduces to
\[
  \dot D(\cdot\,;L,p^\ast)=\frac{K_0+\Delta}{p^\ast}.
\]
Multiplying by $p^\ast$ gives the second equality in~\eqref{eq:Ddot-link}.
The first one is the definition $\Dd=-p^\ast(K_0+\Delta)$ in~\eqref{eq:structureH0}.
It is consistent with the interpretation of $\Dd$ as the derivative of
$t\mapsto D(\cdot\,;L,1/t)$ at $t^\ast$. Indeed,
\eqref{eq:structureH0} gives
$\partial_tD(\cdot\,;L,1/t)|_{t=t^\ast}=\Dd$. Also, for every $p$, the chain rule with
$p=1/t$ and $dp/dt=-p^2$ gives
\begin{equation}\label{eq:chain}
  \partial_t\,D(\cdot\,;L,1/t)=-p^2\,\dot D(\cdot\,;L,p),
\end{equation}
so that $\Dd=-p^{\ast2}\dot D(\cdot\,;L,p^\ast)$, which is again~\eqref{eq:Ddot-link}.
This gives the equalities in~\eqref{eq:ddot}. The squared norm is bounded
away from zero because $\nG{K_0+\Delta}\ge\inf_{s\in S}\nG{K_0+s\Delta}=c_G>0$,
using $1\in S$ under~\ref{A:theta} and~\ref{A:ident}. If instead $K_0=-\Delta$,
then $\dot D(\cdot\,;L,p^\ast)=0$ and $\ddot d(p^\ast)=0$.

(iii) Since $p^c$ is an interior minimiser of the smooth function $d$,
$\dot d(p^c)=0$ and $\ddot d(p^c)\ge0$, with $\ddot d$ given by the formula of
part~(ii). This is~\eqref{eq:ddotc}. Condition~(a) follows from the
Cauchy--Schwarz inequality,
$|\ip{D}{\ddot D}|\le\nG{D}\,\nG{\ddot D}=\sqrt{d_{\min}}\,\nG{\ddot D}$ at $p^c$.
For~(b), set $\delta(t)=d(1/t)$, a quartic polynomial in $t$ with leading
coefficient $\nG{C^\ast}^2$ by~\eqref{eq:poly}, and let $t^c=1/p^c$. Since
$\dot d(p)=-t^2\delta'(t)$, we have $\ddot d(p)=t^4\delta''(t)+2t^3\delta'(t)$, so at the
critical point $\ddot d(p^c)=(t^c)^4\delta''(t^c)$. If $\delta''(t^c)=0$, the Taylor
expansion of the quartic at its local minimum $t^c$ reads
$\delta(t)=d_{\min}+a_3(t-t^c)^3+a_4(t-t^c)^4$ with $a_4=\nG{C^\ast}^2$, and a local
minimum forces $a_3=0$. Conversely, this form gives $\delta''(t^c)=0$. \qed

\subsection{Proof of Theorem~\ref{thm:consistency}}

Recall from Appendix~\ref{app:prelim} that $D_n(t)=D(\cdot\,;\Ln,1/t)=K_0+tB_n-t^2C_n$
and $D(t)=D(\cdot\,;L,1/t)=K_0+tB^\ast-t^2C^\ast$ denote the empirical and
population discrepancies, seen as functions of $t=1/p$.

\emph{Uniform bounds.} Since $p\mapsto1/p$ maps $\Theta$ onto $T$ and
$D(\cdot\,;\Ln,p)=D_n(1/p)$, the suprema over $p\in\Theta$
in~\eqref{eq:unifrate} equal the corresponding suprema over $t\in T$. By
Lemma~\ref{lem:decomp} and $\nI{E_n^{(1)}\otimes E_n^{(2)}}\le\nI{E_n}$,
\begin{equation}\label{eq:Dn-D-sup}
  \sup_{t\in T}\nI{D_n(t)-D(t)}\le(3t_u+3t_u^2)\nI{E_n}\le6t_u^2\,\frac{\kappa_n}{\sqrt n}.
\end{equation}
By~\eqref{eq:Dbound},
$|\delta_n(t)-\delta(t)|=|\ip{D_n(t)-D(t)}{D_n(t)+D(t)}|\le10t_u^2\nI{D_n(t)-D(t)}$, so
\begin{equation}\label{eq:dn-d-sup}
  \sup_{t\in T}|\delta_n(t)-\delta(t)|\le60\,t_u^4\,\frac{\kappa_n}{\sqrt n}.
\end{equation}
Both bounds are deterministic multiples of $\kappa_n/\sqrt n=O_{a.s.}(\sqrt{\log\log n/n})$,
which proves~\eqref{eq:unifrate}.

(i) By definition $\that$ minimises $\delta_n(t)=\nG{D_n(t)}^2$ over $T$. Since
$x\mapsto\sqrt x$ is increasing, it equivalently minimises the norm $\nG{D_n(t)}$.
Using this and $D(t^\ast)=0$,
\[
  \nG{D(\that)}\le\nG{D_n(\that)}+\nI{D_n(\that)-D(\that)}
  \le\nG{D_n(t^\ast)}+6t_u^2\frac{\kappa_n}{\sqrt n}\le12t_u^2\frac{\kappa_n}{\sqrt n}.
\]
The identification bound~\eqref{eq:lowerbound} converts this into a bound on
the parameter:
\begin{equation}\label{eq:hhat-bound}
  |\widehat h_n|=\sqrt n\,|\that-t^\ast|\le\frac{12\,t_u^2}{c_G\,p^\ast}\,\kappa_n .
\end{equation}
Finally $|\ph-p^\ast|=\ph p^\ast|\that-t^\ast|\le|\that-t^\ast|$.

(ii) Let $u_n=\sup_T|\delta_n-\delta|$. Both minima are attained, since $\delta_n$
and $\delta$ are continuous on the compact set $T$. For every $t\in T$,
$\delta_n(t)\ge \delta(t)-u_n\ge\min_T\delta-u_n$, so
$\min_T\delta_n\ge\min_T\delta-u_n$. Conversely, if $t^c$ minimises $\delta$, then
$\min_T\delta_n\le \delta_n(t^c)\le \delta(t^c)+u_n=\min_T\delta+u_n$. Hence
\[
  \bigl|\min_T\delta_n-\min_T\delta\bigr|\le\sup_T|\delta_n-\delta|\to0\quad\text{a.s.},
\]
so
$\delta_n(\that)\to d_{\min}$. Moreover, since $\delta_n(\that)=\min_T\delta_n\ge\min_T\delta-u_n=d_{\min}-u_n$,
the bound~\eqref{eq:dn-d-sup} gives the explicit lower bound
\[
  \mathsf T_n=n\,\delta_n(\that)\ge n\,d_{\min}-60\,t_u^4\sqrt n\,\kappa_n ,
\]
which tends to $+\infty$
a.s.\ when $d_{\min}>0$. Assume now that $t^c=1/p^c$ is the unique minimiser of $\delta$ on $T$. We show
that $\that\to t^c$ on the event
\[
  \Omega_0=\Bigl\{\lim_{n\to\infty}u_n=0\Bigr\},
\]
which has probability one by~\eqref{eq:dn-d-sup}, since
$\kappa_n/\sqrt n=O_{a.s.}(\sqrt{\log\log n/n})$. The argument is deterministic once
$\omega\in\Omega_0$ is fixed, and applies to any choice of minimiser
$\that\in\argmin_T\delta_n$.

\emph{Separation.} Fix $\varepsilon>0$ and let $T_\varepsilon=\{t\in T:|t-t^c|\ge\varepsilon\}$.
If $T_\varepsilon=\emptyset$, there is nothing to prove. Otherwise $T_\varepsilon$ is compact,
as a closed subset of the compact interval $T$, and $\delta$ is continuous, being a
polynomial in $t$. So $\delta$ attains its minimum on $T_\varepsilon$ at some
$t_\varepsilon\neq t^c$, and uniqueness of the minimiser gives
\[
  \eta_\varepsilon:=\min_{t\in T_\varepsilon}\delta(t)-\delta(t^c)=\delta(t_\varepsilon)-\delta(t^c)>0 .
\]

\emph{Near-minimality of $\that$ for $\delta$.} For every $n$, using successively
$|\delta-\delta_n|\le u_n$ on $T$, the minimality of $\that$ for $\delta_n$, and again
$|\delta_n-\delta|\le u_n$,
\begin{equation}\label{eq:near-min}
  \delta(\that)\le \delta_n(\that)+u_n\le \delta_n(t^c)+u_n\le \delta(t^c)+2u_n .
\end{equation}

\emph{Conclusion.} Let $\omega\in\Omega_0$. There exists $N=N(\omega,\varepsilon)$
such that $2u_n(\omega)<\eta_\varepsilon$ for all $n\ge N$. For such $n$, if
$\that(\omega)\in T_\varepsilon$, then $\delta(\that)\ge \delta(t^c)+\eta_\varepsilon>\delta(t^c)+2u_n$,
which contradicts~\eqref{eq:near-min}. Hence $|\that(\omega)-t^c|<\varepsilon$ for all
$n\ge N$. Since $\varepsilon>0$ is arbitrary and $\Omega_0$ does not depend on
$\varepsilon$, $\that\to t^c$ on $\Omega_0$, i.e.\ almost surely. Finally,
$p\mapsto1/p$ is continuous, so $\ph=1/\that\to1/t^c=p^c$ almost surely. \qed

\subsection{Proof of Lemma~\ref{lem:linear}}\label{app:linear}

(i) The map $(\bx,\bv)\mapsto\1\{\bx\le\bv\}$ is Borel on $\Rb^2\times\Rb^2$, so
$f_\tau$ is Borel by Fubini's theorem, and $|f_\tau(\bx)|\le|\mu_\tau|(\Rb^2)\le m_\mu$.
Since $\G_n(\bv)=n^{-1/2}\sum_i(\1\{\bm X_i\le\bv\}-L(\bv))$ is a finite sum,
integrating term by term gives
$\int\G_n\,d\mu_\tau=n^{-1/2}\sum_i\bigl(f_\tau(\bm X_i)-\int L\,d\mu_\tau\bigr)$.
Moreover $\int L\,d\mu_\tau=\int\EE\1\{\bm X\le\bv\}\,d\mu_\tau(\bv)=\EE f_\tau(\bm X)$,
by Fubini's theorem for the bounded integrand and the finite measure
$\PP\otimes|\mu_\tau|$.

(ii) Let $\mathcal B(\Rb^2)$ be the closed subspace of $\ell^\infty(\Rb^2)$
consisting of bounded Borel functions. It contains the paths of $\G_n$, and
those of a version of $\G$, which can be chosen with $\rho_L$-continuous,
hence Borel, paths \citep[Sec.~2.1]{vdVW1996}. On $\mathcal B(\Rb^2)$ the map
$\mathcal M$ is linear with $\sup_\tau|\mathcal M(w)(\tau)|\le m_\mu\nI w$, hence
continuous into $\ell^\infty(\mathcal T)$. Since $\G_n\dto\G$ in
$\ell^\infty(\Rb^2)$, the continuous mapping theorem, in the version for maps
defined on a subset containing the ranges of $\G_n$ and $\G$
\citep[Thm.~1.3.6]{vdVW1996}, gives $\mathcal M(\G_n)\dto\mathcal M(\G)$.

(iii) Fix $\tau_1,\dots,\tau_k\in\mathcal T$. By~(i), the vector
$(\mathcal M(\G_n)(\tau_j))_{j\le k}$ is a normalised sum of i.i.d.\ centred
vectors $(f_{\tau_j}(\bm X_i)-\EE f_{\tau_j}(\bm X))_{j\le k}$, bounded by $2m_\mu$,
whose covariance matrix is $\Sigma_{jl}=\operatorname{Cov}(f_{\tau_j}(\bm X),f_{\tau_l}(\bm X))$
for every $n$. By the multivariate central limit theorem it converges in law to
$\mathcal N_k(0,\Sigma)$. By~(ii) and the continuity of the coordinate
projections of $\ell^\infty(\mathcal T)$, it also converges in law to
$(\mathcal M(\G)(\tau_j))_{j\le k}$. By uniqueness of the limit, the
finite-dimensional laws of $\mathcal M(\G)$ are $\mathcal N_k(0,\Sigma)$, which
proves the first equality in~\eqref{eq:lin-cov}. For the second, write
$f_\tau(\bm X)f_{\tau'}(\bm X)=\iint\1\{\bm X\le\bm a\}\1\{\bm X\le\bm b\}\,d\mu_\tau(\bm a)\,d\mu_{\tau'}(\bm b)$
and use $\1\{\bm X\le\bm a\}\1\{\bm X\le\bm b\}=\1\{\bm X\le\bm a\wedge\bm b\}$.
Taking expectations with Fubini's theorem gives
$\EE[f_\tau(\bm X)f_{\tau'}(\bm X)]=\iint L(\bm a\wedge\bm b)\,d\mu_\tau\,d\mu_{\tau'}$, and
$\EE f_\tau(\bm X)\,\EE f_{\tau'}(\bm X)=\iint L(\bm a)L(\bm b)\,d\mu_\tau\,d\mu_{\tau'}$
by~(i). \qed

\subsection{Proof of Theorem~\ref{thm:unified}}\label{app:unified}

Primes denote derivatives in $t$, $\delta_n(t)=d_n(1/t)=\nG{D_n(t)}^2$, and
throughout $\widehat h_n=\sqrt n(\that-t^c)$ and $\ell_n=\ell_c(\G_n)$.

\paragraph{Strategy: the classical Taylor expansion.} The whole argument
rests on one requirement: \emph{$\that$ must be an interior point of $T$}.
Indeed, $\that$ is defined as a minimiser of $\delta_n$ over the compact
interval $T=[t_l,t_u]$, and such a minimiser satisfies the first-order
condition
\begin{equation}\label{eq:foc}
  \delta_n'(\that)=0
\end{equation}
only when $\that\in(t_l,t_u)$. If $\that$ falls on an endpoint, then in
general $\delta_n'(\that)\neq0$: the contrast is still decreasing (or
increasing) when it reaches the boundary, and~\eqref{eq:foc}, together with
everything derived from it below, is false. Interiority is not automatic: it
holds only with probability tending to one, and is established in Step~1
below. Assuming it for now, a first-order Taylor expansion of $\delta_n'$
around $t^c$ (mean value theorem) gives, for some $\tilde t_n$ between $t^c$ and $\that$,
\[
  \underbrace{0=\delta_n'(\that)}_{\text{requires $\that$ interior}}
  =\delta_n'(t^c)+\delta_n''(\tilde t_n)\,(\that-t^c).
\]
Multiplying by $\sqrt n$ and solving for $\widehat h_n$,
\begin{equation}\label{eq:taylor-main}
  \widehat h_n=-\frac{S_n}{J_n},\qquad
  S_n:=\sqrt n\,\delta_n'(t^c)\ \ \text{(score)},\qquad
  J_n:=\delta_n''(\tilde t_n)\ \ \text{(curvature)}.
\end{equation}
The \emph{score} $S_n$ is the rescaled slope of the empirical contrast at the
pseudo-true value. Its population counterpart vanishes, $\delta'(t^c)=0$
by~\eqref{eq:orth}, so $S_n$ is pure sampling noise, of order one after
rescaling by $\sqrt n$. The \emph{curvature} $J_n$ measures how sharply the
contrast bends around its minimum; its population counterpart is
$J^c=\delta''(t^c)$.

\paragraph{What has to be proved.} Given~\eqref{eq:taylor-main}, the expansion
of the weight follows from three facts:
\begin{enumerate}[label=(P\arabic*),leftmargin=*]
\item \emph{Localisation:} with probability tending to one, $\that$ is an
  \emph{interior point} of $T$, so that~\eqref{eq:foc}
  and hence~\eqref{eq:taylor-main} are valid, and $\tilde t_n$ is within
  $O_P(n^{-1/2})$ of $t^c$, so that the curvature is evaluated where
  $\delta_n''$ is controlled;
\item \emph{Score:} $S_n=2\ell_n+O_P(n^{-1/2})$, a continuous linear functional of
  $\G_n$ up to a negligible term, which is the source of the Gaussian limit;
\item \emph{Curvature:} $J_n=J^c+O_P(n^{-1/2})$, with $J^c>0$. This is where
  identifiability~\ref{A:ident} (case~(a)) or the non-degeneracy
  $\ddot d(p^c)>0$ (case~(b)) enters.
\end{enumerate}
Then $\widehat h_n=-2\ell_n/J^c+O_P(n^{-1/2})$, which gives~\eqref{eq:unif-p}.
Because $D_n$ is a quadratic polynomial in $t$, the discrepancy at $\that$ is
then obtained \emph{exactly} from its value and slope at $t^c$, which
gives~\eqref{eq:unif-D}. Steps~1--3 below establish (P1)--(P3), Step~4 the
weight, Step~5 the discrepancy and Step~6 the joint limit.

\paragraph{Technical tools.} Since $D_n(t)=K_0+tB_n-t^2C_n$,
\begin{equation}\label{eq:dn-derivs}
\begin{gathered}
  D_n'(t)=B_n-2tC_n,\qquad D_n''(t)=-2C_n,\qquad D_n'''(t)=0,\\
  \delta_n'=2\ip{D_n}{D_n'},\qquad \delta_n''=2\nG{D_n'}^2+2\ip{D_n}{D_n''},
\end{gathered}
\end{equation}
where $D_n'''\equiv0$ because $D_n(t)$ is a polynomial of degree two in $t$.
Differentiating $\delta_n''$ once more,
\begin{equation}\label{eq:dn-third}
  \delta_n'''(t)
  =\underbrace{4\ip{D_n'(t)}{D_n''(t)}}_{\text{from }2\nG{D_n'}^2}
  +\underbrace{2\ip{D_n'(t)}{D_n''(t)}+2\ip{D_n(t)}{D_n'''(t)}}_{\text{from }2\ip{D_n}{D_n''}}
  =6\ip{D_n'(t)}{D_n''(t)},
\end{equation}
the last term vanishing since $D_n'''=0$. Inserting~\eqref{eq:dn-derivs} gives
the literal expression
\[
  \delta_n'''(t)=6\ip{B_n-2tC_n}{-2C_n}=-12\ip{B_n}{C_n}+24\,t\,\nG{C_n}^2 ,
\]
in agreement with $\delta_n(t)=\sum_{k=0}^4a_{n,k}t^k$ in~\eqref{eq:dn}, since
$6a_{n,3}=-12\ip{B_n}{C_n}$ and $24a_{n,4}=24\nG{C_n}^2$. On $T$,
$\nI{D_n'}\le3+2t_u\le5t_u$ and $\nI{D_n''}\le2$, hence
$|\delta_n'''|\le6\cdot5t_u\cdot2=60\,t_u$; the same bound holds for $\delta'''$.
Finally, by Lemma~\ref{lem:decomp}, for every $t\in T$,
\begin{equation}\label{eq:Zt}
\begin{aligned}
  \sqrt n\bigl(D_n(t)-D(t)\bigr)&=\mathcal L_t(\G_n)-\frac{t^2}{\sqrt n}\,\G_n^{(1)}\otimes\G_n^{(2)},
  &\qquad&\mathcal L_t(w)=t\beta(w)-t^2\gamma^\ast(w),\\
  \sqrt n\bigl(D_n'(t)-D'(t)\bigr)&=\mathcal L_t'(\G_n)-\frac{2t}{\sqrt n}\,\G_n^{(1)}\otimes\G_n^{(2)},
  &\qquad&\mathcal L_t'(w)=\beta(w)-2t\gamma^\ast(w),\\
  \sqrt n\bigl(D_n''-D''\bigr)&=-2\gamma^\ast(\G_n)-\frac{2}{\sqrt n}\,\G_n^{(1)}\otimes\G_n^{(2)},
\end{aligned}
\end{equation}
the second and third lines following from the first by differentiation in $t$;
$\mathcal L_{t^c}=\mathcal L_c$ and $\mathcal L_{t^c}'=\mathcal L_c'$. Since
$\nI{\beta(w)}\le3\nI w$ and $\nI{\gamma^\ast(w)}\le2\nI w$, the three
right-hand sides are bounded in sup-norm by $9t_u^2\kappa_n(1+\kappa_n)$,
uniformly in $t\in T$. Hence $\sup_T\nI{D_n^{(k)}-D^{(k)}}=O_P(n^{-1/2})$ for
$k=0,1,2$, and, with~\eqref{eq:Dbound}, $\sup_T|\delta_n''-\delta''|=O_P(n^{-1/2})$.

\emph{Step 1: localisation (P1).} \emph{Case~(a).} Here $t^c=t^\ast$. Let
$K=12\,t_u^2/(c_Gp^{\ast2})$ and $M_n=K(1+\kappa_n)$. The bound~\eqref{eq:hhat-bound}
of Theorem~\ref{thm:consistency}(i) gives $|\widehat h_n|\le M_n$, i.e.
\begin{equation}\label{eq:window}
  |\that-t^\ast|\le\frac{M_n}{\sqrt n}\qquad\text{for every sample}.
\end{equation}
We stress that~\eqref{eq:window} does \emph{not} rely on the first-order
condition: it was derived by comparing \emph{values} of the contrast (the
minimality $\nG{D_n(\that)}\le\nG{D_n(t^\ast)}$) together with the
identification bound~\eqref{eq:lowerbound}, and therefore holds whether $\that$
is interior or not. There is thus no circularity in using it to prove
interiority. Recall that $r_T=\min(t^\ast-t_l,\,t_u-t^\ast)>0$ is the
distance from $t^\ast$ to the boundary of $T$ (Assumption~\ref{A:theta}), and define the event
\begin{equation}\label{eq:Omega-int}
  \Omega_n^{\mathrm{int}}:=\bigl\{M_n<r_T\sqrt n\bigr\}.
\end{equation}
On $\Omega_n^{\mathrm{int}}$, \eqref{eq:window} gives
$|\that-t^\ast|\le M_n/\sqrt n<r_T$: the window around $t^\ast$ that traps
$\that$ is narrower than the distance to either endpoint, so
$t_l<\that<t_u$, i.e.\ $\that$ is an interior point of $T$ and~\eqref{eq:foc}
holds. Moreover $|\tilde t_n-t^\ast|\le|\that-t^\ast|\le M_n/\sqrt n$. Since $M_n=O_P(1)$
while $r_T\sqrt n\to\infty$, $\PP(\Omega_n^{\mathrm{int}})\to1$. In finite
samples, however, the complement may carry non-negligible probability when
$r_T$ is small relative to the window width $M_n/\sqrt n$, for instance when
$p^\ast$ and $c_G$ are small (so that $K$ is large); the boundary estimates $\ph=1$ ($\that=t_l$)
observed in Scenario~B with $p^\ast=0.3$, $n=250$ (Section~\ref{sec:simu}) are an
instance of this, and explain the size distortion reported there.

\emph{Case~(b).} Theorem~\ref{thm:consistency}(ii) gives $\that\to t^c$ almost
surely. Since $t^c$ is interior, $\that$ is interior for all $n$ large enough,
almost surely, and~\eqref{eq:foc} holds. The rate $|\that-t^c|=O_P(n^{-1/2})$ is
obtained in Step~3.

\emph{Step 2: the score (P2).} Since $\delta'(t^c)=0$ and $ab-cd=(a-c)b+c(b-d)$,
\[
  S_n=\sqrt n\bigl(\delta_n'-\delta'\bigr)(t^c)
  =2\bigl\langle\underbrace{\sqrt n(D_n-D)(t^c)}_{\mathcal L_c(\G_n)+O_P(n^{-1/2})},\,
   \underbrace{D_n'(t^c)}_{D'_c+O_P(n^{-1/2})}\bigr\rangle_G
  +2\bigl\langle D^c,\,\underbrace{\sqrt n(D_n'-D')(t^c)}_{\mathcal L_c'(\G_n)+O_P(n^{-1/2})}\bigr\rangle_G ,
\]
by~\eqref{eq:Zt}. Hence $S_n=2\ell_n+O_P(n^{-1/2})$ and, in particular,
$S_n=O_P(1)$. Under $H_0$ the second term vanishes because $D^c=0$; under
$H_1$ it accounts for the change of the tangent direction with the data.

\emph{Step 3: the curvature (P3).} \emph{Positivity.} In case~(a),
$J^c=2\nG\Dd^2=2p^{\ast2}\nG\Psi^2\ge2p^{\ast2}c_G^2>0$. In case~(b),
differentiating $\delta(t)=d(1/t)$ twice gives
$\delta''(t)=p^4\ddot d(p)+2p^3\dot d(p)$ with $p=1/t$, so
$J^c=p^{c4}\ddot d(p^c)>0$ because $\dot d(p^c)=0$.
\emph{Approximation.} For any $\bar t_n$ between $t^c$ and $\that$,
by the technical tools and the Lipschitz property of $\delta''$,
\begin{equation}\label{eq:curv}
  \bigl|\delta_n''(\bar t_n)-J^c\bigr|
  \le\underbrace{\sup_T|\delta_n''-\delta''|}_{\text{noise}}
  +\underbrace{60t_u|\that-t^c|}_{\text{drift within the window}}
  =O_P(n^{-1/2})+60t_u|\that-t^c| .
\end{equation}
In case~(a), $|\that-t^c|=O_P(n^{-1/2})$ by~\eqref{eq:window}. In case~(b),
$\that\to t^c$ almost surely gives first $J_n\to J^c>0$ almost surely, hence
$\widehat h_n=-S_n/J_n=O_P(1)$ by Step~2, i.e.\ $|\that-t^c|=O_P(n^{-1/2})$. In
both cases $\delta_n''(\bar t_n)=J^c+O_P(n^{-1/2})$; in particular
$J_n=J^c+O_P(n^{-1/2})$. The role of~\eqref{eq:curv} is to keep the
denominator of~\eqref{eq:taylor-main} away from zero: on the event
$\Omega_n^{\mathrm{curv}}=\{|J_n-J^c|\le J^c/2\}$, whose probability tends to
one, the noise uses at most half of the curvature, and $J_n\ge J^c/2>0$.

\emph{Step 4: the weight.} By Steps~1--3, on an event of probability tending to one,
\begin{equation}\label{eq:unif-t}
  \widehat h_n=-\frac{S_n}{J_n}=-\frac{2\ell_n+O_P(n^{-1/2})}{J^c+O_P(n^{-1/2})}
  =-\frac{2\ell_n}{J^c}+O_P(n^{-1/2}).
\end{equation}
The inversion $p=1/t$ is handled exactly:
$\ph-p^c=-(\that-t^c)/(\that\,t^c)$, and $t^c/\that-1=O_P(n^{-1/2})$, so
$\sqrt n(\ph-p^c)=-p^{c2}\widehat h_n+O_P(n^{-1/2})$, which is~\eqref{eq:unif-p}.

\emph{Step 5: the discrepancy.} Since $D_n$ is a quadratic polynomial in $t$
with $D_n''=-2C_n$, the exact Taylor formula
$D_n(\that)-D_n(t^c)=(\that-t^c)D_n'(t^c)-(\that-t^c)^2C_n$ holds. Hence, with
$D(\cdot\,;\Ln,\ph)=D_n(\that)$ and $D(\cdot\,;L,p^c)=D(t^c)=D^c$,
\[
\begin{aligned}
  \mathbb W_n&=\underbrace{\sqrt n\bigl(D_n(\that)-D_n(t^c)\bigr)}_{\text{re-fitting the weight}}
  +\underbrace{\sqrt n\bigl(D_n(t^c)-D(t^c)\bigr)}_{\text{weight fixed}}\\
  &=\widehat h_n\,D_n'(t^c)-\frac{\widehat h_n^2}{\sqrt n}\,C_n
  +\mathcal L_c(\G_n)-\frac{t^{c2}}{\sqrt n}\,\G_n^{(1)}\otimes\G_n^{(2)},
\end{aligned}
\]
by~\eqref{eq:Zt}. Moreover $\nI{D_n'(t^c)-D'_c}=O_P(n^{-1/2})$,
$\nI{C_n}\le1$, $\widehat h_n=O_P(1)$ and $\nI{\G_n^{(1)}\otimes\G_n^{(2)}}\le\kappa_n^2$.
With~\eqref{eq:unif-t},
\[
  \mathbb W_n=\mathcal L_c(\G_n)-\frac{2\ell_c(\G_n)}{J^c}\,D'_c+O_P(n^{-1/2})
  =\Phi_c(\G_n)+O_P(n^{-1/2})\quad\text{in }\ell^\infty(\Rb^2),
\]
which is~\eqref{eq:unif-D}.

\emph{Step 6: joint convergence.} The map
$w\mapsto\bigl(2p^{c2}\ell_c(w)/J^c,\ \Phi_c(w)\bigr)$ is linear and continuous
from $\ell^\infty(\Rb^2)$ to $\R\times\ell^\infty(\Rb^2)$, because
$\nI{\mathcal L_c(w)}\le5t_u^2\nI w$, $\nI{\mathcal L_c'(w)}\le7t_u\nI w$ and
$\nG\cdot\le\nI\cdot$. Since $\G_n\dto\G$ in $\ell^\infty(\Rb^2)$ (the class of
lower-left quadrants is $L$-Donsker), the continuous mapping theorem and
Slutsky's lemma, applied to~\eqref{eq:unif-p}--\eqref{eq:unif-D}, give~\eqref{eq:joint-clt}.
The limit is a continuous linear image of the Gaussian process $\G$, hence
centred Gaussian, and its covariance follows from the i.i.d.\
representations~\eqref{eq:iid-rep}. \qed

\subsection{Proofs of Corollaries~\ref{cor:repr}--\ref{thm:alt}}

\emph{Corollary~\ref{cor:repr}.} By definition,
$\mathsf T_n=n\,d_n(\ph)=\nG{\sqrt n\,D(\cdot\,;\Ln,\ph)}^2$, and
by the definition of $\mathbb W_n$, $\sqrt n\,D(\cdot\,;\Ln,\ph)=\sqrt n\,D^c+\mathbb W_n$.
Expanding the square, with $\nG{D^c}^2=d_{\min}$, gives~\eqref{eq:repr}. Finally,
\[
  \ip{D^c}{\Phi_c(w)}=\ip{D^c}{\mathcal L_c(w)}-\frac{2\ell_c(w)}{J^c}\,\ip{D^c}{D'_c},
\]
and the last inner product vanishes by~\eqref{eq:orth}. Finally, by~\eqref{eq:unif-D}
and $\nG\cdot\le\nI\cdot$,
$\bigl|\nG{\mathbb W_n}^2-\nG{\Phi_c(\G_n)}^2\bigr|\le\nI{\mathbb W_n-\Phi_c(\G_n)}\bigl(\nI{\mathbb W_n}+\nI{\Phi_c(\G_n)}\bigr)=O_P(n^{-1/2})$.
Since $z\mapsto\nG{\Phi_c(z)}^2$ is continuous on $\ell^\infty(\Rb^2)$ and
$\G_n\dto\G$, the continuous mapping theorem and Slutsky's lemma
give $\nG{\mathbb W_n}^2\dto\nG{\Phi_c(\G)}^2$. The limit $\Phi_c(\G)$ is the centred Gaussian process of
Lemma~\ref{lem:linear}, with covariance $\operatorname{Cov}(\phi_{c,\bu}(\bm X),\phi_{c,\bv}(\bm X))$.
Taking a measurable version, Tonelli's theorem gives
$\EE\,\mathsf T^c_\infty=\int\EE[\Phi_c(\G)(\bu)^2]\,dG(\bu)=\int\operatorname{Var}\phi_{c,\bu}(\bm X)\,dG(\bu)$. \qed

\emph{Corollary~\ref{thm:null}.} Under $H_0$, $D^c=0$, $D'_c=\Dd=-p^\ast\Psi$ and
$J^c=2\nG\Dd^2$. By Lemma~\ref{lem:decomp}(ii), $\mathcal L_c(w)=t^\ast(w-w^{(1)}\otimes F_2-F_1\otimes w^{(2)})$,
so $\mathcal L_c(\G_n)=t^\ast\xi_n$. Since $D^c=0$, $\ell_c(w)=\ip{\mathcal L_c(w)}{\Dd}$ and
\[
  \frac{2\ell_c(w)}{J^c}\,\Dd=\frac{\ip{\mathcal L_c(w)}{\Dd}}{\nG\Dd^2}\,\Dd=\Pi\mathcal L_c(w),
  \qquad\text{i.e.}\qquad \Phi_c(w)=(I-\Pi)\mathcal L_c(w).
\]
By~\eqref{eq:unif-p}, using $t^\ast p^\ast=1$,
$\sqrt n(\ph-p^\ast)=\frac{2p^{\ast2}}{2p^{\ast2}\nG\Psi^2}\,t^\ast\ip{\xi_n}{-p^\ast\Psi}+O_P(n^{-1/2})$,
which is~\eqref{eq:phat-lin}. By~\eqref{eq:repr} with $D^c=0$, $\mathsf T_n=\nG{\mathbb W_n}^2$, and by~\eqref{eq:unif-D}
and $\nG\cdot\le\nI\cdot$,
\[
  \bigl|\mathsf T_n-\nG{\Phi_c(\G_n)}^2\bigr|
  \le\nI{\mathbb W_n-\Phi_c(\G_n)}\bigl(\nI{\mathbb W_n}+\nI{\Phi_c(\G_n)}\bigr)=O_P(n^{-1/2}),
\]
with $\nG{\Phi_c(\G_n)}^2=t^{\ast2}\nG{(I-\Pi)\xi_n}^2$. This is~\eqref{eq:T-lin}.

\emph{Limits.} By Lemma~\ref{lem:linear} applied as after~\eqref{eq:xi},
$\xi_n\dto\xi$ in $\ell^\infty(\Rb^2)$, with covariance~\eqref{eq:Gamma}. The maps
$z\mapsto\ip z\Psi$ and $z\mapsto\nG{(I-\Pi)z}^2$ are continuous on
$\mathcal B(\R^2)$, by~\eqref{eq:L2-linf}.
This subspace contains the paths of $\xi_n$, and, since $\G$ admits a version with
$\rho_L$-continuous (hence Borel) paths \citep[Sec.~2.1]{vdVW1996}, also those of
$\xi$. The continuous mapping
theorem and Slutsky's lemma, applied to~\eqref{eq:phat-lin}--\eqref{eq:T-lin}, give
the joint convergence~\eqref{eq:limits}. The variance of the Gaussian limit
$\ip\xi\Psi$ is $\iint\Gamma(\bu,\bv)\Psi(\bu)\Psi(\bv)dG(\bu)dG(\bv)$ by
Fubini's theorem and~\eqref{eq:Gamma}, or directly by Lemma~\ref{lem:linear}(iii)
with $\mu=\int\mu_\bv\Psi(\bv)\,dG(\bv)$.
\emph{Kernel.} Under $H_0$,
$\phi_{c,\bu}(\bx)=\Phi_c(\mathsf e_\bx)(\bu)=t^\ast\bigl[(I-\Pi)\varphi_\cdot(\bx)\bigr](\bu)$,
because $\mathcal L_c(\mathsf e_\bx)=t^\ast\varphi_\cdot(\bx)$ by~\eqref{eq:phi}.
The map $I-\Pi$ is linear and acts on the variable $\bu$, so it commutes
with $\EE$ by Fubini's theorem. Hence
$\bar\phi_{c,\bu}(\bx)=t^\ast[(I-\Pi)\bar\varphi_\cdot(\bx)](\bu)$ and
$\bar h_c=\bar h_0$ in~\eqref{eq:null-kernel}. Since
$\int\bar h_0(\bx,\bx)\,dL(\bx)=\int\operatorname{Var}\phi_{c,\bu}(\bm X)\,dG(\bu)$ by Fubini's
theorem, the mean is that of Corollary~\ref{cor:repr}.

\emph{Continuity of the cdf.} The process $Y=\Phi_c(\G)=t^\ast(I-\Pi)\xi$ is
centred Gaussian, and $\EE\nG Y^2<\infty$, so it is a centred Gaussian random
element of $L^2(G)$. Its covariance operator is positive, self-adjoint and
trace class, with kernel $\operatorname{Cov}(\phi_{c,\bu}(\bm X),\phi_{c,\bv}(\bm X))$.
Its nonzero eigenvalues $\mu_k$ coincide with those of the integral operator
with kernel $\bar h_0$ on $L^2(L)$. The Karhunen--Lo\`eve expansion gives
$\mathsf T_\infty=\nG Y^2=\sum_k\mu_k Z_k^2$ with $Z_k$ i.i.d.\ $\mathcal N(0,1)$,
the weighted series of Remark~\ref{rem:Vstat}. Under~\ref{A:nondeg},
$\sum_k\mu_k=\EE\nG Y^2=p^{\ast-2}\EE\nG{(I-\Pi)\xi}^2>0$, so $\mu_1>0$. The cdf of
$\mu_1\chi^2_1$ plus an independent nonnegative variable is continuous and
strictly increasing on $[0,\infty)$. \qed

\emph{Corollary~\ref{thm:alt}.} The identity
$\ip{D^c}{\Phi_c(w)}=\ip{D^c}{\mathcal L_c(w)}$ and the
decomposition~\eqref{eq:alt-decomp} are Corollary~\ref{cor:repr}. By~\eqref{eq:unif-D},
$\mathbb W_n=\Phi_c(\G_n)+O_P(n^{-1/2})$ in sup-norm, so $\nG{\mathbb W_n}^2=O_P(1)$ and
\[
  2\sqrt n\,\ip{D^c}{\mathbb W_n}=2\sqrt n\,\ip{D^c}{\Phi_c(\G_n)}+O_P(1)
  =2\sqrt n\,\ip{D^c}{\mathcal L_c(\G_n)}+O_P(1).
\]
Linearity and Fubini's theorem give
\[
  \sqrt n\,\ip{D^c}{\mathcal L_c(\G_n)}=\sum_{i=1}^n\bigl(\chi_c(\bm X_i)-\EE\chi_c(\bm X)\bigr),
\]
which proves~\eqref{eq:alt-d}. Next, \eqref{eq:alt-p} is~\eqref{eq:unif-p}
combined with~\eqref{eq:iid-rep}. Dividing~\eqref{eq:alt-d} by $\sqrt n$ and
using $\mathsf T_n=n\,d_n(\ph)$, both coordinates of~\eqref{eq:alt-limits} are,
up to $O_P(n^{-1/2})$, normalised sums of the i.i.d.\ bounded vectors
$\bigl(2p^{c2}\lambda_c(\bm X_i)/J^c,\ 2\chi_c(\bm X_i)\bigr)$. The bivariate central limit
theorem and Slutsky's lemma give the stated Gaussian limit, with the
moments~\eqref{eq:infl-alt}. \qed

\subsection{Proof of Theorem~\ref{thm:bootstrap}}

The proof compares the bootstrap noise $\mathbb W_n^\zeta=\widehat\Phi_n(\G_n^\zeta)$
with an \emph{oracle} copy $\widetilde{\mathbb W}_n^\zeta=\Phi_c(\widetilde\G_n^\zeta)$, built
from the true map $\Phi_c$, i.e.\ from the true $(t^\ast,F_1,F_2,\Psi)$, and
bounds the difference of their squared norms explicitly.

\emph{Key quantities.} Let $S_\zeta=n^{-1/2}\sum_{i=1}^n\zeta_i$,
\[
  \widetilde\G_n^\zeta(\bu)=\frac1{\sqrt n}\sum_{i=1}^n\zeta_i\bigl(\1\{\bm X_i\le\bu\}-L(\bu)\bigr),\qquad
  \widetilde\xi_n^\zeta=\widetilde\G_n^\zeta-\widetilde\G_n^{\zeta(1)}\otimes F_2-F_1\otimes\widetilde\G_n^{\zeta(2)},
\]
and the oracle statistic $\widetilde{\mathsf T}_n^\zeta=\nG{\widetilde{\mathbb W}_n^\zeta}^2=t^{\ast2}\nG{(I-\Pi)\widetilde\xi_n^\zeta}^2$.

\emph{Reduction to the null form under $H_0$.} Under $H_0$,
$\nI{\widehat D_n}\le\nI{D_n(t^\ast)}+5t_u|\that-t^\ast|=O_P(n^{-1/2})$ by
Lemma~\ref{lem:decomp}(ii) and~\eqref{eq:hhat-bound}, hence
$\widehat J_n-2\nG{\widehat D{}'_n}^2=-4\ip{\widehat D_n}{C_n}=O_P(n^{-1/2})$, and
$\nG{\widehat D{}'_n}\ge p^\ast c_G/2$ with probability tending to one
by~\eqref{eq:dir-rem} below. Since $\nI{\widehat{\mathcal L}_n'(w)}\le7t_u\nI w$, the
definition of $\widehat\Phi_n$ gives
$\lVert\widehat\Phi_n-(I-\widehat\Pi)\widehat{\mathcal L}_n\rVert_{\mathrm{op}}=O_P(n^{-1/2})$ as an
operator from $\ell^\infty(\Rb^2)$ to $L^2(G)$. As $\nI{\G_n^\zeta}=O_P(1)$, the squared norms of
$\widehat\Phi_n(\G_n^\zeta)$ and of its null form $\that(I-\widehat\Pi)\xi_n^\zeta$ differ by
$O_P(n^{-1/2})$. In the rest of part~(i), $\mathsf T_n^\zeta$ therefore denotes
the null form $\that^{\,2}\nG{(I-\widehat\Pi)\xi_n^\zeta}^2$.
Write also $\kappa_n^\zeta=\nI{\widetilde\G_n^\zeta}$.

\emph{Oracle part.} Because the quadrant class is Donsker for every $L$, the
conditional multiplier central limit theorem \citep[Thm~2.9.6]{vdVW1996}
states that $\widetilde\G_n^\zeta\dto\G$ conditionally on the data, in
probability, in the bounded Lipschitz sense. The multiplier inequalities
\citep[Lemma~2.9.1]{vdVW1996} give $\kappa_n^\zeta=O_P(1)$. By the
continuous mapping theorem (as in the proof of Corollary~\ref{thm:null}),
$\widetilde{\mathsf T}_n^\zeta\dto\mathsf T_\infty$ conditionally in probability.

\emph{Plug-in errors.} Since
$\1\{\bm X_i\le\bu\}-\Ln(\bu)=(\1\{\bm X_i\le\bu\}-L(\bu))-E_n(\bu)$, we have exactly
$\G_n^\zeta=\widetilde\G_n^\zeta-S_\zeta E_n$, and $\widehat\varphi_\bu$ is linear in the
three indicators, so
\begin{equation}\label{eq:boot-decomp}
  \xi_n^\zeta=\G^\zeta_n-\G_n^{\zeta(1)}\otimes\widehat F_2-\widehat F_1\otimes\G_n^{\zeta(2)} .
\end{equation}
The plug-in marginals satisfy, exactly,
$\widehat F_j-F_j=(\that-t^\ast)A_{n,j}+t^\ast E_n^{(j)}$ under $H_0$, hence
$\nI{\widehat F_j-F_j}\le|\that-t^\ast|+t_u\nI{E_n}$. Subtracting the oracle
process term by term,
\begin{equation}\label{eq:boot-rem}
  \nI{\xi_n^\zeta-\widetilde\xi_n^\zeta}\le3|S_\zeta|\,\nI{E_n}
  +2\bigl(\kappa_n^\zeta+|S_\zeta|\nI{E_n}\bigr)\bigl(|\that-t^\ast|+t_u\nI{E_n}\bigr)=:\rho_n^\zeta .
\end{equation}
Since $S_\zeta=O_P(1)$, $\kappa_n^\zeta=O_P(1)$, $\nI{E_n}=\kappa_n/\sqrt n$ and
$|\that-t^\ast|=O_P(n^{-1/2})$ by~\eqref{eq:hhat-bound}, $\rho_n^\zeta=O_P(n^{-1/2})$.
The estimated direction satisfies, by Lemma~\ref{lem:decomp} and
$\partial_tD(t)|_{t=t^\ast}=B^\ast-2t^\ast C^\ast=\Dd$,
\begin{equation}\label{eq:dir-rem}
  \nI{\widehat{D}{}'_n-\Dd}\le\nI{B_n-B^\ast}+2t_u\nI{C_n-C^\ast}+2|\that-t^\ast|
  \le(3+6t_u)\nI{E_n}+2|\that-t^\ast|=O_P(n^{-1/2}).
\end{equation}

\emph{Comparison.} Apply Lemma~\ref{lem:proj} with $P=I-\widehat\Pi$, $P'=I-\Pi$,
$z=\xi_n^\zeta$, $z'=\widetilde\xi_n^\zeta$. Since
$\nG\Dd\ge p^\ast c_G$, \eqref{eq:dir-rem} gives
$\lVert\widehat\Pi-\Pi\rVert_{\mathrm{op}}\le4\nI{\widehat{D}{}'_n-\Dd}/(p^\ast c_G)$. Using also
$|\that^2-t^{\ast2}|\le2t_u|\that-t^\ast|$, $\nG{\widetilde\xi_n^\zeta}\le3\kappa_n^\zeta$
and $\nG{\xi_n^\zeta}\le3\kappa_n^\zeta+\rho_n^\zeta$, we get
\[
  \bigl|\mathsf T_n^\zeta-\widetilde{\mathsf T}_n^\zeta\bigr|
  \le t_u^2(6\kappa_n^\zeta+\rho_n^\zeta)\Bigl(\frac{4\nI{\widehat{D}{}'_n-\Dd}}{p^\ast c_G}(3\kappa_n^\zeta+\rho_n^\zeta)+\rho_n^\zeta\Bigr)
  +2t_u|\that-t^\ast|(3\kappa_n^\zeta+\rho_n^\zeta)^2 ,
\]
which is $O_P(n^{-1/2})$ by~\eqref{eq:boot-rem}--\eqref{eq:dir-rem}.

(i) An unconditional $o_P(1)$ bound transfers to the conditional law: for
every $\varepsilon>0$,
$\EE\bigl[\PP(|\mathsf T_n^\zeta-\widetilde{\mathsf T}_n^\zeta|>\varepsilon\mid\bm X_{1:n})\bigr]
=\PP(|\mathsf T_n^\zeta-\widetilde{\mathsf T}_n^\zeta|>\varepsilon)\to0$, so the conditional
probability tends to $0$ in probability. Hence
$\mathsf T_n^\zeta\dto\mathsf T_\infty$ conditionally in probability. Since the
cdf of $\mathsf T_\infty$ is continuous (Corollary~\ref{thm:null}), P\'olya's
theorem yields the uniform convergence of conditional cdfs. Since this cdf is
also strictly increasing, $\widehat q_{1-\alpha}\to q_{1-\alpha}$ in
probability. Slutsky's lemma and Corollary~\ref{thm:null} then give
$\PP(\mathsf T_n>\widehat q_{1-\alpha})\to\PP(\mathsf T_\infty>q_{1-\alpha})=\alpha$.

(ii) \emph{Convergence to the reference law.} In case~(b), $\that\to t^c$ almost
surely (Theorem~\ref{thm:consistency}(ii)), and, by~\eqref{eq:unifrate} and the
uniform continuity of $t\mapsto(D(t),D'(t))$ on $T$, $\widehat D_n\to D^c$,
$\widehat D{}'_n\to D'_c$ in sup-norm and $\widehat J_n\to J^c>0$ almost surely.
Moreover $A_{n,j}\to A_j$ uniformly, so $\widehat\gamma_n\to\gamma^\ast$,
$\widehat{\mathcal L}_n\to\mathcal L_c$ and $\widehat{\mathcal L}_n'\to\mathcal L_c'$ in operator
norm. Hence $\lVert\widehat\Phi_n-\Phi_c\rVert_{\mathrm{op}}\to0$ almost surely. Since
$\G_n^\zeta=\widetilde\G_n^\zeta-S_\zeta E_n$ with $S_\zeta E_n=O_P(n^{-1/2})$,
\[
\begin{aligned}
  \bigl|\nG{\widehat\Phi_n(\G_n^\zeta)}^2-\nG{\Phi_c(\widetilde\G_n^\zeta)}^2\bigr|
  &\le\bigl(\lVert\widehat\Phi_n-\Phi_c\rVert_{\mathrm{op}}\nI{\G_n^\zeta}+\lVert\Phi_c\rVert_{\mathrm{op}}|S_\zeta|\nI{E_n}\bigr)\\
  &\quad\times\bigl(\nG{\widehat\Phi_n(\G_n^\zeta)}+\nG{\Phi_c(\widetilde\G_n^\zeta)}\bigr)=o_P(1).
\end{aligned}
\]
By the conditional multiplier central limit theorem and the continuous
mapping theorem,
\[
  \nG{\Phi_c(\widetilde\G_n^\zeta)}^2\dto\nG{\Phi_c(\G)}^2=\mathsf T^c_\infty
\]
conditionally in probability, and the $o_P(1)$ bound transfers to the
conditional law as in part~(i). In particular $\mathsf T_n^\zeta=O_P(1)$.

\emph{Bounded quantile and power.} For the null form, under any $L$,
$\nI{\widehat F_j}\le t_u+1$ and $\that\le t_u$, so by~\eqref{eq:boot-decomp},
$\that^{\,2}\nG{(I-\widehat\Pi)\xi_n^\zeta}^2\le t_u^2\nI{\xi_n^\zeta}^2\le t_u^2(2t_u+3)^2\nI{\G^\zeta_n}^2=O_P(1)$
unconditionally. In both cases, then, $\mathsf T_n^\zeta=O_P(1)$. By Markov's inequality applied conditionally,
$\PP(\widehat q_{1-\alpha}>x)\le\PP\bigl(\PP(\mathsf T_n^\zeta>x\mid\bm X_{1:n})\ge\alpha\bigr)
\le\alpha^{-1}\PP(\mathsf T_n^\zeta>x)$, so $\widehat q_{1-\alpha}=O_P(1)$,
whereas $\mathsf T_n\to\infty$ a.s.\ by Theorem~\ref{thm:consistency}(ii). \qed

\subsection{Proof of Proposition~\ref{prop:empiricalG}}

Write $\lVert\cdot\rVert_n$ and $\langle\cdot,\cdot\rangle_n$ for the norm and
inner product of $L^2(\Ln)$, and $\lVert\cdot\rVert_L$ for those of $L^2(L)$. All
sup-norm bounds of the previous proofs are unchanged, since they do not
involve $G$.

\emph{Consistency and rate.} The chain of inequalities in the proof of
Theorem~\ref{thm:consistency}(i) holds verbatim with $\lVert\cdot\rVert_n$,
giving $\lVert D(\that)\rVert_n\le12t_u^2\nI{E_n}$. The lower bound becomes
$\lVert D(t)\rVert_n\ge c_np^\ast|t-t^\ast|$ with
$c_n=\inf_{s\in S}\lVert K_0+s\Delta\rVert_n$. Now $\lVert K_0+s\Delta\rVert_n^2$ is a
quadratic polynomial in $s$ whose three coefficients $\int K_0^2d\Ln$,
$\int K_0\Delta\,d\Ln$, $\int\Delta^2d\Ln$ converge a.s.\ to their $L$-counterparts
by the strong law of large numbers. Hence $c_n\to c_L>0$ a.s., and the rate
of Theorem~\ref{thm:consistency}(i) follows. Part (ii) is identical, since
$\sup_T|\delta_n-\delta|\to0$ a.s.\ also holds when $G=\Ln$ in both $\delta_n$ and $\delta$, and
$\int D^2(\cdot\,;1/t)\,d\Ln\to\int D^2(\cdot\,;1/t)\,dL$ uniformly in $t$
(again a polynomial in $t$ with a.s.\ convergent coefficients).

\emph{Null distribution.} The proof of Theorem~\ref{thm:unified} in case~(a)
uses the integrating measure only through $\lVert\cdot\rVert_n\le\nI\cdot$, through
the curvature $J^c=2\lVert\Dd\rVert_n^2$, and through the constant $c_G$ in the
window~\eqref{eq:window}, which is replaced by the random constant $c_n$
above, with $c_n\to c_L>0$ a.s. On the event $\{c_n\ge c_L/2\}$, whose
probability tends to one, $J^c\ge p^{\ast2}c_L^2/2$ and all steps hold
verbatim with $c_G$ replaced by $c_L/2$; the sup-norm expansion~\eqref{eq:unif-D}
does not involve $G$ except through $\Pi$. Hence, as in the proof of
Corollary~\ref{thm:null},
$\mathsf T_n=t^{\ast2}\bigl(\lVert\xi_n\rVert_n^2-\langle\xi_n,\Psi\rangle_n^2/\lVert\Psi\rVert_n^2\bigr)+O_P(n^{-1/2})$,
and similarly for $\ph$, with a remainder $O_P(n^{-1/2})$. It remains to show that
$(\lVert\xi_n\rVert_n^2,\langle\xi_n,\Psi\rangle_n)\dto(\lVert\xi\rVert_L^2,\langle\xi,\Psi\rangle_L)$.
Since $L$ and $F_0$ are continuous, so are $L_j$, $F_j$, $\Psi$, and
$\rho_L(\bu,\bv)^2:=\operatorname{Var}(\1\{\bm X\le\bu\}-\1\{\bm X\le\bv\})\le
|L_1(u_1)-L_1(v_1)|+|L_2(u_2)-L_2(v_2)|$. The Brownian bridge $\G$ has a
version with uniformly $\rho_L$-continuous paths \citep[Sec.~2.1]{vdVW1996},
hence with paths continuous on $\R^2$ for the Euclidean topology. Thus $\xi$
takes its values in the set $C_b(\R^2)$ of bounded continuous functions. Let
$\mathcal P$ be the set of probability measures on $\R^2$ with the weak
topology, and consider the map
$\Phi:\ell^\infty(\Rb^2)\times\mathcal P\to\R^2$,
$\Phi(z,\mu)=(\int z^2d\mu,\int z\Psi d\mu)$. For $z_n\to z\in C_b(\R^2)$ in
sup-norm and $\mu_n\Rightarrow\mu$,
\[
  \Bigl|\int z_n^2d\mu_n-\int z^2d\mu\Bigr|\le\nI{z_n^2-z^2}+\Bigl|\int z^2d(\mu_n-\mu)\Bigr|\to0,
\]
and similarly for the second coordinate, so $\Phi$ is continuous at every
point of $C_b(\R^2)\times\{L\}$. Since $\Ln\Rightarrow L$ almost surely (a
deterministic limit), $(\xi_n,\Ln)\dto(\xi,L)$, and the extended continuous
mapping theorem \citep[Thm~1.11.1]{vdVW1996} gives the claim.

\emph{Bootstrap.} The same argument applies conditionally to $\xi_n^\zeta$,
using the conditional convergence $\G_n^\zeta\dto\G$ and the a.s.\ weak
convergence of $\Ln$; part (ii) is unchanged since
$\lVert\cdot\rVert_n\le\nI\cdot$. \qed

\section{An alternative absorbed by a wrong weight}\label{app:example}

This appendix constructs dependent components $F$ that the IBM contrast cannot
distinguish from the null hypothesis, characterises all of them, and
illustrates the consequence for the test.

\paragraph{Characterisation.} By~\eqref{eq:inverted-explicit},
$pF(\cdot\,;p)=(p-p^\ast)F_0+p^\ast F$. Writing $\lambda=p/p^\ast$ and
$H=F(\cdot\,;p)$, this is equivalent to
\begin{equation}\label{eq:absorbed}
  F=\lambda H+(1-\lambda)F_0 .
\end{equation}
Since $d(p)=0$ if and only if $H=F(\cdot\,;p)$ is of product form ($G$-a.e.),
we obtain the following statement.

\begin{proposition}\label{prop:absorbed}
Let $\Lambda=\{p/p^\ast:\ p\in\Theta\}$. Then $d_{\min}=0$ if and only if
$F=\lambda\,H_1\otimes H_2+(1-\lambda)F_0$ $G$-a.e.\ for some $\lambda\in\Lambda$
and some functions $H_1,H_2$ on $\Rb$, normalised so that $H_1(+\infty)=H_2(+\infty)=1$ (they are then the marginals of $H_1\otimes H_2$). In that case, $\lambda p^\ast$ is a zero of $d$,
\begin{equation}\label{eq:absorbedK}
  K=(1-\lambda)\,K_0+\frac{1-\lambda}{\lambda}\,\Delta ,
\end{equation}
and, for every $p\in(0,1]$, with $s'=\lambda p^\ast/p$ and
$\Delta_H=(H_1-F_{0,1})\otimes(H_2-F_{0,2})$,
\begin{equation}\label{eq:absorbedD}
  D(\cdot\,;L,p)=(1-s')\bigl(K_0+s'\Delta_H\bigr).
\end{equation}
Conversely, if $K_0$ and $\Delta$ are linearly independent in $L^2(G)$
(Definition~\ref{def:linind}) and $K=aK_0+b\Delta$, then $d$ vanishes on
$\Theta$ if and only if $(1-a)(1+b)=1$ and $p^\ast/(1+b)\in\Theta$, i.e.\ if
and only if~\eqref{eq:absorbedK} holds with $\lambda=1/(1+b)\in\Lambda$.
\end{proposition}

\begin{proof}
The first equivalence is~\eqref{eq:absorbed}, with $H=F(\cdot\,;\lambda p^\ast)$.
For~\eqref{eq:absorbedK}, the marginals of $F$ are
$F_j=\lambda H_j+(1-\lambda)F_{0,j}$, so $F_j-F_{0,j}=\lambda(H_j-F_{0,j})$ and
$\Delta=\lambda^2\Delta_H$. Expanding,
\[
\begin{aligned}
  K&=\lambda H_1\otimes H_2+(1-\lambda)F_0-F_1\otimes F_2\\
   &=(1-\lambda)\bigl[F_0-F_{0,1}\otimes F_{0,2}\bigr]
     +\lambda(1-\lambda)\bigl[H_1\otimes H_2-H_1\otimes F_{0,2}-F_{0,1}\otimes H_2+F_{0,1}\otimes F_{0,2}\bigr]\\
   &=(1-\lambda)K_0+\lambda(1-\lambda)\Delta_H
    =(1-\lambda)K_0+\frac{1-\lambda}{\lambda}\Delta .
\end{aligned}
\]
Substituting~\eqref{eq:absorbedK} into~\eqref{eq:structure}, with $s=p^\ast/p$
and $s'=\lambda s$,
\[
  D=(1-s)K_0+s(1-s)\Delta+s(1-\lambda)K_0+\frac{s(1-\lambda)}{\lambda}\Delta
   =(1-s')K_0+\frac{s}{\lambda}(1-s')\Delta
   =(1-s')\bigl(K_0+s'\Delta_H\bigr),
\]
because $(s/\lambda)\Delta=s\lambda\Delta_H=s'\Delta_H$. This
is~\eqref{eq:absorbedD}. For the converse, \eqref{eq:structure} with
$K=aK_0+b\Delta$ gives $D=\bigl(1-s(1-a)\bigr)K_0+s(1+b-s)\Delta$. By linear
independence, $D=0$ in $L^2(G)$ if and only if both coefficients vanish, i.e.\
$s=1+b$ and $(1+b)(1-a)=1$. The weight is then $p=p^\ast/s=p^\ast/(1+b)$. Setting
$\lambda=1/(1+b)$ gives $a=1-\lambda$ and $b=(1-\lambda)/\lambda$.
\end{proof}

Identity~\eqref{eq:absorbedD} is exactly the null-hypothesis
structure~\eqref{eq:structureH0} for the alternative representation
$L=(1-\lambda p^\ast)F_0+\lambda p^\ast H_1\otimes H_2$. The IBM contrast therefore
sees a null model with weight $\lambda p^\ast$. In particular, only a
one-parameter curve in the two-dimensional space
$\operatorname{span}\{K_0,\Delta\}$ is absorbed. Its tangent at $\lambda=1$ is
$\frac{d}{d\lambda}\bigl[(1-\lambda)K_0+\lambda^{-1}(1-\lambda)\Delta\bigr]_{\lambda=1}=-(K_0+\Delta)=-\Psi$,
which is the undetectable local direction of Remark~\ref{rem:local}.

\paragraph{A concrete example.} Take the known source of
Scenario~A in Section~\ref{sec:simu}: $F_0$ has standard exponential marginals
and a Clayton copula with $\theta_0=2$. Let $H=H_1\otimes H_2$ with
$\mathrm{Gamma}(4,1)$ marginals, and fix weights $0<p'<p^\ast\le1$. Define
\[
  F=\lambda\,H_1\otimes H_2+(1-\lambda)F_0,\qquad \lambda=p'/p^\ast\in(0,1).
\]
Then $F$ is a genuine cdf, a mixture of an independent and a dependent
distribution, and it is not of product form: by~\eqref{eq:absorbedK},
$K=(1-\lambda)K_0+\lambda(1-\lambda)\Delta_H$. Here $K_0\ge0$, with strict
inequality inside the quadrant, because the Clayton copula is positively quadrant
dependent. Also $\Delta_H\ge0$, because $\mathrm{Gamma}(4,1)$ is stochastically
larger than $\mathrm{Exp}(1)$, so that $H_j-F_{0,j}\le0$ for $j=1,2$. Hence
$K\neq0$, and $H_1$ holds. The observed
distribution is nevertheless
\[
  L=(1-p^\ast)F_0+p^\ast F=(1-p')F_0+p'\,H_1\otimes H_2 ,
\]
which is a null model with weight $p'$. The pair $(p^\ast,F)$ is not identifiable
from $L$: part of what was declared ``unknown component'' is in fact
contamination. This is the familiar non-identifiability of contamination
models, which is usually resolved by taking the smallest admissible weight
\citep{BDV2006,PatraSen2016}. Here that rule selects $p'$ and the independent
component $H$.

We simulated this example with $p^\ast=0.6$ and $p'=0.3$ (so $\lambda=1/2$), using
the IBM test exactly as in Section~\ref{sec:simu} ($G=\Ln$, $N=200$ multiplier
replications, nominal level $5\%$). The empirical rejection rates were
$5.0\%$, $5.3\%$ and $4.0\%$ for $n=500$, $1000$ and $2000$ ($300$, $300$
and $150$ replications). The weight estimates had mean $0.301$, $0.302$ and
$0.299$, and standard deviations $0.037$, $0.023$ and $0.018$. As predicted
by~\eqref{eq:absorbedD} and Theorem~\ref{thm:consistency}, $\ph$
converges to the absorbing weight $p'$, not to $p^\ast$, and the test behaves
exactly as under the null hypothesis.

\section*{Acknowledgements}
Pierre Vandekerkhove would like to acknowledge that this work was supported by
the CNES under the BIOSWOT-AdAC project and the MIO Axes Transverses. Pierre
Vandekerkhove would also like to acknowledge the support received from the
Research Chair ACTIONS under the aegis of the Risk Foundation, an initiative by
BNP Paribas Cardif and the French Institute of Actuaries.

\bibliographystyle{plainnat}
\bibliography{refs}

\end{document}

%% file: simulation_section.tex
\subsection{Design}\label{sec:simu-design}

The known source $F_0$ has standard exponential marginals
and either a Clayton copula with parameter $\theta_0=2$ (Kendall's $\tau=0.5$;
Scenario~A) or the independence copula (Scenario~B, in which $K_0=0$). The
unknown component $F$ has $\mathrm{Gamma}(4,1)$ marginals. Its copula is the
independence copula under $H_0$. Under $H_1$ it is a Gaussian copula with
$\rho\in\{0.3,0.5,-0.3\}$ or a Clayton copula with $\theta=1$ ($\tau=1/3$).
In both scenarios $K_0\ge0$ and $\Delta\ge0$, so~\ref{A:ident} holds. We
consider $p^\ast\in\{0.3,0.6\}$ and $n\in\{250,500,1000,2000\}$, with
$\Theta=[0.01,1]$ and $G=\Ln$ (Proposition~\ref{prop:empiricalG}). Critical
values come from $N=300$ bootstrap replications with Gaussian multipliers (Section~\ref{sec:bootstrap}), and the nominal
level is $5\%$. Each entry is based on $1000$ Monte Carlo replications ($500$
for $n=2000$).

We compare the IBM test with two benchmarks. The \emph{naive} test is the
BKR Cram\'er--von Mises test \citep{BKR1961}, calibrated by the same
multiplier bootstrap and applied to the observed sample, i.e.\ $\mathsf T_n$
computed at $p=1$ without projection. The \emph{oracle} test is the same BKR
test applied to the latent subsample drawn from $F$ (of random size close to
$np^\ast$). It is not available in practice and only serves as an upper
reference.

Figure~\ref{fig:component} displays the density of the unknown component under the null hypothesis and under the four alternatives. All five
share the same $\mathrm{Gamma}(4,1)$ marginals and differ only through their
copula. Figure~\ref{fig:mixture} displays the known source $F_0$ and the
resulting observed densities $L$ under $H_0$. In Scenario~B both sources have
independent coordinates, yet the mixture $L$ is clearly not of product form:
its contours bend towards the axes where the exponential source dominates.
This is why tests of independence applied to the observed sample are
misleading.

\begin{figure}[htbp]
\centering
\includegraphics[width=\textwidth]{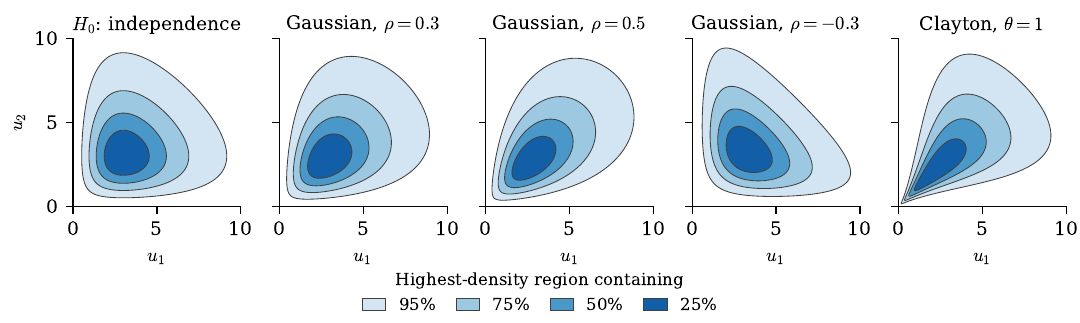}
\caption{Density of the unknown component $F$ in the simulation study:
$\mathrm{Gamma}(4,1)$ marginals with the independence copula ($H_0$, left) and
the four copulas used under $H_1$. Shaded areas are highest-density regions
containing $25\%$, $50\%$, $75\%$ and $95\%$ of the mass.}
\label{fig:component}
\end{figure}

\begin{figure}[htbp]
\centering
\includegraphics[width=0.82\textwidth]{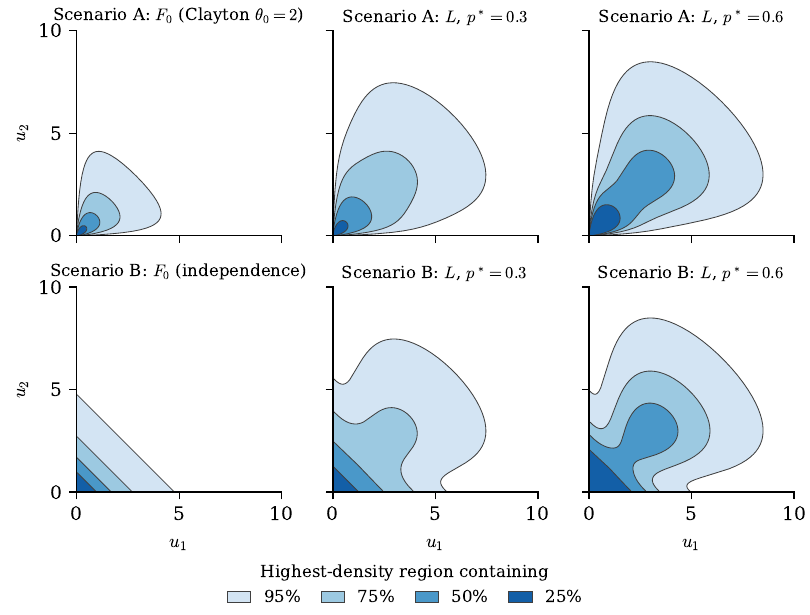}
\caption{Known source $F_0$ (left column, standard exponential marginals) and
observed contamination-model densities $L=(1-p^\ast)F_0+p^\ast F$ under $H_0$, for
$p^\ast=0.3$ (middle) and $p^\ast=0.6$ (right). Top row: Scenario~A
($F_0$ with a Clayton copula, $\theta_0=2$); bottom row: Scenario~B ($F_0$
with independent coordinates). Shaded areas are highest-density regions as in
Figure~\ref{fig:component}.}
\label{fig:mixture}
\end{figure}

\subsection{Sampling the null distribution of the statistic}\label{sec:simu-null}

The test compares $\mathsf T_n$ with the reference distribution of
Section~\ref{sec:bootstrap}, the asymptotic law $\mathsf T_\infty$ of the noise
$\nG{\mathbb W_n}^2$, which under $H_0$ coincides with $\mathsf T_n$
(Corollary~\ref{thm:null}). To visualise this law and its bootstrap estimate,
we consider Scenario~A under $H_0$ with $p^\ast=0.6$ and $n=1000$, and sample it
in three ways.
\begin{enumerate}[label=(\roman*),leftmargin=*]
\item \emph{Monte Carlo.} The experiment is repeated $R=5000$ times; each
  repetition draws a new sample and computes $\mathsf T_n$. The empirical law
  of these values is the finite-sample law of $\mathsf T_n=\nG{\mathbb W_n}^2$
  at $n=1000$. It is not available in practice, since it requires many samples
  from the model.
\item \emph{Multiplier bootstrap.} For a single sample, $N=5000$ replications of
  $\mathsf T_n^\zeta$ in~\eqref{eq:bootstrap} (null form, Gaussian
  multipliers) are drawn. This is what the practitioner computes. We repeat it
  on four independent samples to show its variability.
\item \emph{Limit.} Draws of $\mathsf T_\infty=p^{\ast-2}\nG{(I-\Pi)\xi}^2$ are
  computed with the true $(p^\ast,F_1,F_2)$. The Gaussian process $\xi$ is
  generated by Gaussian multipliers applied to $\varphi_\bu(\bm X_i)$
  in~\eqref{eq:phi}, for a reference sample of size $4000$ from $L$, which also
  serves as integrating measure in place of $L$ (Proposition~\ref{prop:empiricalG}).
  $10{,}000$ draws are used.
\end{enumerate}
Figure~\ref{fig:null} and Table~\ref{tab:nullq} show that the three laws agree
closely. The left panel compares the two approximations of the reference
law available from simulated data through kernel density estimates of their
samples (Gaussian kernel, Silverman's bandwidth): the density of the $5000$
Monte Carlo values of $\mathsf T_n$, and the densities of the $5000$
bootstrap values $\mathsf T_n^\zeta$ obtained from each of the four single
samples. The estimated density of the $\mathsf T_\infty$ draws is added as a
reference. The right panel is a
quantile--quantile plot against $\mathsf T_\infty$. The law is markedly
right-skewed, as expected for the limit of a degenerate $V$-statistic
(Remark~\ref{rem:Vstat}) whose kernel is dominated by a few eigenfunctions. Its mean, which equals the trace
$\int\bar h_0(\bx,\bx)\,dL(\bx)$, is $0.073$, and its $95\%$ quantile is $0.140$. The
Monte Carlo law of $\mathsf T_n$ already matches it at $n=1000$: the rejection
rate of the rule $\mathsf T_n>q_{0.95}(\mathsf T_\infty)$ over the $5000$
repetitions is $5.3\%$. The bootstrap estimates the whole law from one sample.
Two of the four bootstrap laws are indistinguishable from $\mathsf T_\infty$;
the other two are slightly compressed, with $95\%$ quantiles $0.115$ and
$0.128$. This sample-to-sample variability comes mainly from the factor
$\that^{\,2}=\ph^{-2}$ in~\eqref{eq:bootstrap}: the most compressed bootstrap
law corresponds to the sample with the largest estimate, $\ph=0.634$. On these
four samples the bootstrap $p$-values are $0.46$, $0.30$, $0.28$ and $0.034$; the
last sample, with $\mathsf T_n=0.140$, is one of the expected $5\%$ of false
rejections.

\begin{figure}[htbp]
\centering
\includegraphics[width=\textwidth]{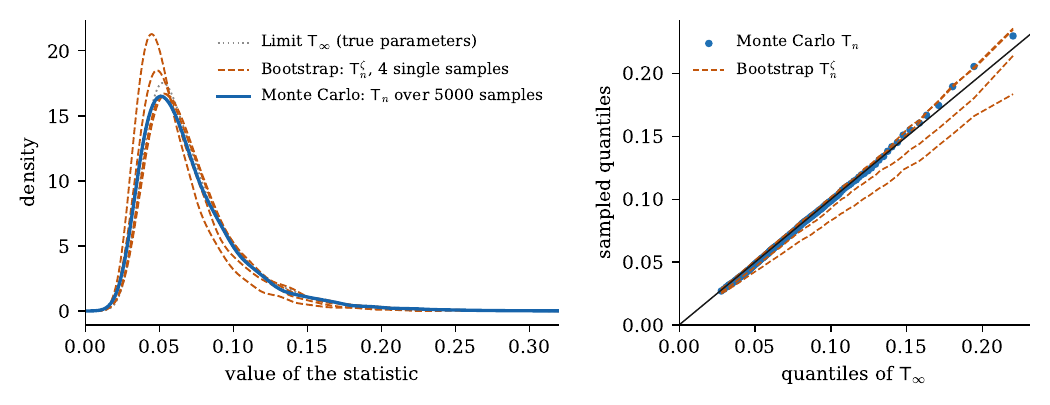}
\caption{Null distribution of the test statistic in Scenario~A, $p^\ast=0.6$,
$n=1000$. Left: kernel density estimates (Gaussian kernel, Silverman's
bandwidth) of the Monte Carlo sample of $\mathsf T_n$ ($5000$ repetitions,
solid), of the multiplier bootstrap samples of $\mathsf T_n^\zeta$ ($5000$
replications) for four single samples (dashed), and of the limit
$\mathsf T_\infty$ computed with the true parameters (dotted). Right:
quantiles of the Monte Carlo and bootstrap laws against those of
$\mathsf T_\infty$.}
\label{fig:null}
\end{figure}

\begin{table}[htbp]
\centering\small
\caption{Quantiles and mean of the sampled null distributions of
Figure~\ref{fig:null}. For the bootstrap, the range over the four samples is
reported.}\label{tab:nullq}
\begin{tabular}{lccccc}
\toprule
 & median & $q_{0.90}$ & $q_{0.95}$ & $q_{0.99}$ & mean\\
\midrule
Limit $\mathsf T_\infty$ & 0.064 & 0.117 & 0.140 & 0.194 & 0.073\\
Monte Carlo $\mathsf T_n$ & 0.063 & 0.116 & 0.142 & 0.206 & 0.072\\
Bootstrap $\mathsf T_n^\zeta$ & 0.053--0.066 & 0.096--0.118 & 0.115--0.143 & 0.166--0.205 & 0.061--0.075\\
\bottomrule
\end{tabular}
\end{table}

\subsection{Level and power}\label{sec:simu-level}

\paragraph{Results under the null.} Table~\ref{tab:size} confirms the
discussion of Section~\ref{sec:intro}. The naive test rejects the null in
every replication, including Scenario~B, where both sources are product
distributions, so it cannot be used. The IBM test holds its nominal level
well in all configurations except Scenario~B with $p^\ast=0.3$ and $n=250$,
where the empirical level is $20.8\%$. This distortion is a weak-separation
effect. When $K_0=0$, \eqref{eq:structureH0} gives
$d(p)=s^2(1-s)^2\nG\Delta^2$ with $s=p^\ast/p$, which is small over the whole
range $p\in[p^\ast,1]$ when $p^\ast$ is small. In about one in five of the
samples of size $250$, $\widehat p_n$ then falls on the boundary $p=1$, as
the large bias and dispersion of $\widehat p_n$ in the table show. The effect
has disappeared at $n=500$. A dependent known source (Scenario~A)
sharpens the contrast. Thus the condition
$F_0\neq F_{0,1}\otimes F_{0,2}$ is not needed asymptotically, but it helps
in small samples. The last column agrees with
Corollary~\ref{thm:null}: once $n\ge500$ (Scenario~A) or $n\ge1000$
(Scenario~B), $\sqrt n\,\mathrm{sd}(\widehat p_n)$ is close to the
asymptotic standard deviation $\sigma$, which was computed by Monte Carlo
integration from the true $F_1,F_2$.

\paragraph{Power.} Table~\ref{tab:power} shows that the power of the IBM test
increases with $n$ against all alternatives, as
Theorem~\ref{thm:bootstrap}(ii) predicts. The price of not observing the
unknown component is substantial when it is rare. With $p^\ast=0.3$, the
IBM test needs about $2000$ observations to reach a power that the oracle
reaches with $250$. With $p^\ast=0.6$ the gap is much smaller, and the IBM
power exceeds $80\%$ from $n=500$ onwards. This agrees with the local
analysis of Remark~\ref{rem:local}: the limiting noise $p^{\ast-1}\xi$ is inflated by the factor $1/p^\ast$
relative to the uncontaminated case, and the direction $\Psi$ is lost. The power figures for Scenario~B with $p^\ast=0.3$ and $n=250$ are
inflated by the size distortion discussed above and should not be read as
evidence of power.

\paragraph{Behaviour under the alternative.} We checked
Theorem~\ref{thm:unified} and Corollary~\ref{thm:alt} in Scenario~A with $p^\ast=0.6$ and a Gaussian copula with
$\rho=0.5$. To match the theory, the integrating measure $G$ was held fixed:
it was the empirical measure of $1500$ points drawn once from $L$. The
pseudo-true values are $p^c=0.563$ and $d_{\min}=1.24\times10^{-3}$. Over $400$
samples of size $n=1000$, the standard deviations of $\sqrt n(\ph-p^c)$ and of
$\sqrt n(d_n(\ph)-d_{\min})$ were $0.631$ and $0.0087$. The asymptotic values
$\sigma_c=0.633$ and $\tau=0.0091$ were computed from
\eqref{eq:infl-alt}.

\input{tab_size}
\input{tab_power}

\FloatBarrier

%% file: tab_size.tex
\begin{table}[t]\centering\small
\caption{Behaviour under $H_0$. Empirical level (in \%) of the IBM test, of the naive BKR test applied to the observed sample and of the oracle BKR test applied to the latent observations from $F$, at nominal level $5\%$; bias and $\sqrt n\times$ standard deviation of $\widehat p_n$, to be compared with the asymptotic value $\sigma$ of Theorem~\ref{thm:null}. The naive test was not run for $n\ge1000$ (--), since it already rejects in every replication for $n\le500$.}\label{tab:size}
\begin{tabular}{llrrrrrr}\toprule
Scenario & $p^\ast$ & $n$ & IBM & naive & oracle & bias$(\widehat p_n)$ & $\sqrt n\,\mathrm{sd}(\widehat p_n)$\\\midrule
A & 0.3 & 250 & 5.4 & 100.0 & 6.2 & +0.008 & 1.16\\
 &  & 500 & 3.7 & 100.0 & 5.2 & +0.002 & 0.73\\
 &  & 1000 & 3.3 & 100.0 & 4.2 & +0.001 & 0.72\\
 &  & 2000 & 5.6 & -- & 4.4 & +0.001 & 0.74\\
\multicolumn{8}{r}{\footnotesize asymptotic $\sigma=0.76$}\\
\addlinespace
A & 0.6 & 250 & 4.4 & 100.0 & 4.3 & 0.000 & 0.63\\
 &  & 500 & 5.1 & 100.0 & 5.0 & 0.000 & 0.64\\
 &  & 1000 & 4.8 & -- & 4.7 & -0.001 & 0.64\\
 &  & 2000 & 4.4 & -- & 5.0 & -0.001 & 0.65\\
\multicolumn{8}{r}{\footnotesize asymptotic $\sigma=0.63$}\\
\addlinespace
B & 0.3 & 250 & 20.8 & 100.0 & 6.4 & +0.146 & 4.21\\
 &  & 500 & 5.7 & 100.0 & 5.4 & +0.019 & 2.14\\
 &  & 1000 & 3.7 & 100.0 & 5.3 & +0.001 & 0.76\\
 &  & 2000 & 5.4 & -- & 4.0 & +0.001 & 0.82\\
\multicolumn{8}{r}{\footnotesize asymptotic $\sigma=0.77$}\\
\addlinespace
B & 0.6 & 250 & 4.6 & 100.0 & 6.2 & -0.003 & 0.66\\
 &  & 500 & 5.3 & 100.0 & 6.0 & -0.001 & 0.61\\
 &  & 1000 & 4.1 & -- & 5.3 & -0.001 & 0.65\\
 &  & 2000 & 3.6 & -- & 5.8 & 0.000 & 0.62\\
\multicolumn{8}{r}{\footnotesize asymptotic $\sigma=0.64$}\\
\bottomrule\end{tabular}\end{table}

%% file: tab_power.tex
\begin{table}[t]\centering\small
\caption{Empirical power (in \%) at level $5\%$ of the IBM test and, in parentheses, of the oracle BKR test computed on the unobservable subsample from $F$.}\label{tab:power}
\begin{tabular}{llrcccc}\toprule
Scenario & $p^\ast$ & $n$ & Gauss $\rho=0.3$ & Gauss $\rho=0.5$ & Gauss $\rho=-0.3$ & Clayton $\theta=1$\\\midrule
A & 0.3 & 250 & 8.5 (72.8) & 14.7 (99.0) & 8.1 (56.0) & 17.1 (99.2)\\
 &  & 500 & 10.8 (93.5) & 24.2 (100.0) & 7.8 (89.9) & 23.8 (100.0)\\
 &  & 1000 & 18.7 (99.8) & 45.2 (100.0) & 14.3 (99.6) & 44.7 (100.0)\\
 &  & 2000 & 32.2 (100.0) & 85.2 (100.0) & 37.4 (100.0) & 81.6 (100.0)\\
\addlinespace
A & 0.6 & 250 & 45.6 (93.4) & 91.4 (100.0) & 56.8 (87.7) & 81.0 (100.0)\\
 &  & 500 & 81.1 (99.9) & 99.8 (100.0) & 86.7 (99.9) & 99.1 (100.0)\\
 &  & 1000 & 98.7 (100.0) & 100.0 (100.0) & 99.9 (100.0) & 100.0 (100.0)\\
 &  & 2000 & 100.0 (100.0) & 100.0 (100.0) & 100.0 (100.0) & 100.0 (100.0)\\
\addlinespace
B & 0.3 & 250 & 27.4 (72.1) & 34.9 (98.3) & 28.1 (53.0) & 35.8 (98.8)\\
 &  & 500 & 11.7 (93.8) & 27.6 (100.0) & 9.8 (86.1) & 26.8 (100.0)\\
 &  & 1000 & 20.7 (99.9) & 52.1 (100.0) & 17.8 (99.7) & 48.5 (100.0)\\
 &  & 2000 & 36.2 (100.0) & 93.8 (100.0) & 43.2 (100.0) & 85.8 (100.0)\\
\addlinespace
B & 0.6 & 250 & 48.2 (94.8) & 92.9 (100.0) & 57.8 (87.1) & 80.2 (100.0)\\
 &  & 500 & 81.9 (99.8) & 99.7 (100.0) & 89.5 (99.9) & 99.3 (100.0)\\
 &  & 1000 & 99.2 (100.0) & 100.0 (100.0) & 99.9 (100.0) & 100.0 (100.0)\\
 &  & 2000 & 100.0 (100.0) & 100.0 (100.0) & 100.0 (100.0) & 100.0 (100.0)\\
\bottomrule\end{tabular}\end{table}

%% file: realdata_section.tex
\section{Real data applications}\label{sec:real}

We present three applications. In the first two (Sections~\ref{sec:real-lymphoma}
and~\ref{sec:real-leukemia}), built from gene expression data, the known source is
exact and the unknown component turns out to be dependent. In the third
(Section~\ref{sec:real-spam}), the unknown component is a stream of spam
emails contaminating legitimate mail, the known source is given by a sample of
verified legitimate emails, and the conclusion is the opposite.

\paragraph{Bivariate contamination with an explicitly known source.} In
two-class gene expression studies, each gene $g$ is either \emph{null}, i.e.\
its expression has the same distribution in both classes, or differentially
expressed. Bivariate contamination data arise naturally by computing, for
each gene, two test statistics on two \emph{disjoint} subsets of patients. Let
the patients of each class be split at random into two halves, and let $S_g^{(1)}$
and $S_g^{(2)}$ be rank statistics computed on halves~$1$ and~$2$. For a null gene,
and patients independent within each class with a continuous expression
distribution, the ranks of the second class in each half form a uniformly random
subset, whatever this distribution. The null law $\pi_h$ of $S_g^{(h)}$ is
therefore known exactly, and computable by enumeration. We use the randomised
version of its cdf \citep{HabigerPena2011},
\begin{equation}\label{eq:randomised}
  U_g^{(h)}=\Pi_h\bigl(S_g^{(h)}-1\bigr)+V_g^{(h)}\,\pi_h\bigl(S_g^{(h)}\bigr),\qquad h=1,2,
\end{equation}
where $\Pi_h$ is the null cdf and the $V_g^{(h)}$ are independent
$\mathrm{Uniform}(0,1)$ variables. For a null gene, $U_g^{(h)}$ is exactly
$\mathrm{Uniform}(0,1)$, and $U_g^{(1)}$ and $U_g^{(2)}$ are independent since they
are computed from disjoint sets of patients. The pairs
$\bm X_g=(U_g^{(1)},U_g^{(2)})$, $g=1,\dots,n$, thus follow the contamination
model~\eqref{eq:model} with weight $p^\ast$ equal to the proportion of
differentially expressed genes, unknown component $F$ the law of the pairs of
differentially expressed genes, and known source
\[
  F_0(u_1,u_2)=u_1u_2,\qquad (u_1,u_2)\in[0,1]^2,
\]
which holds exactly and without any distributional assumption on the
expression levels. Here $K_0=0$, and Assumption~\ref{A:ident} holds as soon as
both marginals of $F$ differ from the uniform distribution. The question
addressed by the IBM test is whether the two coordinates are independent
\emph{among the differentially expressed genes}, which the naive BKR test
applied to all genes cannot answer: by Section~\ref{sec:intro}, the mixture
of the two product laws $F_0$ and $F_1\otimes F_2$ is itself dependent.
With $G=\Ln$ the statistic is invariant under increasing transformations of
the coordinates; the figures display $Z_g^{(h)}=\Phi_{\mathcal N}^{-1}(U_g^{(h)})$,
which is standard normal for null genes.

Two caveats are standard in large-scale inference \citep{Efron2010}. The genes
are not independent, and the theoretical null presumes that no hidden
factor affects the two classes differently; the dependence between genes
affects the variability of the statistics but not the marginal law of the
pairs. In each application we report results over $20$ random splits of the
patients, with $N=2000$ bootstrap replications~\eqref{eq:pvalue} and $\Theta=[0.01,1]$. To check
that the test is reliable at these sample sizes and contamination levels, we also
apply it to \emph{pseudo-null} data: $n$ pairs drawn from the model with
$F_0(u_1,u_2)=u_1u_2$, the estimated weight $\widetilde p$ (the larger of the
two marginal estimates of \citet{Storey2002}, with $\lambda=0.5$), and an
unknown component $\widetilde F_1\otimes\widetilde F_2$ with independent coordinates
and marginals $\widetilde F_h$ obtained by inverting the empirical marginals at
$\widetilde p$, as in~\eqref{eq:family2}. These data satisfy $H_0$ and mimic the
real ones; $200$ such data sets are generated in each case. Table~\ref{tab:real}
summarises the results, and Figure~\ref{fig:real} displays the pairs.

\begin{figure}[t]
\centering
\includegraphics[width=0.92\textwidth]{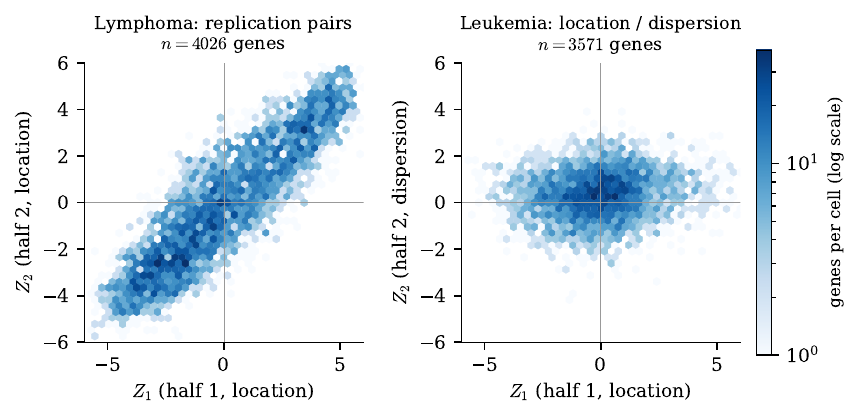}
\caption{Pairs $(Z_g^{(1)},Z_g^{(2)})=(\Phi_{\mathcal N}^{-1}(U_g^{(1)}),\Phi_{\mathcal N}^{-1}(U_g^{(2)}))$
for one random split of the patients. Null genes follow the standard
bivariate normal law. Left: lymphoma data, rank-sum statistics in the two halves.
Right: leukemia data, rank-sum (location) statistic in half~$1$ and
Ansari--Bradley (dispersion) statistic in half~$2$.}
\label{fig:real}
\end{figure}

\begin{table}[t]
\centering\small
\caption{Real data applications. $\widetilde p_1,\widetilde p_2$: marginal estimates of
the weight \citep{Storey2002}; $\mathsf T_n$: median over $20$ random splits (range in
parentheses); $\widehat q_{0.95}$: bootstrap critical value for the first split;
$p$-values are bootstrap $p$-values, $1/2001$ being the smallest attainable
value; last two columns: rejection rates at level $5\%$ over $200$ pseudo-null
data sets satisfying $H_0$.}\label{tab:real}
\setlength{\tabcolsep}{4pt}
\begin{tabular}{lcccccccc}
\toprule
 & & & & & & & \multicolumn{2}{c}{pseudo-null}\\
\cmidrule(lr){8-9}
 & $n$ & $(\widetilde p_1,\widetilde p_2)$ & $\mathsf T_n$ & $\widehat q_{0.95}$ & $p$-value & $\ph$ & IBM & naive\\
\midrule
Lymphoma & 4026 & $(0.67,0.69)$ & 56.6 (45.3--62.5) & 0.052 & $<0.001$ & $1$ & $6.5\%$ & $100\%$\\
Leukemia & 3571 & $(0.37,0.05)$ & 0.62 (0.21--1.17) & 0.035 & $<0.001$ & $1$ & $4.0\%$ & $100\%$\\
\bottomrule
\end{tabular}
\end{table}

\paragraph{Why $\ph=1$ signals dependence.} In both applications the
estimated weight equals $1$ in every split, far above the marginal estimates
$\widetilde p_h$, which are lower bounds for $p^\ast$. On the pseudo-null data,
in contrast, the median of $\ph$ is $0.69$ and $0.38$, respectively, close to
$\widetilde p$, as Corollary~\ref{thm:null} predicts. The boundary value is a
signature of a dependent unknown component when $K_0=0$. By the dependence
form~\eqref{eq:structure}, $D(\cdot\,;L,p)=sK+s(1-s)\Delta$ with
$s=p^\ast/p$, so that
\[
  d(p)=s^2\,\nG{K+(1-s)\Delta}^2 .
\]
Under $H_0$ ($K=0$) the contrast vanishes at $s=1$, i.e.\ at $p=p^\ast$. When
$K\ne0$, the factor $s^2$ favours the smallest value $s=p^\ast$, i.e.\ $p=1$, as
soon as $\nG K>p^\ast\nG{K+(1-p^\ast)\Delta}$. In that case $\mathsf T_n$ coincides
with the BKR statistic of the observed pairs, but its calibration does not:
the naive test rejects $H_0$ on every pseudo-null data set, whereas the IBM
test keeps a level close to $5\%$. A rejection by the IBM test is therefore
evidence of dependence within the unknown component, not an artefact of the
contamination.

\subsection{Lymphoma: replicability of differential expression}\label{sec:real-lymphoma}

The lymphoma data of \citet{Alizadeh2000}, in the preprocessed version
of \citet{ChunKeles2010}, contain the expression of $n=4026$ genes for $42$
patients with diffuse large B-cell lymphoma (DLBCL) and $20$ patients with
follicular lymphoma or chronic lymphocytic leukemia (FL/CLL). In each class
the patients are split into two halves ($21+21$ DLBCL, $10+10$ FL/CLL), and
$S_g^{(h)}$ is the Mann--Whitney statistic \citep{MannWhitney1947} comparing
DLBCL with FL/CLL in half~$h$, with exact null pmf
$\pi_h(s)=c_{10,21}(s)/\binom{31}{10}$, where $c_{m,k}(s)$ is the number of
arrangements of $m+k$ ranks with $s$ pairs in which a FL/CLL value is below a
DLBCL value. Large values of $U_g^{(h)}$ indicate overexpression in DLBCL.
About two thirds of the genes are differentially expressed
($\widetilde p_1=0.67$, $\widetilde p_2=0.69$).

The question is whether, among differentially expressed genes, the evidence
obtained on two independent groups of patients is independent. Independence
would mean that the size of the effect of a gene in one half carries no
information on its size in the other, i.e.\ that differential expression is
not reproducible beyond its mere occurrence. Figure~\ref{fig:real} (left)
suggests the opposite, and the IBM test rejects $H_0$ in all $20$ splits,
with $\mathsf T_n$ between $45$ and $63$, against a bootstrap critical value of
about $0.05$. On pseudo-null data with the same marginals, the IBM test
rejects in $6.5\%$ of the cases, so the rejection is not an artefact of the
contamination: the effects of differentially expressed genes are strongly
reproducible across groups of patients, which is the expected behaviour of
genes whose effect sizes vary from gene to gene.

\subsection{Leukemia: location versus dispersion effects}\label{sec:real-leukemia}

The leukemia data of \citet{Golub1999}, in the version of \citet{EfronHastie2016},
contain the expression of $n=3571$ genes for $47$ patients with acute
lymphoblastic leukemia (ALL) and $25$ with acute myeloid leukemia (AML). The
patients are split into halves ($23+24$ ALL, $12+13$ AML). In half~$1$,
$S_g^{(1)}$ is the Mann--Whitney statistic, which detects a shift in
\emph{location} between AML and ALL. In half~$2$, $S_g^{(2)}$ is the
Ansari--Bradley statistic \citep{AnsariBradley1960}
$\sum_{i\in\mathrm{AML}}a(R_i)$ with scores $a(r)=\min(r,N+1-r)$, where the $R_i$
are the ranks in the pooled second half ($24$ ALL and $13$ AML, $N=37$), which
detects a difference in \emph{dispersion}. Its exact null pmf is that of the
sum of a uniformly random subset of $13$ scores among $a(1),\dots,a(37)$,
computed by recursion. It is oriented so that large values of $U_g^{(2)}$ indicate a larger
dispersion in AML. About $37\%$ of the genes show a location effect
($\widetilde p_1=0.37$) and $5\%$ a dispersion effect ($\widetilde p_2=0.05$), so that the
second marginal of $F$ is close to uniform; the weight is identified through the first
coordinate.

The question is whether, among genes that behave differently in the two
leukemias, the location effect and the dispersion effect are independent.
The overall Spearman correlation between the two coordinates is small
($0.12$ for the first split; range $0.05$--$0.17$), and
Figure~\ref{fig:real} (right) shows no obvious structure. The IBM test
nevertheless rejects $H_0$ in all $20$ splits, with $\mathsf T_n$ between $0.21$
and $1.17$, against a bootstrap critical value of $0.035$; on pseudo-null data
its rejection rate is $4.0\%$. Among differentially expressed genes, the
location effect and the dispersion effect are therefore dependent. A simple
description of this dependence is asymmetric: for the first split, the
$421$ genes with $Z_g^{(1)}>2$ (overexpressed in AML) have an average
$Z_g^{(2)}$ of $0.60$, against $0.23$ for the $541$ genes with $Z_g^{(1)}<-2$ and
$0.32$ for genes with $|Z_g^{(1)}|<0.5$. Genes overexpressed in AML thus also tend
to be more variable across AML patients, a mean--variance relation
that a correlation coefficient computed on all genes barely reveals.

\paragraph{A cautionary example.} The prostate cancer data of
\citet{Singh2002} ($n=6033$ genes, $50$ healthy and $52$ cancer patients),
analysed with the replication design, illustrate the weak-separation effect
discussed in Section~\ref{sec:simu}. There only about $8\%$ of the genes are
differentially expressed ($\widetilde p_1=0.08$, $\widetilde p_2=0.08$). Since $K_0=0$,
the separation between the zero of $d$ at $p^\ast$ and its value at $p=1$ is of
order $p^{\ast2}(1-p^\ast)^2\nG\Delta^2$, whereas the sampling noise of
$\delta_n$ near $t^\ast=1/p^\ast$ is inflated by the factor $t^{\ast2}\approx140$. On
pseudo-null data with these characteristics, $\ph=1$ in every replication and
the IBM test rejects $H_0$ in all of them, so the rejection observed on the
real data ($\mathsf T_n$ between $0.80$ and $2.43$) cannot be interpreted. The test
should not be used when the estimated weight is small and $F_0$ is a product
distribution, unless the number of observations is much larger; a pseudo-null
experiment as above is a simple way to check this in practice.

\subsection{Spam filtering: contamination of legitimate email}\label{sec:real-spam}

Unsolicited and fraudulent email accounts for a large share of the traffic
received by mail servers. A server, or a user, typically has an archive of
\emph{verified} legitimate messages, which determines the law of normal
activity, whereas incoming messages are unlabelled: whether a message is spam
is learnt later, and only partially, through user reports. This is a
contamination model in which the known source is the law of legitimate mail
and the unknown component is the law of spam, with a large weight $p^\ast$. We
use the Spambase data \citep{Spambase}, distributed with the R package
\texttt{kernlab} \citep{kernlab}: $4601$ emails received at Hewlett-Packard Labs,
of which $1813$ ($39.4\%$) are spam, each described by $57$ features
(frequencies of words and characters, and statistics of runs of capital
letters). The labels are used only to validate the conclusions.

\paragraph{Construction.} A random sample of $m=1788$ legitimate emails plays
the role of the verified archive (the training data), and the known source
$F_0$ is taken to be its bivariate empirical cdf: conditionally on the archive,
$F_0$ is a fixed, known function, as the method requires. The
incoming batch consists of the $1000$ other legitimate emails and $667$ spam
emails, which keeps the spam proportion of the collection: $n=1667$ and
$p^\ast=0.40$. The legitimate emails of the batch, however, follow the law of
legitimate mail, which the archive only estimates: the batch satisfies the
contamination model with this law, not with its empirical version, and the
difference, of order $m^{-1/2}$, enters the statistic multiplied by $\sqrt n$. It is
negligible when $m/n$ is large, but here $m/n\approx1$, so the
multiplier bootstrap includes the error of $F_0$. Perturbing $F_0$ by $w_0$ moves the
inverted cdf $F(\cdot\,;1/t)=tL+(1-t)F_0$ by $(1-t)w_0$, hence moves $D(t)$ by
$(1-t)\bigl(w_0-w_0^{(1)}\otimes F_2(\cdot\,;1/t)-F_1(\cdot\,;1/t)\otimes w_0^{(2)}\bigr)$.
The bootstrap noise~\eqref{eq:bootstrap} (null form) therefore receives the
additional, independent term $\sqrt{n/m}\,(1-\that)(I-\widehat\Pi)\,\xi_m^{\zeta'}$, where
$\xi_m^{\zeta'}$ is the multiplier process computed on the archive with independent
multipliers $\zeta'$ and the same plug-in marginals $\widehat F_1,\widehat F_2$. This term
vanishes when $m/n\to\infty$, in which case treating the empirical cdf of the
archive as the exact source is legitimate. Calibration is
checked with pseudo-null data, as in Sections~\ref{sec:real-lymphoma}--\ref{sec:real-leukemia}, but now directly: in pseudo-null batches, one coordinate of
the spam emails is randomly permuted, so that the spam component has
independent coordinates with the observed marginals and $H_0$ holds exactly,
everything else being unchanged.

\paragraph{Repeated constructions.} One \emph{construction} consists of (i)
splitting the $2788$ legitimate emails at random into the $1000$ emails of the
batch and the $1788$ emails of the archive, (ii) drawing $667$ of the $1813$
spam emails at random without replacement, and (iii) computing, on the
resulting batch, the IBM test, the naive BKR test, and the oracle BKR test on
the $667$ spam emails. A pseudo-null construction is the same, except that one
coordinate of the $667$ spam emails is randomly permuted among them before the
batch is formed. All rates below are proportions of rejections at level $5\%$
over $100$ independent constructions, and medians are taken over the same
$100$ constructions. Bootstrap $p$-values use $N=400$ replications in each
construction and $N=4000$ for the single analysis of Figure~\ref{fig:spam}. Since all constructions reuse the same $4601$ emails, the
rates obtained on the real data are not estimates of a level or of a power: they
measure how stable the conclusions are with respect to the choice of the archive
and of the batch. On pseudo-null data, in contrast, each permutation produces a
new data set satisfying $H_0$, and the rejection rate estimates the level of
the test for these marginals. The batch of Figure~\ref{fig:spam} is one further
construction, used to report a single analysis.

\begin{figure}[t]
\centering
\includegraphics[width=\textwidth]{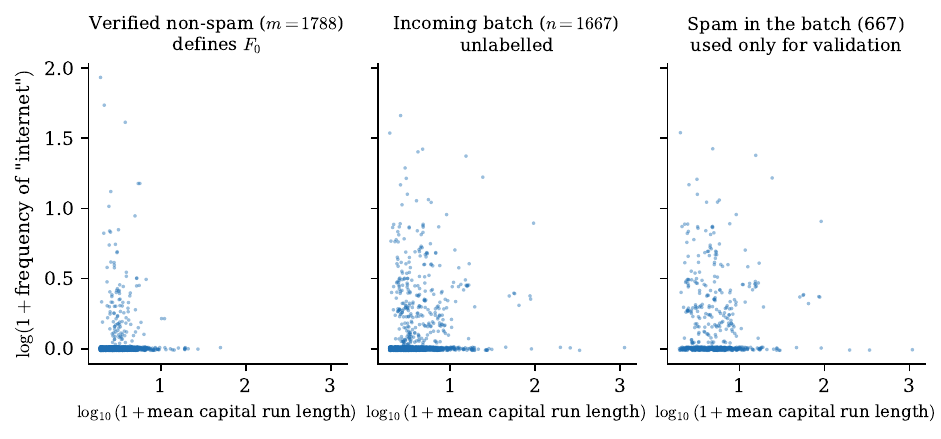}
\caption{Spambase data, frequency of the word ``internet'' against the mean
length of runs of capital letters (log scales, small jitter). Left: archive of
verified legitimate emails defining $F_0$. Centre: unlabelled incoming batch
($1000$ legitimate and $667$ spam emails). Right: the spam emails of the batch,
which are not identifiable in practice.}
\label{fig:spam}
\end{figure}

\paragraph{Results.} We consider two pairs of features. The first pair is the
frequency of the word ``internet'' and the mean length of runs of capital
letters (Figure~\ref{fig:spam}). Both are typical of spam, but there is no
obvious reason for them to be linked within spam, whereas they are dependent
among legitimate emails (BKR test on the archive: $p$-value $<0.001$), so that
$K_0\ne0$. For the batch of Figure~\ref{fig:spam}, $\mathsf T_n=0.057$, the
bootstrap critical value is $\widehat q_{0.95}=0.124$ and the $p$-value is $0.43$:
independence among spam emails is not rejected. The weight estimate is
$\ph=0.46$. Applied directly to the unlabelled batch, the classical BKR test
(the naive test of Section~\ref{sec:simu}) finds strong dependence ($p$-value
$<0.001$), but this reflects the mixture of legitimate and spam emails. Applied to
the $667$ spam emails alone (the oracle test), which is possible here only because
the labels are known, the same test finds no evidence of dependence ($p$-value
$0.60$), in agreement with the IBM test. Over the $100$ constructions, the IBM test rejects in
$8\%$ of them and the median of $\ph$ is $0.40$: the conclusion does not depend
on the particular split of the data. On pseudo-null batches it rejects in $9\%$
of the cases, which estimates its level, against $100\%$ for the naive test in
both situations. The behaviour on the real batches is thus that expected under
$H_0$, in agreement with the oracle test, which rejects in $1\%$ of the
constructions.
The weight is identified only through the bivariate structure. Marginal
estimates, obtained with the estimator of \citet{PatraSen2016} applied to each
coordinate with the corresponding marginal of the archive, are $0.11$ for the
frequency of ``internet'' and $0.35$ for the capital run length (medians over
$30$ constructions). They estimate only the part of the spam that differs from
legitimate mail in each coordinate: for instance, $66\%$ of the spam emails do
not contain the word ``internet'', against $93\%$ of the legitimate ones. Under
$H_0$, the independence of the coordinates of the unknown component identifies
the weight (Assumption~\ref{A:ident}), and $\ph$ recovers $p^\ast=0.40$.
The correction for the error of the archive matters here: without it, i.e.\
treating the empirical cdf of the archive as the exact source, the IBM test
rejects in $18\%$ of the pseudo-null batches and in $20\%$ of the real ones.

The second pair is the frequency of the character ``!'' and that of the
character ``\textdollar'', which are clearly dependent within spam (the oracle test
rejects in all $100$ constructions). For the batch of Figure~\ref{fig:spam},
the IBM test rejects ($\mathsf T_n=0.255$, $\widehat q_{0.95}=0.234$, $p$-value $0.038$). Over the $100$
constructions, however, it rejects in only $18\%$ of them, so that this
rejection is not stable with respect to the split, while its level
on pseudo-null batches is $7\%$ (without the correction for the error of the
archive, $48\%$ and $11\%$, respectively). At this sample size the test is therefore
reliable but not powerful against this alternative. The dependence is weak
(Spearman correlation $0.11$ among spam emails); it is easily detected from $667$
spam emails observed alone, but it is diluted by the $1000$ legitimate emails of
the batch, by the error of the archive, and by the estimation of the weight, part
of which it absorbs: $\ph$ has median $0.36$ (the marginal estimates are $0.31$
and $0.23$).

These results complement the gene expression applications. When the
contamination weight is large, the test keeps a level close to the nominal
one even with a small reference sample, and it does not report dependence
where the unknown component has none, unlike tests applied to the observed
sample.

%% file: main.bbl
\begin{thebibliography}{39}
\providecommand{\natexlab}[1]{#1}
\providecommand{\url}[1]{\texttt{#1}}
\expandafter\ifx\csname urlstyle\endcsname\relax
  \providecommand{\doi}[1]{doi: #1}\else
  \providecommand{\doi}{doi: \begingroup \urlstyle{rm}\Url}\fi

\bibitem[Alexander(1984)]{Alexander1984}
K.~S. Alexander.
\newblock Probability inequalities for empirical processes and a law of the
  iterated logarithm.
\newblock \emph{Ann. Probab.}, 12\penalty0 (4):\penalty0 1041--1067, 1984.

\bibitem[Alizadeh et~al.(2000)Alizadeh, Eisen, Davis, et~al.]{Alizadeh2000}
A.~A. Alizadeh, M.~B. Eisen, R.~E. Davis, et~al.
\newblock Distinct types of diffuse large {B}-cell lymphoma identified by gene
  expression profiling.
\newblock \emph{Nature}, 403:\penalty0 503--511, 2000.

\bibitem[Ansari and Bradley(1960)]{AnsariBradley1960}
A.~R. Ansari and R.~A. Bradley.
\newblock Rank-sum tests for dispersions.
\newblock \emph{Ann. Math. Statist.}, 31:\penalty0 1174--1189, 1960.

\bibitem[Blum et~al.(1961)Blum, Kiefer, and Rosenblatt]{BKR1961}
J.~R. Blum, J.~Kiefer, and M.~Rosenblatt.
\newblock Distribution free tests of independence based on the sample
  distribution function.
\newblock \emph{Ann. Math. Statist.}, 32:\penalty0 485--498, 1961.

\bibitem[Bordes et~al.(2006{\natexlab{a}})Bordes, Delmas, and
  Vandekerkhove]{BDV2006}
L.~Bordes, C.~Delmas, and P.~Vandekerkhove.
\newblock Semiparametric estimation of a two-component mixture model where one
  component is known.
\newblock \emph{Scand. J. Stat.}, 33\penalty0 (4):\penalty0 733--752,
  2006{\natexlab{a}}.

\bibitem[Bordes et~al.(2006{\natexlab{b}})Bordes, Mottelet, and
  Vandekerkhove]{BMV2006}
L.~Bordes, S.~Mottelet, and P.~Vandekerkhove.
\newblock Semiparametric estimation of a two-component mixture model.
\newblock \emph{Ann. Statist.}, 34\penalty0 (3):\penalty0 1204--1232,
  2006{\natexlab{b}}.

\bibitem[B\"ucher and Dette(2010)]{BucherDette2010}
A.~B\"ucher and H.~Dette.
\newblock A note on bootstrap approximations for the empirical copula process.
\newblock \emph{Statist. Probab. Lett.}, 80\penalty0 (23--24):\penalty0
  1925--1932, 2010.

\bibitem[Chun and Kele\c{s}(2010)]{ChunKeles2010}
H.~Chun and S.~Kele\c{s}.
\newblock Sparse partial least squares regression for simultaneous dimension
  reduction and variable selection.
\newblock \emph{J. R. Stat. Soc. Ser. B. Stat. Methodol.}, 72\penalty0
  (1):\penalty0 3--25, 2010.

\bibitem[Deheuvels(1981)]{Deheuvels1981}
P.~Deheuvels.
\newblock An asymptotic decomposition for multivariate distribution-free tests
  of independence.
\newblock \emph{J. Multivariate Anal.}, 11:\penalty0 102--113, 1981.

\bibitem[Dudley(1999)]{Dudley1999}
R.~M. Dudley.
\newblock \emph{Uniform Central Limit Theorems}.
\newblock Cambridge University Press, Cambridge, 1999.

\bibitem[Durbin(1973)]{Durbin1973}
J.~Durbin.
\newblock Weak convergence of the sample distribution function when parameters
  are estimated.
\newblock \emph{Ann. Statist.}, 1:\penalty0 279--290, 1973.

\bibitem[Dynkin and Mandelbaum(1983)]{DynkinMandelbaum1983}
E.~B. Dynkin and A.~Mandelbaum.
\newblock Symmetric statistics, {P}oisson point processes, and multiple
  {W}iener integrals.
\newblock \emph{The Annals of Statistics}, 11:\penalty0 739--745, 1983.

\bibitem[Efron(2010)]{Efron2010}
B.~Efron.
\newblock \emph{Large-Scale Inference: Empirical {B}ayes Methods for
  Estimation, Testing, and Prediction}.
\newblock Cambridge University Press, Cambridge, 2010.

\bibitem[Efron and Hastie(2016)]{EfronHastie2016}
B.~Efron and T.~Hastie.
\newblock \emph{Computer Age Statistical Inference}.
\newblock Cambridge University Press, Cambridge, 2016.

\bibitem[Efron et~al.(2001)Efron, Tibshirani, Storey, and Tusher]{Efron2001}
B.~Efron, R.~Tibshirani, J.~D. Storey, and V.~Tusher.
\newblock Empirical {B}ayes analysis of a microarray experiment.
\newblock \emph{J. Amer. Statist. Assoc.}, 96:\penalty0 1151--1160, 2001.

\bibitem[Genest and R\'emillard(2004)]{GenestRemillard2004}
C.~Genest and B.~R\'emillard.
\newblock Tests of independence and randomness based on the empirical copula
  process.
\newblock \emph{Test}, 13:\penalty0 335--369, 2004.

\bibitem[Gin\'e and Zinn(1990)]{GineZinn1990}
E.~Gin\'e and J.~Zinn.
\newblock Bootstrapping general empirical measures.
\newblock \emph{Ann. Probab.}, 18\penalty0 (2):\penalty0 851--869, 1990.

\bibitem[Golub et~al.(1999)Golub, Slonim, Tamayo, et~al.]{Golub1999}
T.~R. Golub, D.~K. Slonim, P.~Tamayo, et~al.
\newblock Molecular classification of cancer: class discovery and class
  prediction by gene expression monitoring.
\newblock \emph{Science}, 286:\penalty0 531--537, 1999.

\bibitem[Habiger and Pe\~na(2011)]{HabigerPena2011}
J.~D. Habiger and E.~A. Pe\~na.
\newblock Randomised {P}-values and nonparametric procedures in multiple
  testing.
\newblock \emph{J. Nonparametr. Stat.}, 23\penalty0 (3):\penalty0 583--604,
  2011.

\bibitem[Hoeffding(1948)]{Hoeffding1948}
W.~Hoeffding.
\newblock A non-parametric test of independence.
\newblock \emph{Ann. Math. Statist.}, 19:\penalty0 546--557, 1948.

\bibitem[Hopkins et~al.(1999)Hopkins, Reeber, Forman, and Suermondt]{Spambase}
M.~Hopkins, E.~Reeber, G.~Forman, and J.~Suermondt.
\newblock Spambase.
\newblock UCI Machine Learning Repository, 1999.

\bibitem[Imhof(1961)]{Imhof1961}
J.~P. Imhof.
\newblock Computing the distribution of quadratic forms in normal variables.
\newblock \emph{Biometrika}, 48:\penalty0 419--426, 1961.

\bibitem[Karatzoglou et~al.(2004)Karatzoglou, Smola, Hornik, and
  Zeileis]{kernlab}
A.~Karatzoglou, A.~Smola, K.~Hornik, and A.~Zeileis.
\newblock kernlab -- an {S4} package for kernel methods in {R}.
\newblock \emph{J. Stat. Softw.}, 11\penalty0 (9):\penalty0 1--20, 2004.

\bibitem[Kiefer(1961)]{Kiefer1961}
J.~Kiefer.
\newblock On large deviations of the empiric d.f. of vector chance variables
  and a law of the iterated logarithm.
\newblock \emph{Pacific J. Math.}, 11:\penalty0 649--660, 1961.

\bibitem[Kojadinovic and Holmes(2009)]{KojadinovicHolmes2009}
I.~Kojadinovic and M.~Holmes.
\newblock Tests of independence among continuous random vectors based on
  {C}ram\'er--von {M}ises functionals of the empirical copula process.
\newblock \emph{J. Multivariate Anal.}, 100:\penalty0 1137--1154, 2009.

\bibitem[Kojadinovic et~al.(2011)Kojadinovic, Yan, and Holmes]{KYH2011}
I.~Kojadinovic, J.~Yan, and M.~Holmes.
\newblock Fast large-sample goodness-of-fit tests for copulas.
\newblock \emph{Statist. Sinica}, 21:\penalty0 841--871, 2011.

\bibitem[Kosorok(2008)]{Kosorok2008}
M.~R. Kosorok.
\newblock \emph{Introduction to Empirical Processes and Semiparametric
  Inference}.
\newblock Springer, New York, 2008.

\bibitem[Mann and Whitney(1947)]{MannWhitney1947}
H.~B. Mann and D.~R. Whitney.
\newblock On a test of whether one of two random variables is stochastically
  larger than the other.
\newblock \emph{Ann. Math. Statist.}, 18:\penalty0 50--60, 1947.

\bibitem[Milhaud et~al.(2022)Milhaud, Pommeret, Salhi, and
  Vandekerkhove]{MPSV2022}
X.~Milhaud, D.~Pommeret, Y.~Salhi, and P.~Vandekerkhove.
\newblock Semiparametric two-sample admixture components comparison test: the
  symmetric case.
\newblock \emph{J. Statist. Plann. Inference}, 2022.

\bibitem[Milhaud et~al.(2024{\natexlab{a}})Milhaud, Pommeret, Salhi, and
  Vandekerkhove]{MPSV2024}
X.~Milhaud, D.~Pommeret, Y.~Salhi, and P.~Vandekerkhove.
\newblock Two-sample contamination model test.
\newblock \emph{Bernoulli}, 30\penalty0 (1), 2024{\natexlab{a}}.
\newblock \doi{10.3150/23-BEJ1593}.

\bibitem[Milhaud et~al.(2024{\natexlab{b}})Milhaud, Pommeret, Salhi, and
  Vandekerkhove]{MPSV2024b}
X.~Milhaud, D.~Pommeret, Y.~Salhi, and P.~Vandekerkhove.
\newblock Contamination-source based {K}-sample clustering.
\newblock \emph{J. Mach. Learn. Res.}, 25, 2024{\natexlab{b}}.

\bibitem[Patra and Sen(2016)]{PatraSen2016}
R.~K. Patra and B.~Sen.
\newblock Estimation of a two-component mixture model with applications to
  multiple testing.
\newblock \emph{J. R. Stat. Soc. Ser. B. Stat. Methodol.}, 78\penalty0
  (4):\penalty0 869--893, 2016.

\bibitem[Pommeret and Vandekerkhove(2019)]{PommeretVandekerkhove2019}
D.~Pommeret and P.~Vandekerkhove.
\newblock Semiparametric density testing in the contamination model.
\newblock \emph{Electron. J. Stat.}, 13, 2019.

\bibitem[R\'emillard and Scaillet(2009)]{RemillardScaillet2009}
B.~R\'emillard and O.~Scaillet.
\newblock Testing for equality between two copulas.
\newblock \emph{J. Multivariate Anal.}, 100:\penalty0 377--386, 2009.

\bibitem[Shorack and Wellner(1986)]{ShorackWellner1986}
G.~R. Shorack and J.~A. Wellner.
\newblock \emph{Empirical Processes with Applications to Statistics}.
\newblock Wiley, New York, 1986.

\bibitem[Singh et~al.(2002)Singh, Febbo, Ross, et~al.]{Singh2002}
D.~Singh, P.~G. Febbo, K.~Ross, et~al.
\newblock Gene expression correlates of clinical prostate cancer behavior.
\newblock \emph{Cancer Cell}, 1:\penalty0 203--209, 2002.

\bibitem[Storey(2002)]{Storey2002}
J.~D. Storey.
\newblock A direct approach to false discovery rates.
\newblock \emph{J. R. Stat. Soc. Ser. B. Stat. Methodol.}, 64\penalty0
  (3):\penalty0 479--498, 2002.

\bibitem[Stute et~al.(1993)Stute, Gonz\'alez~Manteiga, and
  Presedo~Quindimil]{Stute1993}
W.~Stute, W.~Gonz\'alez~Manteiga, and M.~Presedo~Quindimil.
\newblock Bootstrap based goodness-of-fit tests.
\newblock \emph{Metrika}, 40:\penalty0 243--256, 1993.

\bibitem[van~der Vaart and Wellner(1996)]{vdVW1996}
A.~W. van~der Vaart and J.~A. Wellner.
\newblock \emph{Weak Convergence and Empirical Processes}.
\newblock Springer, New York, 1996.

\end{thebibliography}
